\documentclass[11pt, reqno]{amsart}
\usepackage{amssymb,latexsym,amsmath,amsfonts,mathdots,enumitem}
\usepackage{latexsym}
\usepackage[mathscr]{eucal}
\usepackage{colortbl,xcolor}
\usepackage{lmodern}
\usepackage{sansmathaccent}
\usepackage[latin1]{inputenc}
\usepackage{tikz}
\usepackage{physics}
\usepackage[normalem]{ulem}
\usetikzlibrary{shapes,arrows}
\usetikzlibrary{matrix,calc,shapes,arrows,positioning}
\pdfmapfile{+sansmathaccent.map}

\numberwithin{equation}{section}
\theoremstyle{plain}

\newtheorem{thm}{Theorem}[section]

\newtheorem{cor}[thm]{Corollary}
\newtheorem{lem}[thm]{Lemma}
\newtheorem{prop}[thm]{Proposition}
\theoremstyle{definition}
\newtheorem{defn}[thm]{Definition}
\newtheorem{rem}[thm]{Remark}

\numberwithin{equation}{section}

\def\a{\alpha}

\def\Re{\operatorname{Re}}

\def\beq{\begin{eqnarray}}
\def\eeq{\end{eqnarray}}
\def\beqa{\begin{eqnarray*}}
\def\eeqa{\end{eqnarray*}}

\def\beqn{\begin{equation}}
\def\eeqn{\end{equation}}

\def\mg#1{}

\renewcommand{\epsilon}{\varepsilon}
\renewcommand{\phi}{\varphi}

\def\red{\color{red}}

\renewcommand{\bf}[1]{\textbf{#1}}
\renewcommand{\it}[1]{\textit{#1}}

\renewcommand{\sf}[1]{\textsf{#1}}

\renewcommand{\Re}[1]{\sf{Re}(#1)}

\numberwithin{equation}{section}
\allowdisplaybreaks[4] 

\setlist[enumerate]{font=\upshape,noitemsep, topsep=0pt} 
\setlist[itemize]{noitemsep, topsep=0pt}

\begin{document}
	
	\title{Operators on Hilbert Space having $\Gamma_{E(3; 3; 1, 1, 1)}$ and $\Gamma_{E(3; 2; 1, 2)}$ as Spectral Sets}
	\author{Dinesh Kumar Keshari ~ \hspace{0.1cm} Avijit Pal ~ \hspace{0.1cm} Bhaskar Paul}
	\address[ D. K. Keshari]{School of Mathematical Sciences, National Institute of Science Education and Research Bhubaneswar, An OCC of Homi Bhabha National Institute, Jatni, Khurda,  Odisha-752050, India}
\email{dinesh@niser.ac.in}

\address[A. Pal]{Department of Mathematics, IIT Bhilai, 6th Lane Road, Jevra, Chhattisgarh 491002}
\email{A. Pal:avijit@iitbhilai.ac.in}

\address[B. Paul]{Department of Mathematics, IIT Bhilai, 6th Lane Road, Jevra, Chhattisgarh 491002}
\email{B. Paul:bhaskarpaul@iitbhilai.ac.in }

\subjclass[2010]{47A20, 47A25, 47A45.}

\keywords{$\Gamma_{E(3; 3; 1, 1, 1)} $-contraction, $\Gamma_{E(3; 2; 1, 2)} $-contraction, Spectral set, Complete spectral set, $\Gamma_{E(3; 3; 1, 1, 1)}$-unitary,  $\Gamma_{E(3; 2; 1, 2)}$-unitary, Wold Decomposition, $\Gamma_{E(3; 3; 1, 1, 1)}$-isometry, $\Gamma_{E(3; 2; 1, 2)}$-isometry}
	
	\maketitle
	
	\begin{abstract}

  A $7$-tuple of commuting bounded operators $\textbf{T} = (T_1, \dots, T_7)$ on a Hilbert space $\mathcal{H}$ is called a \textit{$\Gamma_{E(3; 3; 1, 1, 1)} $-contraction} if $\Gamma_{E(3; 3; 1, 1, 1)}$ is a spectral set for $\textbf{T}. $ Let $(S_1, S_2, S_3)$ and $(\tilde{S}_1, \tilde{S}_2)$ be tuples of commuting bounded operators defined on a Hilbert space $\mathcal{H}$ with $S_i\tilde{S}_j = \tilde{S}_jS_i$ for $1 \leqslant i \leqslant 3$ and $1 \leqslant j \leqslant 2$. We say that $\textbf{S} = (S_1, S_2, S_3, \tilde{S}_1, \tilde{S}_2)$ is a $\Gamma_{E(3; 2; 1, 2)} $-contraction if $ \Gamma_{E(3; 2; 1, 2)}$ is a spectral set for $\textbf{S}$.  We derive various properties of $\Gamma_{E(3; 3; 1, 1, 1)}$-contractions and $\Gamma_{E(3; 2; 1, 2)}$-contractions and establish a relationship between them. We discuss the fundamental equations for $\Gamma_{E(3; 3; 1, 1,1 )}$-contractions and $\Gamma_{E(3; 2; 1, 2)}$-contractions. We explore the structure of $\Gamma_{E(3; 3; 1, 1, 1)}$-unitaries and $\Gamma_{E(3; 2; 1, 2)}$-unitaries and elaborate on the relationship between them. We also study various properties of $\Gamma_{E(3; 3; 1, 1, 1)}$-isometries and $\Gamma_{E(3; 2; 1, 2)}$-isometries. We discuss the Wold Decomposition for a $\Gamma_{E(3; 3; 1, 1, 1)}$-isometry and a $\Gamma_{E(3; 2; 1, 2)}$-isometry. We further outline the structure theorem for a pure $\Gamma_{E(3; 3; 1, 1, 1)}$-isometry and a pure $\Gamma_{E(3; 2; 1, 2)}$-isometry.
  

\end{abstract}
	\section{Introduction}
	
Let $\mathbb C[z_1,\dots,z_n]$ denotes the polynomial ring in $n$ variables over the field of complex numbers. Let $\Omega$ be a compact subset of $\mathbb C^m,$ and let $\mathcal{O}(\Omega)$ denotes the algebra of holomorphic functions on an open set containing $\Omega.$ Let  $\mathbf{T}=(T_1,\ldots,T_m)$ be a commuting $m$-tuple of bounded operators defined on a  Hilbert space $\mathcal H$ and $\sigma(\mathbf T)$ denotes the joint spectrum of $\mathbf {T}.$  Consider the map $\rho_{\mathbf T}:\mathcal{O}(\Omega)\rightarrow\mathcal B(\mathcal H)$  defined by $$1\to I~{\rm{and}}~z_i\to T_i ~{\rm{for}}~1\leq i\leq m.$$ Clearly, $\rho_{\mathbf T}$ is a homomorphism. A compact set $\Omega\subset \mathbb C^m$ is a spectral set for a  $m$-tuple of commuting bounded operators $\mathbf{T}=(T_1,\ldots,T_m)$ if $\sigma(\mathbf T)\subseteq \Omega$ and the homomorphism $\rho_{\mathbf T}:\mathcal{O}(\Omega)\rightarrow\mathcal B(\mathcal H)$ is contractive. The following theorem says that the closed unit disc is a spectral set for a contraction defined on a Hilbert space $\mathcal H$ [Corollary 1.2, \cite {paulsen}].
\begin{thm} 
Suppose that $T\in \mathcal B(\mathcal H)$ is a contraction on a Hilbert space $\mathcal H.$ Then
$$\|p(T)\|\leq \|p\|_{\infty, \bar {\mathbb D}}:=\sup\{|p(z)|: |z|\leq1\}$$ for every polynomial $p.$

\end{thm}
  The following theorem is a refined version of the Sz.-Nagy dilation theorem [Theorem 1.1,  \cite {paulsen}].
\begin{thm} Let $T\in \mathcal B(\mathcal H)$ be a contraction on a Hilbert space $\mathcal H.$ Then there exists a Hilbert space $\mathcal K$ containing $\mathcal H$ as a subspace and a unitary operator $U$ acting on a Hilbert space $\mathcal K \supseteq \mathcal H$ with the property that $\mathcal K$ is the smallest closed reducing subspace for $U$ containing $\mathcal H$
such that
$$P_\mathcal H\,U^n_{|\mathcal H}=T^n,$$ for all non-negative integer $n.$
\end{thm}
The existence  of unitary dilation for a  contraction $T$ was constructed by Schaffer. The spectral theorem for a unitary operator and power dilation for a contraction guarantee the von Neumann inequality. Let $\Omega \subset \mathbb C^m$ be a compact and let $F=\left(\!(f_{ij})\!\right)$  be a matrix valued polynomial on  
 $\Omega,$ we refer to $\Omega$ as a complete spectral set (complete $\Omega$-contraction) for $\mathbf T$, if $\|F(\mathbf T) \| \leq \|F\|_{\infty,
 	\Omega}$ for all $F\in \mathcal O(\Omega)\otimes \mathcal
 M_{k\times k}(\mathbb C), k\geq 1$.  We say that the $\Omega$ has the property $P$ if following holds: if $\Omega$ is a spectral set for a commuting $ m$-tuple of operators $\mathbf{T},$ then it is a complete spectral set for $\mathbf{T}$. A commuting $m$-tuples of operators $\mathbf{T}$ with $\Omega$ as a spectral set, have a $\partial \Omega$ normal  dilation if there exists a Hilbert space $\mathcal K$ containing $\mathcal H$ as a subspace and a commuting $m$-tuples of normal operators $\mathbf{N}=(N_1,\ldots,N_m)~{\rm{ on }}~\mathcal K~{\rm{with}}~\sigma(N)\subseteq \partial \Omega$ such that
$$P_{\mathcal H}F(\mathbf N)\mid_{\mathcal H}=F(\mathbf T) ~{\rm{for~ all~}} F\in \mathcal O(\Omega).$$
In 1969, Arveson \cite{A} demonstrated that a commuting $m$-tuple of operators $\mathbf{T}$ having $\Omega$ as a spectral set for $\mathbf{T}$ admits a $\partial \Omega$ normal dilation if and only if it satisfies the property $P.$  J. Agler \cite{agler} proved in 1984 that the annulus possesses the property $P$. M. Dristchell and S. McCullough \cite{michel} showed that the property $P$ is not satisfied for any domain in $\mathbb{C}$ of connectivity $n\geq 2$. In the multi-variable setting, both the symmetrized bi-disc and the bi-disc possess the property $P$, as demonstrated by Agler and Young \cite{young} and Ando \cite{paulsen}, respectively. The first counterexample in the setting of multi-variable context was given by Parrott \cite{paulsen} for $\mathbb D^n$ when $n > 2. $ G. Misra \cite{GM,sastry}, V. Paulsen \cite{vern}, and E. Ricard \cite{pisier} demonstrated that $\Omega$ cannot possess property $P$ if it is a ball in $\mathbb{C}^m$, with respect to some norm $\|\cdot\|_{\Omega}$, for $m \geq 3$. It is demonstrated in \cite{cv} that if $B_1$ and $B_2$ are not simultaneously diagonalizable via unitary, then the set $ \Omega_{\mathbf B}:=   \{(z_1
  ,z_2) :\|z_1 B_1 + z_2 B_2 \|_{\rm op} < 1\}, $ fails to possess the property $P,$ where $\mathbf B=(B_1, B_2)$ in $\mathbb C^2 \otimes \mathcal
  M_2(\mathbb C)$ and $B_1$ and $B_2$ are independent.  

We recall the definition of generalized tetrablock from \cite{Zapalowski}. Let $n \geqslant 2$, $s \leqslant n$, and $r_1, \dots, r_s \geqslant 1$ be such that $\sum_{j=1}^{s} r_j = n$. Let us consider the set $$A(r_1, \dots, r_s) = \{0, \dots, r_1\} \times \dots \times \{0, \dots, r_s\} \setminus \{(0, \dots, 0)\}$$ and introduce an order on this set. For any two elements $\alpha = (\alpha_1, \dots, \alpha_s), \beta = (\beta_1, \dots, \beta_s) \in A(r_1, \dots, r_s)$, we assert that 

\begin{equation*}
\alpha < \beta  ~\text{if and only if }~\alpha_{j_0} < \beta_{j_0}, ~\text{where}~ j_0 = \max\{j : \alpha_j \ne \beta_j\}.
\end{equation*}
For ${\textbf{x}=(x_{1},\dots,x_{N})\in \mathbb{C}^{N}} $ and ${\bold z=(z_{1},\dots,z_{s})\in \mathbb{C}^{s}}, $ we define 
\begin{equation}\label{Rz1}
R^{(n;s;r_1,\ldots,r_s)}_{\textbf{x}}(\bold z):= 1+\sum_{j=1}^{N} (-1) ^{|\alpha^{(j)}|}x_{j}{\textbf z}^{\alpha^{(j)}}.
\end{equation}
Let $$\mathcal A_{(n;s;r_1,\ldots,r_s)}:=\{\textbf{x}=(x_1,\ldots,x_N)\in \mathbb C^N:R^{(n;s;r_1,\ldots,r_s)}_{\textbf{x}} (\bold z)\neq 0~{\rm{for~all}}~\bold z\in {\bar{\mathbb D}}^s\}$$ and 
$$\mathcal  B_{(n;s;r_1,\ldots,r_s)}:=\{\textbf{x}=(x_1,\ldots,x_N)\in \mathbb C^N:R^{(n;s;r_1,\ldots,r_s)}_{\textbf{x}}(\bold z)\neq 0~{\rm{for~all}}~\bold z\in {\mathbb D}^s\}. $$ We call the set $\mathcal  A_{(n;s;r_1,\ldots,r_s)}$ as a generalized tetrablock. 

Let  $E(n;s;r_{1},\dots,r_{s})\subset \mathcal M_{n\times n}(\mathbb{C})$ is the vector subspace comprising block diagonal matrices, defined as follows:
\begin{equation}\label{ls}
    	E=E(n;s;r_{1},...,r_{s}):=\{\operatorname{diag}[z_{1}I_{r_{1}},....,z_{s}I_{r_{s}}]\in \mathcal M_{n\times n}(\mathbb{C}): z_{1},...,z_{s}\in \mathbb{C}\},
\end{equation}
 where $\sum_{i=1}^{s}r_i=n.$ Let $$\Omega_{E(n;s;r_{1},\dots,r_{s})}:=\{A\in \mathcal M_{n\times n}(\mathbb{C}):\mu_{E(n;s;r_{1},\dots,r_{s})}(A)<1\}.$$ Define $$\mathcal I^{j}:=\{(i_1,\ldots,,i_j)\in {\mathbb N}^j:1\leq i_1<\ldots<i_j\leq n, j\leq n\}.$$ For each $\alpha\in A(r_1,\ldots,r_s)$, we define $I^{|\alpha|}_{\alpha}$ as follows:
   \begin{align}
 \mathcal I^{|\alpha|}_{\alpha}:\nonumber&=\{(i_1,\ldots,,i_{|\alpha|})\in \mathcal I^{|\alpha|}:r_0+r_1+\ldots+r_{j-1}+1\leq i_{(\alpha_1+\ldots+\alpha_{j-1}+1)}\\&<\ldots <i_{(\alpha_1+\ldots+\alpha_{j})}\leq r_1+\ldots+r_{j}, j=1,\ldots,s\},\end{align} where $r_0:=0$ and $\alpha_0=0.$ Observe that $\cup_{j=1}^{n}\mathcal I^{j}=\cup_{j=1}^{N}\mathcal I_{\alpha^{(j)}}^{|\alpha^{(j)}|}$ and $\mathcal  I_{\alpha}^{|\alpha|}\cap \mathcal I_{\beta}^{|\beta|}=\emptyset$ for $\alpha \neq \beta.$ For $I\in \mathcal I^{j}$ and $A\in \mathcal M_{n\times n}(\mathbb{C}),$ let $A_{I}$ indicate the $j\times j$ submatrix of $A$ with rows and columns are indexed by $I.$ We define a polynomial map $\pi_{E(n;s;r_{1},\dots,r_{s})}: \mathcal M_{n\times n}(\mathbb{C}) \to \mathbb C^N$ as follows:
   $$\pi_{E(n;s;r_{1},\dots,r_{s})}(A):=\big(\sum_{I\in \mathcal I_{\alpha^{(1)}}^{|\alpha^{(1)}|}} \det A_I,\ldots,\sum_{I\in \mathcal I_{\alpha^{(N)}}^{|\alpha^{(N)}|}} \det A_I\big).$$ Define 
  $$G_{E{(n;s;r_1,\ldots,r_s)}}:=\{\pi_{E(n;s;r_{1},\dots,r_{s})}(A):\mu_{E(n;s;r_{1},\dots,r_{s})}(A)<1\}$$ $${\rm{and }}$$
   $$\Gamma_{E{(n;s;r_1,\ldots,r_s)}}:=\{\pi_{E(n;s;r_{1},\dots,r_{s})}(A):\mu_{E(n;s;r_{1},\dots,r_{s})}(A)\leq 1\}.$$   
It is important to observe that $G_{E{(n;s;r_1,\ldots,r_s)}}$ forms an open set, whereas  $\Gamma_{E{(n;s;r_1,\ldots,r_s)}}$ represents the closure of $G_{E{(n;s;r_1,\ldots,r_s)}}.$  However, if $$E=E(n;1;n)=\{zI_n\in \mathcal M_{n\times n}(\mathbb{C}):z\in \mathbb{C}\},$$ then the domain associated with this $\mu_{E}$-synthesis problem indicates a \textit{symmetrized polydisc} \cite{costara}. For the case where $n=2,$ the corresponding domain is identified as \textit{symmetrized bidisc} \cite{JAgler,ay,ay1,JAY, JANY}. If $$E=E(n;2;1,n-1)=\{\operatorname{diag}(z_1,z_2I_{n-1})\in \mathcal M_{n\times n}(\mathbb{C}):z_1,z_2\in \mathbb{C}\},$$ then the domain associated with this $\mu_{E}$-synthesis problem is defined as $\mu_{1,n}$-\textit{quotients} \cite{Bharali} and the corresponding domains are designated as $\mathbb E_n$ \cite{Bharali}. The domain  $G_{E(2;2;1,1)}$ is represented by the $\mu_{1,2}$-\textit{quotients} and its closure is denoted by $\Gamma_{E(2;2;1,1)}$ which is referred as tetrablock. We use the notation $G_{E(3;2;1,2)}$ in place of $\mathbb E_3$, where $E(3;2;1,2)$ represents the vector subspace of $\mathcal M_{3\times 3}(\mathbb C)$ defined by the following form:
			$$E(3;2;1,2)=\left\{\left(\begin{array}{cc}z_1& 0\\0 &z_2I_2
			\end{array}\right)\in \mathbb{C}^{3\times 3}:z_1,z_2\in \mathbb{C} \right\}. $$ The closure of $G_{E(3;2;1,2)}$ is denoted as $\Gamma_{E(3;2;1,2)}$. We recall the definition of the domain $G_{E(3;2;1,2)}$ from \cite{Bharali}. 
 Let $\tilde{\textbf{x}}=(x_1,x_2,x_3,y_1,y_2)\in \mathbb{C}^{5}$ and for $\tilde{\textbf{z}}=(z_1,z_2)\in \bar{\mathbb{D}}^2,$ consider the polynomial $$R_3(\tilde{\textbf{z}};\tilde{\textbf{x}})=\left(y_2z_2^2-y_1z_2+1\right)-z_1\left(x_3z_2^2-x_2z_2+x_1\right)=Q_3(z_2;y_1,y_2)-z_1P_3(z_2;x_1,x_2,x_3).$$ 
We define the domains $\mathcal  A_{E(3;2;1,2)}$ and $\mathcal  B_{E(3;2;1,2)}$ as follows
			$$\mathcal  A_{E(3;2;1,2)}:=\left \{\tilde{\textbf{x}}\in \mathbb{C}^{5}:R_3(\tilde{\textbf{z}};\tilde{\textbf{x}})\neq0, \,\, \text{for all} \,\, \tilde{\textbf{z}}\in \bar{\mathbb{D}}^2 \right \}$$ $${\rm{and }}$$ $$\mathcal  B_{E(3;2;1,2)}:=\left \{\tilde{\textbf{x}}\in \mathbb{C}^{5}:R_3(\tilde{\textbf{z}};\tilde{\textbf{x}})\neq0, \,\, \text{for all} \,\, \tilde{\textbf{z}}\in \mathbb{D}^2 \right \}.$$ Let us consider the polynomial map $\pi_{{E(3;2;1,2)}}:\mathcal M_{3\times 3}(\mathbb C)\to \mathbb C^5$ \cite{Bharali} which is defined by 
			\begin{equation}\label{pi123}\pi_{{E(3;2;1,2)}}(A):=\left( a_{11},\det \left(\begin{smallmatrix} a_{11} & a_{12}\\
					a_{21} & a_{22}
				\end{smallmatrix}\right)+\det \left(\begin{smallmatrix}
					a_{11} & a_{13}\\
					a_{31} & a_{33}
				\end{smallmatrix}\right),\operatorname{det}A, a_{22}+a_{33}, \det  \left(\begin{smallmatrix}
					a_{22} & a_{23}\\
					a_{32} & a_{33}\end{smallmatrix}\right)\right).
					\end{equation}			
			
Define $$G_{E(3;2;1,2)}:=\{\pi_{{E(3;2;1,2)}}(A):\mu_{E(3;2;1,2)}(A)<1\}.$$
The domain $G_{E(3;2;1,2)}$ is referred to as $\mu_{1,3}-$\textit{quotients} \cite{Bharali}. For $\tilde{\textbf{x}}=(x_1,x_2,x_3,y_1,y_2)\in \mathbb{C}^{5},$ we consider the rational function 
\begin{equation}\label{psi1}\Psi_{3}(z_2,\tilde{\textbf{x}})=\frac{P_3(z_2;x_1,x_2,x_3)}{Q_3(z_2;y_1,y_2)}=\frac{x_3z_2^2-x_2z_2+x_1}{y_2z_2^2-y_1z_2+1}~{\rm{ for ~all}}~z_2\in \mathbb C ~{\rm{with}}~y_2z_2^2-y_1z_2+1\neq0.\end{equation}
The following proposition characterizes $G_{E(3;2;1,2)}$ [Theorem $3.5$,~\cite{Bharali}]. 
	\begin{prop}\label{bhch}
				For $\tilde{\textbf{x}}\in \mathbb{C}^{5}$ the following are equivalent:
				\begin{enumerate}
					\item $\tilde{\textbf{x}}\in G_{E(3;2;1,2)}$.
					\item $\tilde{\textbf{x}} \in \mathcal  A_{E(3;2;1,2)}.$
					\item $y_2z_2^2-y_1z_2+1\neq0$ for all $z_2\in \bar{\mathbb{D}}$ and $\|\Psi_{3}(\cdot,\tilde{\textbf{x}})\|_{H^{\infty}(\bar{\mathbb{D}})}:=\sup_{z_2\in \bar{\mathbb{D}}}|\Psi_{3}(z_2,\tilde{\textbf{x}})|<1$.  Moreover, if  $x_3=x_1y_2,x_2=x_1y_1$ then $|x_1|<1$.
					\item For every  $z_1\in \bar{\mathbb{D}}$, we have
					$\left(\frac{y_1-z_1x_2}{1-z_1x_1},\frac{y_2-z_1x_3}{1-z_1x_1}\right)\in G_{E(2;1;2)}.$ 
				\end{enumerate}
							\end{prop} 

\begin{prop}\label{bhch1}
				For $\tilde{\textbf{x}}\in \mathbb{C}^{5}$ the following are equivalent:
				\begin{enumerate}
					\item $\tilde{\textbf{x}}\in \Gamma_{E(3;2;1,2)}$.
					\item $\tilde{\textbf{x}} \in \mathcal  B_{E(3;2;1,2)}.$
					\item $y_2z_2^2-y_1z_2+1\neq 0$ for all $z_2\in {\mathbb{D}}$ and $\|\Psi_{3}(\cdot,\tilde{\textbf{x}})\|_{H^{\infty}({\mathbb{D}})}:=\sup_{z_2\in {\mathbb{D}}}|\Psi_{3}(z_2,\tilde{\textbf{x}})|\leq 1$.  Moreover, if  $x_3=x_1y_2,x_2=x_1y_1$ then $|x_1|\leq 1$.
					\item For every  $z_1\in {\mathbb{D}}$, we have
					$\left(\frac{y_1-z_1x_2}{1-z_1x_1},\frac{y_2-z_1x_3}{1-z_1x_1}\right)\in \Gamma_{E(2;1;2)}.$ 
				\end{enumerate}
							\end{prop} 														
From one of the characterization of the domain $G_{E{(3;3;1,1,1)}}$ and its closure $\Gamma_{E{(3;3;1,1,1)}}$ \cite{pal1}, we have
\begin{equation}
\begin{aligned}
G_{E{(3;3;1,1,1)}}:=\Big \{ \textbf{x}=(x_1=a_{11}, x_2=a_{22}, x_3=a_{11}a_{22}-a_{12}a_{21}, x_4=a_{33}, x_5=a_{11}a_{33}-a_{13}a_{31},&\\ x_6=a_{22}a_{33}-a_{23}a_{32},x_7=\det A )\in \mathbb C^7: A\in \mathcal M_{3\times 3}(\mathbb C)~{\rm{and}}~\mu_{E(3;3;1,1,1)}(A)< 1 \Big \}
\end{aligned}
\end{equation}
$${\rm{and }}$$
\begin{equation}\label {GammaE}
\begin{aligned}
\Gamma_{E{(3;3;1,1,1)}}:=\Big \{\textbf{x}=(x_1=a_{11}, x_2=a_{22}, x_3=a_{11}a_{22}-a_{12}a_{21}, x_4=a_{33}, x_5=a_{11}a_{33}-a_{13}a_{31},&\\ x_6=a_{22}a_{33}-a_{23}a_{32},x_7=\det A)\in \mathbb C^7: A\in \mathcal M_{3\times 3}(\mathbb C)~{\rm{and}}~\mu_{E(3;3;1,1,1)}(A)\leq 1\Big \}.
\end{aligned}
\end{equation}

The study of the symmetrized bidisc $\Gamma_{E(2;1;2)}$ and the tetrablock $\Gamma_{E(2; 2;1,1)}$ holds significant importance in the realms of complex analysis and operator theory. Edigarian, Kosinski, Costara, and Zwonek investigated the geometric properties of the symmetrized bidisc $\Gamma_{E(2;1;2)}$ and the tetrablock $\Gamma_{E(2; 2;1,1)}$ \cite{Abouhajar, Ccostara, Ae1, Ae2, Ae3, Ae4, Zwonek}. Young studied the symmetrized bidisc $\Gamma_{E(2;1;2)}$ and the tetrablock $\Gamma_{E(2; 2;1,1)}$, in collaboration with several co-authors \cite{Abouhajar,JAgler,young,ay,ay1,JAY, JANY}, from an operator theoretic point of view. The symmetrized bidisc $\Gamma_{E(2;1;2)}$ and the tetrablock $\Gamma_{E(2; 2;1,1)}$ are both non-convex domains. Also, they cannot be exhausted by domains biholomorphic to convex ones. However, both domains have identical Caratheodory distance and the Lempert function. Normal dilation is valied for a pair of commuting operators that possess the symmetrized bidisc $\Gamma_{E(2;1;2)}$ as a spectral set \cite{JAgler, young}. S. Biswas and S. S. Roy investigated the properties of $\Gamma_{E(n; 1; n)}$ isometries, $\Gamma_{E(n; 1; n)}$ unitaries, and the Wold decomposition associated with $\Gamma_{E(n; 1; n)}$ \cite{SS, A. Pal}.  T. Bhattacharyya studied the tetrablock isometries, tetrablock unitaries, the Wold decomposition, and conditional dilation for the tetrablock $\Gamma_{E(2; 2;1,1)}$  \cite{Bhattacharyya}. However, the question of whether the tetrablock $\Gamma_{E(2; 2;1,1)}$ and $\Gamma_{E(n; 1; n)}$ possess the property $P$ remains unresolved.

Various characterizations of $G_{E(3; 3; 1, 1, 1)} $ have been discussed in \cite{pal1}. By the definition of $G_{E(3; 3; 1, 1, 1)} ,$ it follows that 
\begin{equation*}
\begin{aligned}
1 - zx_2 - wx_4 + zwx_6\neq 0, 1 - zx_1 - wx_4 + zwx_5\neq 0~{\rm{and}}~1 - zx_1 - wx_2 + zwx_3\neq 0
\end{aligned}
\end{equation*}
for each $\textbf{x}=(x_1,x_2,x_3,x_4,x_5,x_6,x_7)\in G_{E(3; 3; 1, 1, 1)} $ and any $z,w\in \mathbb {\bar{D}}$. Consider the rational functions $\Psi^{(1)}, \Psi^{(2)}$ and $\Psi^{(3)}$ defined on $\mathbb {\bar{D}}^2\times G_{E(3; 3; 1, 1, 1)}$  as follows:	\begin{equation}\label{psi11}
		\begin{aligned}
			\Psi^{(1)}(z, w, \textbf{x})
			&= \frac{x_1 - zx_3 - wx_5 + zwx_7}{1 - zx_2 - wx_4 + zwx_6}, \, z,w \in \overline{\mathbb{D}},
		\end{aligned}
	\end{equation}
	\begin{equation}\label{psi12}
		\begin{aligned}
			\Psi^{(2)}(z, w, \textbf{x})
			&= \frac{x_2 - zx_3 - wx_6 + zwx_7}{1 - zx_1 - wx_4 + zwx_5}, \, z,w \in \overline{\mathbb{D}},
		\end{aligned}
	\end{equation}
	
	\begin{equation}\label{psi13}
		\begin{aligned}
			\Psi^{(3)}(z, w, \textbf{x})
			&= \frac{x_4 - zx_5 - wx_6 + zwx_7}{1 - zx_1 - wx_2 + zwx_3}, \, z,w \in \overline{\mathbb{D}}.
		\end{aligned}
	\end{equation}
We provide various characterization from \cite{pal1} for determining whether a point $\textbf{x}$ belongs to $G_{E(3;3;1,1,1)}$. 
\begin{thm}\label{mainthm}
				For $\textbf{x}=(x_1,\ldots,x_7)\in \mathbb C^7$ the following are equivalent.
				\begin{enumerate}
					\item[1.] $\textbf{x}=(x_1,\ldots,x_7)\in G_{E(3;3;1,1,1)}$;
					
					\item[2.] $R_{\bf{x}}^{(3;3;1,1,1)}(\textbf{z})= 1-x_1z_1-x_2z_2+x_3z_1z_2-x_4z_3+x_5z_1z_3+x_6z_2z_3-x_7z_1z_2z_3\neq 0$, for all, $z_1,z_2,z_3\in \bar{\mathbb{D}}$;
					
					\item[$3.$] $(x_2,x_4,x_6)\in G_{E(2;2;1,1)}$ and 
					$\|\Psi^{(1)}(\cdot,\textbf{x})\|_{H^{\infty}(\bar{\mathbb{D}}^{2})}=\|\Psi^{(1)}(\cdot,\textbf{x})\|_{H^{\infty}(\mathbb{T}^{2})}< 1$ and if  we suppose that $x_{7}=x_{6}x_{1},\,\,x_{3}=x_{2}x_{1},\,\,x_{5}=x_{4}x_{1}$ then $|x_1|<1;$
					
					\item[$3^{\prime}.$]
					$(x_1,x_4,x_5)\in G_{E(2;2;1,1)}$ and 
					$\|\Psi^{(2)}(\cdot,\textbf{x})\|_{H^{\infty}(\bar{\mathbb{D}}^{2})}=\|\Psi^{(2)}(\cdot,\textbf{x})\|_{H^{\infty}(\mathbb{T}^{2})}< 1$  and if  we suppose that $x_{7}=x_{5}x_{2},\,\,x_{3}=x_{2}x_{1},\,\,x_{6}=x_{4}x_{2}$ then $|x_2|<1;$
					
					\item[$3^{\prime\prime}.$]
					$(x_1,x_2,x_3)\in G_{E(2;2;1,1)}$ and 
					$\|\Psi^{(3)}(\cdot,\textbf{x})\|_{H^{\infty}(\bar{\mathbb{D}}^{2})}=\|\Psi^{(3)}(\cdot,\textbf{x})\|_{H^{\infty}(\mathbb{T}^{2})}< 1$  and if  we suppose that $x_{7}=x_{3}x_{4},\,\,x_{6}=x_{2}x_{4},\,\,x_{5}=x_{4}x_{1}$ then $|x_4|<1$ ;
					
					\item[$4.$] $\left(\frac{x_1-z_3x_5}{1-x_4z_3},\frac{x_2-z_3x_6}{1-x_4z_3},\frac{x_3-z_3x_7}{1-x_4z_3}\right)\in G_{E(2;2;1,1)}$ for all $z_3\in \bar{\mathbb D};$

					\item[$4^{\prime}.$] $\left(\frac{x_2-z_1x_3}{1-x_1z_1},\frac{x_4-z_1x_5}{1-x_1z_1},\frac{x_6-z_1x_7}{1-x_1z_1}\right)\in G_{E(2;2;1,1)}$  for all $z_1\in \bar{\mathbb D}$;
					
					\item[$4^{\prime\prime}.$] $\left(\frac{x_4-z_2x_6}{1-x_2z_2},\frac{x_1-z_2x_3}{1-x_2z_2},\frac{x_5-z_2x_7}{1-x_2z_2}\right)\in G_{E(2;2;1,1)}$  for all $z_2\in \bar{\mathbb D}$;

					\item[$5.$]
					There exists a $2\times 2$ symmetric matrix $A(z_3)$ with $\|A(z_3)\|<1$ such that $$A_{11}(z_3)=\frac{x_1-z_3x_5}{1-x_4z_3},A_{22}(z_3)=\frac{x_2-z_3x_6}{1-x_4z_3}~{\rm{ and}}~ \det(A(z_3))=\frac{x_3-z_3x_7}{1-x_4z_3}~{\rm{for~all}}~ z_3\in \bar{\mathbb D};$$

					\item[$5^{\prime}.$]
					There exists a $2\times 2$ symmetric matrix $B(z_1)$ with $\|B(z_1)\|<1$ such that $$B_{11}(z_1)=\frac{x_2-z_1x_3}{1-x_1z_1},B_{22}(z_1)=\frac{x_4-z_1x_5}{1-x_1z_1}~{\rm{ and }}~\det(B(z_1))=\frac{x_6-z_1x_7}{1-x_1z_1}~{\rm{for ~all }}~z_1\in \bar{\mathbb D};$$
					
					\item[$5^{\prime\prime}.$]
					There exists a $2\times 2$ symmetric matrix $C(z_2)$ with $\|C(z_2)\|<1$ 		such that $$C_{11}(z_2)=\frac{x_4-z_2x_6}{1-x_2z_2},C_{22}(z_2)=\frac{x_1-z_2x_3}{1-x_2z_2}~{\rm{and}}~ \det(C(z_2))=\frac{x_5-z_2x_7}{1-x_2z_2}~{\rm{for~all}}~z_2\in \bar{\mathbb D};$$
					
				\item[$6$.]  There exists a $3\times 3$ matrix $A\in \mathcal M_{3\times 3}(\mathbb C)$ such that \small{$x_{1}=a_{11},x_{2}=a_{22},x_{3}=a_{11}a_{22}-a_{12}a_{21},$}\\$x_{4}=a_{33},x_{5}=a_{11}a_{33}-a_{13}a_{31}, x_{6}=a_{33}a_{22}-a_{23}a_{32}~{\rm{and}}~x_{7}=\operatorname{det}A,$ $$\sup_{(z_2,z_3)\in\bar{\mathbb D}^2} |\mathcal G_{\mathcal F_{A}(z_3)}(z_2)|=\sup_{(z_2,z_3)\in\bar{\mathbb D}^2} |\mathcal G_{A}\left(\left(\begin{smallmatrix}z_2 & 0 \\0 &z_{3}
				\end{smallmatrix}\right)\right)|<1$$ and $\det\left(I_{2}-\left(\begin{smallmatrix}a_{22}&a_{23}\\a_{32} & a_{33}
				\end{smallmatrix}\right)\left(\begin{smallmatrix}z_2 & 0 \\0 &z_{3}
				\end{smallmatrix}\right)\right)\neq 0$ for all $(z_2,z_3)\in\bar{\mathbb D}^2$.

				\item[$6^{\prime}$.]  There exists a $3\times 3$ matrix $\tilde{A}\in \mathcal M_{3\times 3}(\mathbb C)$ such that \small{$x_{1}=\tilde{a}_{33},x_{2}=\tilde{a}_{11},x_{3}=\tilde{a}_{11}\tilde{a}_{33}-\tilde{a}_{13}\tilde{a}_{31},$} \\$x_{4}=\tilde{a}_{22},x_{5}=\tilde{a}_{22}\tilde{a}_{33}-\tilde{a}_{23}\tilde{a}_{32}, x_{6}=\tilde{a}_{11}\tilde{a}_{22}-\tilde{a}_{12}\tilde{a}_{21}~{\rm{and}}~{x}_{7}=\operatorname{det} \tilde{A},$ $$\sup_{(z_1,z_2)\in\bar{\mathbb D}^2} |\mathcal G_{\mathcal F_{\tilde{A}}(z_2)}(z_1)|=\sup_{(z_1,z_2)\in\bar{\mathbb D}^2} |\mathcal G_{\tilde{A}}\left(\left(\begin{smallmatrix}z_1 & 0 \\0 &z_{2}
				\end{smallmatrix}\right)\right)|<1$$ and  $\det\left(I_{2}-\left(\begin{smallmatrix}\tilde{a}_{11}&\tilde{a}_{12}\\\tilde{a}_{21} & \tilde{a}_{22}
				\end{smallmatrix}\right)\left(\begin{smallmatrix}z_1 & 0 \\0 &z_{2}
				\end{smallmatrix}\right)\right)\neq 0$ for all $(z_1,z_2)\in\bar{\mathbb D}^2$.
				
				\item[$6^{\prime\prime}$.] There exists a $3\times 3$ matrix $\tilde{B}\in \mathcal M_{3\times 3}(\mathbb C)$  such that \small{$ x_1=\tilde{b}_{22},x_2=\tilde{b}_{33},x_3=\tilde{b}_{22}\tilde{b}_{33}-\tilde{b}_{23}\tilde{b}_{32},$} \\$ x_4=\tilde{b}_{11},x_{5}=\tilde{b}_{11}\tilde{b}_{22}-\tilde{b}_{12}\tilde{b}_{21}, x_{6}=\tilde{b}_{11}\tilde{b}_{33}-\tilde{b}_{13}\tilde{b}_{31}~{\rm{and}}~x_{7}=\operatorname{det} \tilde{B},$  $$\sup_{(z_1,z_3)\in\bar{\mathbb D}^2} |\mathcal G_{\mathcal F_{\tilde{B}}(z_1)}(z_3)|=\sup_{(z_1,z_3)\in\bar{\mathbb D}^2} |\mathcal G_{\tilde{B}}\left(\left(\begin{smallmatrix}z_1 & 0 \\0 &z_{3}
				\end{smallmatrix}\right)\right)|<1$$ and $\det\left(I_{2}-\left(\begin{smallmatrix}\tilde{b}_{11}&\tilde{b}_{13}\\\tilde{b}_{31} & \tilde{b}_{33}
				\end{smallmatrix}\right)\left(\begin{smallmatrix}z_1 & 0 \\0 &z_{3}
				\end{smallmatrix}\right)\right)\neq 0$  for all $(z_1,z_3)\in\bar{\mathbb D}^2$.
					
				\end{enumerate}
			\end{thm}
The following theorem establishes a connection between $G_{E(3;2;1,2)}$ and $G_{E(3;3;1,1,1)}$.
			\begin{thm}[Theorem $2.45$, \cite{pal1}]\label{matix AE12}
				Suppose $\textbf{x}=(x_1,x_2,x_3,x_4,x_5,x_6,x_7) \in \mathbb{C}^7 .$ Then the following conditions are equivalent:
\begin{enumerate}
\item  $\textbf{x} \in G_{E(3;3;1,1,1)}$;
\item  $(x_1,x_3+\eta x_5,\eta x_7,x_2+\eta x_4,\eta x_6) \in G_{E(3;2;1,2)}$ for all $\eta \in \mathbb{T}$;
\item $(x_1,x_3+\eta x_5,\eta x_7,x_2+\eta x_4,\eta x_6) \in G_{E(3;2;1,2)}$ for all $\eta \in \bar{\mathbb{D}}.$
\end{enumerate}
			\end{thm}

{Let $\Omega$ be a domain in $ \mathbb {C} ^m$ and $\bar{\Omega}$ denote its closure. Let $C\subseteq \bar{\Omega}.$ We define $C$ to be a boundary  for $\Omega$ if every function in the algebra $\mathcal{A}(\Omega)$ of functions continuous on $\bar{\Omega}$ and analytic on  $\Omega$, attains its maximum modulus on $C.$  The distinguished boundary of $\Omega$ is the smallest closed subset of $\bar{\Omega}$ on which every function in $\mathcal{A}(\Omega)$ attains its maximum modulus. We denote it as $b\Omega.$ The distinguished boundary of  $\Omega$ exists if $\bar{\Omega}$ is polynomially convex [Corollary $2.2.10$, \cite{Browder}]. The polynomially convexity of  $\Gamma_{E(3;3;1,1,1)} $ and $\Gamma_{E(3;2;1,2)}$ ensure the existence of its distinguished boundary.

Set $\tilde{x}_1(z_1)=\frac{x_2-z_1x_3}{1-x_1z_1},~\tilde{x}_2(z_1)=\frac{x_4-z_1x_5}{1-x_1z_1},~\tilde{x}_3(z_1)=\frac{x_6-z_1x_7}{1-x_1z_1},~\tilde{y}_1(z_2)=\frac{x_1-z_2x_3}{1-x_2z_2},~\tilde{y}_2(z_2)=\frac{x_4-z_2x_6}{1-x_2z_2},~\tilde{y}_3(z_2)=\frac{x_5-z_2x_7}{1-x_2z_2}, \tilde{z}_1(z_3)=\frac{x_1-z_3x_5}{1-x_4z_3},\tilde{z}_2(z_3)=\frac{x_2-z_3x_6}{1-x_4z_3}~{\rm{and}}~\tilde{z}_3(z_3)=\frac{x_3-z_3x_7}{1-x_4z_3}.$ 
 For each $(z_2,z_3)\in \mathbb{D}^2$ and $\textbf{x}=(x_1,x_2,x_3,x_4,x_5,x_6,x_7) \in \Gamma_{E(3;3;1,1,1)}$, it is important to observe the following:
	\begin{align}\label{psi11154}
		\Psi^{(1)}((z_2,z_3),\textbf{x})\nonumber&=\frac{x_1-z_2x_3-z_3x_5+z_2z_3x_7}{1-z_2x_2-z_3x_4+z_2z_3x_6}\\ \nonumber &= \frac{\frac{x_1-z_3x_5}{1-z_3x_4}- z_2\frac{x_3-z_3x_7}{1-z_3x_4}}{1- z_2\frac{x_2-z_3 x_6}{1-z_3 x_4}}
		\\ \nonumber&=\frac{\tilde{z}_1(z_3)-z_2\tilde{z}_3(z_3)}{1-z_2\tilde{z}_2(z_3)}\\&=\Psi(z_2,(\tilde{z}_1(z_3),\tilde{z}_2(z_3),\tilde{z}_3(z_3)))
	\end{align}
$${\rm{and }}$$
\begin{align}\label{psi111546}
		\Psi^{(1)}((z_2,z_3),\textbf{x})\nonumber&=\frac{x_1-z_2x_3-z_3x_5+z_2z_3x_7}{1-z_2x_2-z_3x_4+z_2z_3x_6}\\  &=\Psi(z_3,(\tilde{y}_1(z_2),\tilde{y}_2(z_2),\tilde{y}_3(z_2))). 
	\end{align}
In a similar manner,  for $\textbf{x}=(x_1,x_2,x_3,x_4,x_5,x_6,x_7) \in \Gamma_{E(3;3;1,1,1)}$, we note that
\begin{align}\label{psi1115467}
		\Psi^{(2)}((z_1,z_3),\textbf{x})=\Psi(z_3,(\tilde{x}_1(z_1),\tilde{x}_2(z_1),\tilde{x}_3(z_1)))=\Psi(z_1,(\tilde{z}_1(z_3),\tilde{z}_2(z_3),\tilde{z}_3(z_3)))~{\rm{for~all~}} z_1,z_3\in \mathbb D
	\end{align}
and \begin{align}\label{psi11154678}
		\Psi^{(3)}((z_1,z_2),\textbf{x})=\Psi(z_2,(\tilde{x}_2(z_1),\tilde{x}_1(z_1),\tilde{x}_3(z_1)))=\Psi(z_1,(\tilde{y}_2(z_2),\tilde{y}_1(z_2),\tilde{y}_3(z_2)))~{\rm{for~all~}} z_1,z_2\in \mathbb D
	\end{align} respectively. 	
Let $\mathcal U(3)$ be the set of all $3\times 3$ unitary matrix and $$K=\{\textbf{x}=(x_1,\ldots,x_7)\in \Gamma_{E(3;3;1,1,1)} :x_1=\bar{x}_6x_7, x_3=\bar{x}_4x_7,x_5=\bar{x}_2x_7 ~{\rm{and}}~|x_7|=1\}.$$ 
\begin{thm}[Theorem $4.4$, \cite{pal1}]\label{distinguish}
$\pi_{E(3;3;1,1,1)}(\mathcal U(3))\subseteq K.$
\end{thm}
Let	\[ K_1 = \{x = (x_1, x_2, x_3, y_1, y_2) \in \Gamma_{E(3;2;1,2)} : x_1 = \overline{y}_2 x_3, x_2 = \overline{y}_1 x_3, |x_3| = 1 \}. \]
\begin{thm}[Theorem $4.9$, \cite{pal1}]\label{distinguish1}
$\pi_{E(3;2;1,2)}(\mathcal U(3))\subseteq K_1.$
\end{thm}
Set $\tilde{\textbf{x}}^{(z_1)}=(\tilde{x_1}(z_1),\tilde{x_2}(z_1),\tilde{x_3}(z_1)), \tilde{\textbf{y}}^{(z_2)}= (\tilde{y_1}(z_2),\tilde{y_2}(z_2),\tilde{y_3}(z_2))$ and $\tilde{\textbf{z}}^{(z_3)}=(\tilde{z_1}(z_3),\tilde{z_2}(z_3),\tilde{z_3}(z_3)).$

\begin{thm}[Theorem $4.5$, \cite{pal1}]\label{relation bw}
	Let $\textbf{x}\in \mathbb{C}^7$. Then the following conditions are equivalent:

	\begin{enumerate}
		\item $\textbf{x}\in K=\{(x_1,\dots,x_7)\in \Gamma_{E(3;3;1,1,1)}:x_1=\bar{x}_6x_7,x_2=\bar{x}_5x_7,x_4=\bar{x}_3x_7,|x_7|=1\}$.
		
		\item 	$\begin{cases}
			\tilde{\textbf{z}}^{(z_3)}\in b\Gamma_{E(2;2;1,1)}~\rm{and}~\tilde{\textbf{y}}^{(z_2)}\in b\Gamma_{E(2;2;1,1)} ~~{\rm{for~~all}}~z_2,z_3\in \mathbb{T}~\rm{with}~|x_4|<1 ~\rm{and}~|x_2|<1 \\
			\tilde{\textbf{z}}^{(z_3)}\in b\Gamma_{E(2;2;1,1)} ~~{\rm{for~~all}} ~z_3\in \mathbb{D}~\rm{with}~|x_4|=1\\
	                \tilde{\textbf{y}}^{(z_2)}\in b\Gamma_{E(2;2;1,1)} ~~{\rm{for~~all}}~z_2\in \mathbb{D}~\rm{with}~|x_2|=1;
		\end{cases}$
	    \item 	$\begin{cases}
	\tilde{\textbf{x}}^{(z_1)}\in b\Gamma_{E(2;2;1,1)}~\rm{and}~\tilde{\textbf{y}}^{(z_2)}\in b\Gamma_{E(2;2;1,1)} ~~{\rm{for~~all}}~z_1,z_2\in \mathbb{T}~\text{with}~|x_1|<1 ~\rm{and}~|x_2|<1 \\
		  \tilde{\textbf{x}}^{(z_1)}\in b\Gamma_{E(2;2;1,1)}~~{\rm{for~~all}}~z_1\in \mathbb{D}~\rm{with}~|x_1|=1\\
			 \tilde{\textbf{y}}^{(z_2)}\in b\Gamma_{E(2;2;1,1)} ~~{\rm{for~~all}}~z_2\in \mathbb{D}~\rm{with}~|x_2|=1;
	     \end{cases}$
		\item 	$\begin{cases}
			\tilde{\textbf{z}}^{(z_3)}\in b\Gamma_{E(2;2;1,1)}~\rm{and}~\tilde{\textbf{x}}^{(z_1)}\in b\Gamma_{E(2;2;1,1)} ~~{\rm{for~~all}}~z_2,z_3\in \mathbb{T}~\rm{with}~|x_4|<1 ~\rm{and}~|x_1|<1 \\
		\tilde{\textbf{z}}^{(z_3)}\in b\Gamma_{E(2;2;1,1)}~~{\rm{for~~all}}~z_3\in \mathbb{D}~\rm{with}~|x_1|=1\\
			\tilde{\textbf{x}}^{(z_1)}\in b\Gamma_{E(2;2;1,1)} ~~{\rm{for~~all}}~z_1\in \mathbb{D}~\rm{with}~|x_4|=1.
	        	\end{cases}$
	   	\end{enumerate}
\end{thm}
Set $p_1(z)=\frac{2x_1-zx_2}{2-y_1z},p_2(z)=\frac{y_1-2zy_2}{2-y_1z}~{\rm{and}}~p_3(z)=\frac{x_2-2zx_3}{2-y_1z}.$ 
\begin{thm}[Theorem $4.9$,\cite{pal1}]\label{relation bw1}
	Let $\textbf{x}\in \mathbb{C}^5$. Then the following conditions are equivalent:

	\begin{enumerate}
		\item $\tilde{\textbf{x}}\in K_1=\{(x_1,x_2,x_3,y_1,y_2)\in \Gamma_{E(3;2;1,2)}:x_1=\bar{y}_2x_3, x_2=\bar{y}_1x_3, |x_3|=1\}.	$	
		\item 	$\begin{cases}
			(p_1(z),p_2(z),p_3(z))\in b\Gamma_{E(2;2;1,1)} ~~{\rm{for~~all}}~z\in \mathbb{T}~\rm{with}~|y_1|<2 \\
			(p_1(z),p_2(z),p_3(z))\in b\Gamma_{E(2;2;1,1)} ~~{\rm{for~~all}}~z\in \mathbb{D}~\rm{with}~|y_1|=2.
		\end{cases}$
		\end{enumerate}
\end{thm}
It is important to notice  that a point  $(x_1,x_3+\eta x_5,\eta x_7,x_2+\eta x_4,\eta x_6)\in K_1 ~{\rm{for~ all }}~\eta \in \mathbb T$  if and only if $\textbf{x}=(x_1,x_2,x_3,x_4,x_5,x_6,x_7)\in K$  \cite{pal1}.
By Theorem \ref{relation bw} (respectively, Theorem \ref{relation bw1}) and  Theorem \ref{distinguish} (respectively, Theorem \ref{distinguish1}), it is also noteworthy to determine whether a point $\textbf{x}\in b \Gamma_{E(3;3;1,1,1)}$ (respectively,  $\textbf{x}\in b \Gamma_{E(3;2;1,2)}$) is equivalent to $\textbf{x}\in K$ (respectively, $\textbf{x}\in K_1$) or not. We propose the following conjecture:

\noindent{\textbf{Conjecture}:} The subset $K$ ( respectively, $K_1$) of $\Gamma_{E(3;3;1,1,1)}$ ( respectively, $\Gamma_{E(3;2;1,2)}$ ) is distinguish boundary of $\mathcal A(G_{E(3;3;1,1,1)})$ ( respectively, $\mathcal A(G_{E(3;2;1,2)})$). 

We begin with the following definitions that will be essential for our discussion.

\begin{defn}\label{def-1}
		\begin{enumerate}
			\item If $\Gamma_{E(3; 3; 1, 1, 1)}$ is a spectral set for $\textbf{T} = (T_1, \dots, T_7)$, then a $7$-tuple of commuting bounded operators $\textbf{T}$ defined on a  Hilbert space $\mathcal{H}$ is referred to as a \textit{$\Gamma_{E(3; 3; 1, 1, 1)}$-contraction}.
			
			\item Let $(S_1, S_2, S_3)$ and $(\tilde{S}_1, \tilde{S}_2)$ be tuples of commuting bounded operators defined on a Hilbert space $\mathcal{H}$ with $S_i\tilde{S}_j = \tilde{S}_jS_i$ for $1 \leqslant i \leqslant 3$ and $1 \leqslant j \leqslant 2$. We say that  $\textbf{S} = (S_1, S_2, S_3, \tilde{S}_1, \tilde{S}_2)$ is a $\Gamma_{E(3; 2; 1, 2)}$-contraction if $ \Gamma_{E(3; 2; 1, 2)}$ is a spectral set for $\textbf{S}$.
			
\item A commuting $7$-tuple of normal operators $\textbf{N} = (N_1, \dots, N_7)$ defined on a Hilbert space $\mathcal{H}$ is  a \textit{$\Gamma_{E(3; 3; 1, 1, 1)}$-unitary} if the Taylor joint spectrum $\sigma(\textbf{N})$ is contained in the set $K$.  		
			\item A commuting $5$-tuple of normal operators $\textbf{M} = (M_1, M_2, M_3, \tilde{M}_1, \tilde{M}_2)$ on a Hilbert space $\mathcal{H}$ is referred as a \textit{$\Gamma_{E(3; 2; 1, 2)}$-unitary} if the Taylor joint spectrum $\sigma(\textbf{M})$ is contained in $K_1.$ 
					
\item A  $\Gamma_{E(3; 3; 1, 1, 1)}$-isometry (respectively, $\Gamma_{E(3; 2; 1, 2)}$-isometry) is defined as the restriction of a $\Gamma_{E(3; 3; 1, 1, 1)}$-unitary (respectively, $\Gamma_{E(3; 2; 1, 2)}$-unitary)  to a joint invariant subspace. In other words, a $\Gamma_{E(3; 3; 1, 1, 1)}$-isometry ( respectively, $\Gamma_{E(3; 2; 1, 2)}$-isometry) is a $7$-tuple (respectively, $5$-tuple) of commuting bounded operators that possesses simultaneous extension to a \textit{$\Gamma_{E(3; 3; 1, 1, 1)}$-unitary} (respectively, \textit{$\Gamma_{E(3; 2; 1, 2)}$-unitary}). It is important to observe that a $\Gamma_{E(3; 3; 1, 1, 1)}$-isometry (respectively, $\Gamma_{E(3; 2; 1, 2)}$-isometry ) $\textbf{V}=(V_1\dots,V_7)$ (respectively,  $\textbf{W}=(W_1,W_2,W_3,\tilde{W}_1,\tilde{W}_2)$) consists of commuting subnormal operators with $V_7$ (respectively, $W_3$)  is an isometry.

\item  We say that $\textbf{V}$ (respectively, $\textbf{W}$)   is a pure $\Gamma_{E(3; 3; 1, 1, 1)}$-isometry (respectively, pure $\Gamma_{E(3; 2; 1, 2)}$-isometry) if $V_7$ (respectively, $W_3$) is a  pure isometry,  that is, a shift of some multiplicity.			
\end{enumerate}
	\end{defn}	

In Section 2, we discuss various properties of $\Gamma_{E(3; 3; 1, 1, 1)}$-contractions and $\Gamma_{E(3; 2; 1, 2)}$-contractions, and investigate the fundamental equations associated with these operators. Section 3 is devoted to the study of $\Gamma_{E(3; 3; 1, 1, 1)}$-unitaries and $\Gamma_{E(3; 2; 1, 2)}$-unitaries, where we examine their properties and establish relationships between the two classes of unitaries. Section 4 focuses on $\Gamma_{E(3; 3; 1, 1, 1)}$-isometries and $\Gamma_{E(3; 2; 1, 2)}$-isometries. In this section, we present the Wold decomposition for both types of isometries. We also provide the structure theorem for pure $\Gamma_{E(3; 3; 1, 1, 1)}$-isometries and pure $\Gamma_{E(3; 2; 1, 2)}$-isometries.

	\section{$\Gamma_{E(3; 3; 1, 1, 1)}$-Contractions and $\Gamma_{E(3; 2; 1, 2)}$-Contractions}
	
In this section, we discuss $\Gamma_{E(3; 3; 1, 1, 1)}$-contractions and $\Gamma_{E(3; 2; 1, 2)}$-contractions. We start with few lemmas which will be helpful in the characterization of $\Gamma_{E(3; 3; 1, 1, 1)}$-contractions.

\begin{lem}\label{lem-2}
		Let $(x_1, \dots, x_7)$ be a point of  $ G_{E(3; 3; 1, 1, 1)}$. Then $|x_i| < 1$ for $1 \leqslant i \leqslant 7$.
	\end{lem}
	
	\begin{proof}
 By Theorem \ref{mainthm}, it follows that 	  $(x_1, \dots, x_7) \in G_{E(3; 3; 1, 1, 1)}$ if and only if 
		\begin{equation*}
			\begin{aligned}
				\Big(\frac{x_2 - z_1x_3}{1 - z_1x_1}, \frac{x_4 - z_1x_5}{1 - z_1x_1}, \frac{x_6 - z_1x_7}{1 - z_1x_1}\Big) \in G_{E(2; 2; 1, 1)} ~ \text{for} ~ z_1 \in \overline{\mathbb{D}}.
			\end{aligned}
		\end{equation*}
Thus, by the characterization of tetrablock [\cite{Abouhajar}, Theorem $2.2$], we get  $(x_1, x_2, x_3)$, $(x_1, x_4, x_5)$ and $(x_1, x_6, x_7)$ are in  $G_{E(2; 2; 1, 1)}$ and hence $|x_i| < 1$ for $1 \leqslant i \leqslant 7$. This completes the proof.
	\end{proof}
The following lemma's proof is identical to that of lemma \ref{lem-2}. We therefore omit the proof.

\begin{lem}\label{lemm-2}
		Let $(x_1, \dots, x_7)$ be a point of $\Gamma_{E(3; 3; 1, 1, 1)}$. Then $|x_i| \leq 1$ for $1 \leqslant i \leqslant 7$.
	\end{lem}

\begin{lem}\label{lem-1}
		Suppose $\textbf{x} = (x_1, \dots, x_7)$ is a point of $\mathbb{C}^7$.  The followings are equivalent:
		\begin{enumerate}
			\item $\textbf{x} \in G_{E(3; 3; 1, 1, 1)}.$ 
			
			\item $(\omega x_1, x_2, \omega x_3, \omega x_4, \omega^2 x_5, \omega x_6, \omega^2 x_7) \in G_{E(3; 3; 1, 1, 1)}$ for all $\omega \in \mathbb T.$
			
			\item[$(2^{'})$] $ (\omega x_1, \omega x_2, \omega^2 x_3, x_4, \omega x_5, \omega x_6, \omega^2 x_7) \in G_{E(3; 3; 1, 1, 1)}$ for all $\omega \in \mathbb T.$
			
			\item[$(2^{''})$] $ (x_1, \omega x_2, \omega x_3, \omega x_4, \omega x_5, \omega^2 x_6, \omega^2 x_7) \in G_{E(3; 3; 1, 1, 1)}$ for all $\omega \in \mathbb T.$
		\end{enumerate}
	\end{lem}
	
	\begin{proof}
		We only prove the equivalence of $(1)$ and $(2)$. The proof for other cases is analogous. By Theorem \ref{mainthm},  we have $\textbf{x}\in G_{E(3; 3; 1, 1, 1)}$ if and only if 
		$$\Big(\frac{x_1 - z_2x_3}{1 - z_2x_2}, \frac{x_4 - z_2x_6}{1 - z_2x_2}, \frac{x_5 - z_2x_7}{1 - z_2x_2}\Big) \in G_{E(2; 2; 1, 1)} ~\text{for}~ z_2 \in \overline{\mathbb{D}}.$$ 
		
By [Lemma $2.1$ \cite{Pal}], we get $$\Big(\frac{x_1 - z_2x_3}{1 - z_2x_2}, \frac{x_4 - z_2x_6}{1 - z_2x_2}, \frac{x_5 - z_2x_7}{1 - z_2x_2}\Big) \in G_{E(2; 2; 1, 1)} ~\text{for}~ z_2 \in \overline{\mathbb{D}}$$ if and only if $$\Big(\frac{(\omega x_1) - z_2(\omega x_3)}{1 - z_2x_2}, \frac{(\omega x_4) - z_2(\omega x_6)}{1 - z_2x_2}, \frac{(\omega^2 x_5) - z_2(\omega^2 x_7)}{1 - z_2x_2}\Big) \in G_{E(2; 2; 1, 1)} ~\text{for}~ z_2 \in \overline{\mathbb{D}} ~{\rm{and~for~}} \omega \in \mathbb T.$$ Hence, again by Theorem \ref{mainthm}, we deduce  that $\textbf{x} \in G_{E(3; 3; 1, 1, 1)}$ if and only if $(\omega x_1, x_2, \omega x_3, \omega x_4, \omega^2 x_5, \omega x_6, \omega^2 x_7) \in G_{E(3; 3; 1, 1, 1)}$ for $\omega \in \mathbb{T}$. This completes the proof.
	\end{proof}
	
The proof of following lemma is identical to that of Lemma \ref{lem-1}. We therefore omit the proof.
	
	\begin{lem}\label{lem-121}
		Suppose $\textbf{x} = (x_1, \dots, x_7)$ is a point of $\mathbb{C}^7$. The followings are equivalent:
		\begin{enumerate}
			\item $\textbf{x} \in \Gamma_{E(3; 3; 1, 1, 1)}.$ 
			
			\item $(\omega x_1, x_2, \omega x_3, \omega x_4, \omega^2 x_5, \omega x_6, \omega^2 x_7) \in \Gamma_{E(3; 3; 1, 1, 1)}$ for all $\omega \in \mathbb T.$
			
			\item[$(2^{'})$] $ (\omega x_1, \omega x_2, \omega^2 x_3, x_4, \omega x_5, \omega x_6, \omega^2 x_7) \in \Gamma_{E(3; 3; 1, 1, 1)}$ for all $\omega \in \mathbb T.$
			
			\item[$(2^{''})$] $ (x_1, \omega x_2, \omega x_3, \omega x_4, \omega x_5, \omega^2 x_6, \omega^2 x_7) \in \Gamma_{E(3; 3; 1, 1, 1)}$ for all $\omega \in \mathbb T.$
		\end{enumerate}
	\end{lem}

	It is important to note that the defining condition for $\Gamma_{E(3; 3; 1, 1, 1)}$-contractions can be significantly simplified in the following proposition.
	\begin{prop}\label{prop-1}
		Let $\textbf{T} = (T_1, \dots, T_7)$ be a commuting $7$-tuple of bounded operators acting on a Hilbert space $\mathcal{H}$. Then, the following are equivalent:
		\begin{enumerate}
			\item $\textbf{T}$ is a $\Gamma_{E(3; 3; 1, 1, 1)}$-contraction.
			
			\item For any holomorphic polynomial $p$ of seven variables,
			\begin{equation}
				\begin{aligned}
					||p(T_1, \dots, T_7)||
					&\leqslant \sup\{|p(x_1, \dots, x_7)| : (x_1, \dots, x_7) \in \Gamma_{E(3; 3; 1, 1, 1)}\}\\
					&= ||p||_{\infty, \Gamma_{E(3; 3; 1, 1, 1)}}.
				\end{aligned}
			\end{equation}
			
			\item $(\omega T_1,T_2, \omega T_3, \omega T_4, \omega^2 T_5, \omega T_6, \omega^2 T_7)$ is a $\Gamma_{E(3; 3; 1, 1, 1)}$-contraction for all $\omega \in \mathbb{T}.$
			
			\item[$(3^{'})$] $ (\omega T_1, \omega T_2, \omega^2 T_3, T_4, \omega T_5, \omega T_6, \omega^2 T_7)$ is a $\Gamma_{E(3; 3; 1, 1, 1)}$-contraction for all $\omega \in \mathbb{T}.$
			
			\item[$(3^{''})$] $ (T_1, \omega T_2, \omega T_3, \omega T_4, \omega T_5, \omega^2 T_6, \omega^2 T_7)$ is a $\Gamma_{E(3; 3; 1, 1, 1)}$-contraction for all $\omega \in \mathbb{T}.$
		\end{enumerate}
	\end{prop}
	
	\begin{proof}
		We just prove $(1) \Leftrightarrow (2) \Leftrightarrow (3),$ as the proofs for the equivalence $(1) \Leftrightarrow (2^{'}) \Leftrightarrow (3^{'})$ and $(1) \Leftrightarrow (2^{''}) \Leftrightarrow (3^{''})$ are same.
		
		First we show that $(1)$ implies $(2).$  Let $\textbf{T} = (T_1, \dots, T_7)$ be a $\Gamma_{E(3; 3; 1, 1, 1)}$-contraction. Then by definition of $\Gamma_{E(3; 3; 1, 1, 1)}$-contraction, we have
		\begin{equation*}
			\begin{aligned}
				||p(T_1, \dots, T_7)||
				&\leqslant \sup\{|p(x_1, \dots, x_7)| : (x_1, \dots, x_7) \in \Gamma_{E(3; 3; 1, 1, 1)}\}\\
				&= ||p||_{\infty, \Gamma_{E(3; 3; 1, 1, 1)}}.
			\end{aligned}
		\end{equation*}
		
		\vspace{0.2cm}
		
		To establish the implication of $(2) \Rightarrow (1),$ we use the polynomial convexity of $\Gamma_{E(3; 3; 1, 1, 1)}.$ Suppose that the joint spectrum of $(T_1, \dots, T_7)$ is not contained in $\Gamma_{E(3; 3; 1, 1, 1)}. $ Then there exists a point $(y_1, \dots, y_7)$ in $\sigma(T_1, \dots, T_7)$ that is not in $ \Gamma_{E(3; 3; 1, 1, 1)}$.  Since $\Gamma_{E(3; 3; 1, 1, 1)}$ is polynomially convex, then there exists a polynomial $p$ in $7$ variables such that $|p(y_1, \dots, y_7)| > ||p||_{\infty, \Gamma_{E(3; 3; 1, 1, 1)}}.$ By the polynomial spectral mapping theorem, we have
		
		\begin{equation*}
			\begin{aligned}
				\sigma(p(T_1, \dots, T_7)) = p(\sigma(T_1, \dots, T_7)).
			\end{aligned}
		\end{equation*}
		Therefore, we  deduce that the spectral radius of $p(T_1, \dots, T_7)$ is bigger than $\rvert |p\lvert |_{\Gamma_{E(3; 3; 1, 1, 1)}, \infty}.$ 
		This is a contradiction. By \textit{Oka-Weil} theorem [Theorem $5.1$, \cite{Gamelin}] and functional calculus in several commuting operators [Theorem $9.9$, \cite{Vasilescu}], we can extend the result to holomorphic functions $f$ in a neighbourhood of $\Gamma_{E(3; 3; 1, 1, 1)}$. 
		\vspace{0.3cm}
		
		We now prove $(2) \Rightarrow (3). $ Let $\textbf{T}^{'} = (\omega T_1, T_2, \omega T_3, \omega T_4, \omega^2 T_5, \omega T_6, \omega^2 T_7)$. In order to show the implication $(2) \Rightarrow (3)$, by \textit{Oka-Weil} theorem, it suffices to show $\rvert |p(\textbf{T}^{\prime} )\lvert |\leq \rvert| p \lvert|_{\infty,\Gamma_{E(3; 3; 1, 1, 1)}}$ for any holomorphic polynomial $p$ in the co-ordinates of $\Gamma_{E(3; 3; 1, 1, 1)}.$ Let $p$ be any polynomial that is holomorphic in the co-ordinates of $\Gamma_{E(3; 3; 1, 1, 1)}$ and $\omega \in \mathbb T,$ let
		\begin{equation*}
			\begin{aligned}
				p_1(x_1, \dots, x_7)
				&= p(\omega x_1, x_2, \omega x_3, \omega x_4, \omega^2 x_5, \omega x_6, \omega^2 x_7).
			\end{aligned}
		\end{equation*}
		Since $(x_1,\ldots,x_7) \in \Gamma_{E(3; 3; 1, 1, 1)}$, according to the Lemma \ref{lem-1}, it follows that  $$(\omega x_1, x_2, \omega x_3, \omega x_4, \omega^2 x_5, \omega x_6, \omega^2 x_7)\in\Gamma_{E(3; 3; 1, 1, 1)}~{\rm{for ~all}}~\omega \in \mathbb{T}.$$ Note that $\rvert |p\lvert |_{\infty, \Gamma_{E(3; 3; 1, 1, 1)}}= \rvert |p_1\lvert |_{\infty, \Gamma_{E(3; 3; 1, 1, 1)}}.$
		Thus for any holomorphic polynomial $p$ in the co-ordinates of $\Gamma_{E(3; 3; 1, 1, 1)},$ we have
		\begin{equation*}
			\begin{aligned}
				||p(\textbf{T}^{'})||
				= ||p_1(T_1, \dots, T_7)||
				&\leqslant \sup\{|p_1(x_1, \dots, x_7)| : (x_1, \dots, x_7) \in \Gamma_{E(3; 3; 1, 1, 1)}\}\\
				&= \sup\{|p(x_1, \dots, x_7)| : (x_1, \dots, x_7) \in \Gamma_{E(3; 3; 1, 1, 1)}\}\\
				&= ||p||_{\infty, \Gamma_{E(3; 3; 1, 1, 1)}}.
			\end{aligned}
		\end{equation*}
		This shows that  $\textbf{T}^{'} $ is a $\Gamma_{E(3; 3; 1, 1, 1)}$-contraction.
		
		The implication of $(3) \Rightarrow (1)$ can be easily demonstrated by choosing $\omega$ to be equal to $1$. This completes the proof.
	\end{proof}
	
The following lemmas give the 	characterization of $G_{E(3; 2; 1, 2)}$.
\begin{lem}\label{lem-3}
		Let $(x_1, x_2, x_3, y_1, y_2)$ be an element of $G_{E(3; 2; 1, 2)}$. Then $|x_1|<1, |x_3|<1, |y_2| < 1$ and $|x_2|<2, |y_1| < 2$. 
	\end{lem}
	
	\begin{proof}
By [Theorem 3.5, \cite{Bharali}], it follows that $(x_1, x_2, x_3, y_1, y_2) \in G_{E(3; 2; 1, 2)}$ if and only if
		\begin{equation*}
			\Big(\frac{2x_1 - zx_2}{2 - zy_1}, \frac{y_1 - 2zy_2}{2 - zy_1}, \frac{x_2 - 2zx_3}{2 - zy_1}\Big) \in G_{E(2; 2; 1, 1)} ~ \text{for} ~ z \in \overline{\mathbb{D}}.
		\end{equation*}
It follows from the characterization of the tetrablock [Theorem $2.2$, \cite{Abouhajar}] that $\Big(x_1, \frac{y_1}{2}, \frac{x_2}{2}\Big), \Big(\frac{x_2}{2}, \frac{y_1}{2}, x_3\Big)$ and $\Big(\frac{y_1}{2}, \frac{y_1}{2}, y_2\Big)$ are in the tetrablock $G_{E(2; 2; 1, 1)}$ and hence   $|x_1|<1, |x_3|<1, |y_2|< 1$ and $|x_2|<2, |y_1| < 2$.  This completes the proof.
	\end{proof}	
\begin{lem}\label{lemm1}
		Let $\tilde{\textbf{x}} = (x_1, x_2, x_3, y_1, y_2)$ be an element of $\mathbb{C}^5$. Then $\tilde{\textbf{x}} \in G_{E(3; 2; 1; 2)}$ if and only if
		\[(x_1, \omega x_2, \omega^2 x_3, \omega y_1, \omega^2 y_2) \in G_{E(3; 2; 1, 2)} ~\text{for all}~ \omega \in \mathbb{T}.\]
	\end{lem}
	
	\begin{proof}
	
	It follows from Theorem \ref{bhch} that  $\tilde{\textbf{x}} \in G_{E(3; 2; 1, 2)}$ if and only if
		\begin{eqnarray}\label{tetra}
		\Big(\frac{y_1 - zx_2}{1 - zx_1}, \frac{y_2 - zx_3}{1 - zx_1}\Big) \in G_{E(2; 1; 2)} ~\text{for}~ z \in \overline{\mathbb{D}}.
		\end{eqnarray}
		By [Theorem 1.1, \cite{ay1}], we get
		\begin{eqnarray}\label{tetra1}
		\Big(\frac{y_1 - zx_2}{1 - zx_1}, \frac{y_2 - zx_3}{1 - zx_1}\Big) \in G_{E(2; 1; 2)} ~\text{for}~ z \in \overline{\mathbb{D}}\end{eqnarray}
		if and only if
		\begin{eqnarray}\label{tetra2}
		\Big(\frac{(\omega y_1) - z(\omega x_2)}{1 - zx_1}, \frac{(\omega^2y_2) - z(\omega^2x_3)}{1 - zx_1}\Big) \in G_{E(2; 1; 2)} ~\text{for}~ z \in \overline{\mathbb{D}} ~\text{and for all}~ \omega \in \mathbb{T}.\end{eqnarray}
From \eqref{tetra}, \eqref{tetra1} and \eqref{tetra2}, we deduce  that $\tilde{\textbf{x}} \in G_{E(3; 2; 1, 2)}$ if and only if $(x_1, \omega x_2, \omega^2 x_3, \omega y_1, \omega^2 y_2) \in G_{E(3; 2; 1, 2)} ~\text{for all}~ \omega \in \mathbb{T}$.
This completes the proof.
\end{proof}

	We state a result for $\Gamma_{E(3; 2; 1, 2)},$ which is similar to the Lemma \ref{lemm1}. The proof is similar to the Lemma \ref{lemm1}. Therefore, we omit the proof.

	\begin{lem}
		Let $\tilde{\textbf{x}} = (x_1, x_2, x_3, y_1, y_2)$ be an element of $\mathbb{C}^5$. Then $\tilde{\textbf{x}} \in \Gamma_{E(3; 2; 1; 2)}$ if and only if
		\[(x_1, \omega x_2, \omega^2 x_3, \omega y_1, \omega^2 y_2) \in \Gamma_{E(3; 2; 1, 2)} ~\text{for all}~ \omega \in \mathbb{T}.\]
	\end{lem}
Similar to the Proposition \ref{prop-1}, the defining condition for $\Gamma_{E(3;2; 1, 2)}$-contractions is greatly simplified in the following proposition. The proof is identical to that of Proposition \ref{prop-1}. Thus, we omit the proof.
	\begin{prop}\label{prop-2}
		Let $\textbf{S} = (S_1, S_2, S_3, \tilde{S}_1, \tilde{S}_2)$ be a $5$-tuple of commuting bounded operators on a Hilbert space $\mathcal{H}$. Then the following are equivalent:
		\begin{enumerate}
			\item $\textbf{S}$ is a $\Gamma_{E(3; 2; 1, 2)}$-contraction.
			
			\item For any holomorphic polynomial $p$,
			\begin{equation}
				\begin{aligned}
					||p(S_1, S_2, S_3, \tilde{S}_1, \tilde{S}_2))||
					&\leqslant \sup\{|p(x_1, x_2, x_3, y_1, y_2))| : (x_1, x_2, x_3, y_1, y_2) \in \Gamma_{E(3; 2; 1, 2)}\}\\
					&= ||p||_{\infty, \Gamma_{E(3; 2; 1, 2)}}.
				\end{aligned}
			\end{equation}
			
			\item $(S_1, \omega S_2, \omega^2S_3, \omega \tilde{S}_1, \omega^2 \tilde{S}_2)$ is a $\Gamma_{E(3; 2; 1, 2)}$-contraction for all $\omega \in \mathbb{T}$.
		\end{enumerate}
	\end{prop}
The relationship between $\Gamma_{E(3; 2; 1, 2)}$-contraction and $\Gamma_{E(3; 3; 1, 1, 1)}$-contraction is given by the following propositions. \begin{prop}\label{prop-3}
Let $(T_1, \dots, T_7)$ be a $\Gamma_{E(3; 3; 1, 1, 1)}$-contraction. Then
		\[\{(T_1, T_3 + \eta T_5, \eta T_7, T_2 + \eta T_4, \eta T_6) : \eta \in \overline{\mathbb{D}}\}\]
		is a family of $\Gamma_{E(3; 2; 1, 2)}$-contraction.
	\end{prop}
	
	\begin{proof}
Observe that a point $(x_1, \dots, x_7) \in \Gamma_{E(3; 3; 1, 1, 1)}$ if and only if $(x_1, x_3 + \eta x_5, \eta x_7, x_2 + \eta x_4, \eta x_6) \in \Gamma_{E(3; 2; 1, 2)}$ for all $\eta \in \overline{\mathbb{D}}$. For $\eta \in \overline{\mathbb{D}}$, we define the map $\pi_{\eta} : \mathbb{C}^7 \to \mathbb{C}^5$ by
		\begin{equation*}
			\begin{aligned}
				\pi_{\eta}(x_1, \dots, x_7)
				&= (x_1, x_3 + \eta x_5, \eta x_7, x_2 + \eta x_4, \eta x_6).
			\end{aligned}
		\end{equation*}
We note that for any $p \in \mathbb{C}[z_1, \dots, z_5]$, we have $p \circ \pi_{\eta} \in \mathbb{C}[z_1, \dots, z_7]$, therefore by hypothesis, we deduce that \begin{equation*}
\begin{aligned}
||p(T_1, T_3 + \eta T_5, \eta T_7, T_2 + \eta T_4, \eta T_6)|| &= ||p \circ \pi_{\eta}(T_1, \dots, T_7)||\\
&\leqslant ||p \circ \pi_{\eta}||_{\infty, \Gamma_{E(3; 3; 1, 1, 1)}}\\
&= ||p||_{\infty, \pi_{\eta}(\Gamma_{E(3; 3; 1, 1, 1)})}\\
&\leqslant ||p||_{\infty, \Gamma_{E(3; 2; 1, 2)}}.
\end{aligned}
\end{equation*}
This completes the proof.
	\end{proof}
\begin{prop}\label{props-1}
Let $(T_1, \dots, T_7)$ be a $\Gamma_{E(3; 3; 1, 1, 1)}$-contraction. Then $(T_1,T_6,T_7), (T_2,T_5,T_7)$ and $(T_3,T_4,T_7)$ are 
$\Gamma_{E(2; 2; 1,1)}$-contractions.
\end{prop}
\begin{proof}
Define the map $\pi_1:\mathbb C^7 \to \mathbb C^3$ by 
$$\pi_1(x_1, \dots, x_7)=(x_1,x_6,x_7).$$
We notice that for any $p \in \mathbb{C}[z_1, z_6,z_7]$, we have $p \circ \pi_1 \in \mathbb{C}[z_1, \dots, z_7]$, therefore by hypothesis, we have \begin{equation*}
\begin{aligned}
||p(T_1, T_6,T_7)|| &= ||p \circ \pi_{1}(T_1, \dots, T_7)||\\
&\leqslant ||p \circ \pi_{1}||_{\infty, \Gamma_{E(3; 3; 1, 1, 1)}}\\
&= ||p||_{\infty, \pi_{1}(\Gamma_{E(3; 3; 1, 1, 1)})}\\
&\leqslant ||p||_{\infty, \Gamma_{E(2; 2;1,1)}}.
\end{aligned}
\end{equation*}
This show that $(T_1,T_6,T_7)$ is a $\Gamma_{E(2; 2; 1,1)}$-contraction. Similarly, we can also show that $(T_2,T_5,T_7)$ and $(T_3,T_4,T_7)$ are 
$\Gamma_{E(2; 2; 1,1)}$-contractions.
This completes the proof.
\end{proof}
\begin{prop}\label{prop-4}
Suppose that  $(S_1, S_2, S_3, \tilde{S_1}, \tilde{S_2})$ is a $\Gamma_{E(3; 2; 1, 2)}$-contraction. Then $(S_1, \frac{\tilde{S}_1}{2}, \frac{S_2}{2}, \frac{\tilde{S}_1}{2}, \frac{S_2}{2}, \tilde{S}_2, S_3)$ is a $\Gamma_{E(3; 3; 1, 1, 1)}$-contraction.
	\end{prop}
	
	\begin{proof}
		In order to show $\left(S_1, \frac{\tilde{S}_1}{2}, \frac{S_2}{2}, \frac{\tilde{S}_1}{2}, \frac{S_2}{2}, \tilde{S}_2, S_3\right)$ is a $\Gamma_{E(3; 3; 1, 1, 1)}$-contraction, consider the map $\pi : \mathbb{C}^5 \to \mathbb{C}^7$ defined by
		\begin{equation*}
			\begin{aligned}
				\pi(x_1, x_2, x_3, y_1, y_2)
				&= \Big(x_1, \frac{y_1}{2}, \frac{x_2}{2}, \frac{y_1}{2}, \frac{x_2}{2}, y_2, x_3\Big).
			\end{aligned}
		\end{equation*}
We first demonstrate that $\pi( \Gamma_{E(3; 2; 1, 2)})\subseteq \Gamma_{E(3; 3; 1, 1, 1)}$. It follows from  [\cite{Bharali},Theorem 3.5] that $(x_1, x_2, x_3, y_1, y_2) $ belongs to $ G_{E(3; 2; 1, 2)}$ if and only if
\begin{equation*}
\Big(\frac{2x_1 - zx_2}{2 - zy_1}, \frac{y_1 - 2zy_2}{2 - zy_1}, \frac{x_2 - 2zx_3}{2 - zy_1}\Big) \in \Gamma_{E(2; 2; 1, 1)} ~ \text{for} ~ z \in \mathbb{D}.
\end{equation*}
By Theorem \ref{mainthm}, we deduce that $\Big(x_1, \frac{y_1}{2}, \frac{x_2}{2}, \frac{y_1}{2}, \frac{x_2}{2}, y_2, x_3\Big)\in 
\Gamma_{E(3; 3; 1, 1, 1)}$. This shows that $\pi( \Gamma_{E(3; 2; 1, 2)})\subseteq \Gamma_{E(3; 3; 1, 1, 1)}$. For any $p \in \mathbb{C}[z_1, \dots, z_7],$ clearly, $p \circ \pi \in \mathbb{C}[z_1, \dots, z_5]$. We have 
		\begin{equation*}
			\begin{aligned}
				||p(S_1, \frac{\tilde{S}_1}{2}, \frac{S_2}{2}, \frac{\tilde{S}_1}{2}, \frac{S_2}{2}, \tilde{S}_2, S_3)||
				&= ||p \circ \pi(S_1, S_2, S_3, \tilde{S_1}, \tilde{S_2})||\\
				&\leqslant ||p \circ \pi||_{\infty, \Gamma_{E(3; 2; 1, 2)}}\\
				&= ||p||_{\infty, \pi(\Gamma_{E(3; 2; 1, 2)})}\\
				&\leqslant ||p||_{\infty, \Gamma_{E(3; 3; 1, 1, 1)}}.
			\end{aligned}
		\end{equation*}
This completes the proof.
	\end{proof}
	
\begin{prop}\label{prop-41}
Let  $(S_1, S_2, S_3, \tilde{S_1}, \tilde{S_2})$ be a $\Gamma_{E(3; 2; 1, 2)}$-contraction. Then $(S_1, \tilde{S}_2, S_3), ( \frac{\tilde{S}_1}{2},\frac{S_2}{2},  S_3)$ and $(\frac{S_2}{2}, \frac{\tilde{S}_1}{2}, S_3)$ are $\Gamma_{E(2; 2; 1,1)}$-contractions.

\end{prop}

\begin{proof}
By Proposition \ref{props-1} and Proposition \ref{prop-4}, we deduce that  $(S_1, \tilde{S}_2, S_3), ( \frac{\tilde{S}_1}{2},\frac{S_2}{2},  S_3)$ and $(\frac{S_2}{2}, \frac{\tilde{S}_1}{2}, S_3)$ are $\Gamma_{E(2; 2; 1,1)}$-contractions. This completes the proof.

\end{proof}

The following give the relationship between $\Gamma_{E(3; 2; 1, 2)}$-contractions and $\Gamma_{E(2; 1; 2)}$-contractions.
	
	\begin{prop}\label{prop-5}
Let $(S_1, S_2, S_3, \tilde{S}_1, \tilde{S}_2)$ be a $\Gamma_{E(3; 2; 1, 2)}$-contraction. Then
		\[\{((\tilde{S}_1 - zS_2)(I - zS_1)^{-1}, (\tilde{S}_2 - zS_3)(I - zS_1)^{-1}) : z \in \mathbb{D}\}\]
		is a family of $\Gamma_{E(2; 1; 2)}$-contractions.
	\end{prop}
	
	\begin{proof}
For any $z\in \mathbb D,$ consider the map $\pi_z : \Gamma_{E(3; 2; 1, 2)} \to \mathbb{C}^2$ defined by
		\begin{equation*}
			\begin{aligned}
				\pi_z(x_1, x_2, x_3, y_1, y_2) &= \Big(\frac{y_1 - zx_2}{1 - zx_1}, \frac{y_2 - zx_3}{1 - zx_1}\Big).
			\end{aligned}
		\end{equation*}
For any $p \in \mathbb{C}[z_1, z_2]$ we observe that  $p \circ \pi_z $ is a rational function defined on $\Gamma_{E(3; 2; 1, 2)}$. By proposition \ref{bhch1}, it follows that  $\pi_z (\Gamma_{E(3; 2; 1, 2)})\subseteq \Gamma_{E(2; 1; 2)}$. Hence by hypothesis we conclude that
		\begin{equation*}
			\begin{aligned}
				||p((\tilde{S}_1 - zS_2)(I - zS_1)^{-1}, (\tilde{S}_2 - zS_3)(I - zS_1)^{-1}))|| &= ||p \circ \pi_z(S_1, S_2, S_3, \tilde{S}_1, \tilde{S}_2)||\\ &\leqslant ||p \circ \pi_z||_{\infty, \Gamma_{E(3; 2; 1, 2)}}\\
				&= ||p||_{\infty, \pi_z(\Gamma_{E(3; 2; 1, 2)})}\\
				&\leqslant ||p||_{\infty, \Gamma_{E(2; 1; 2)}}.
			\end{aligned}
		\end{equation*}
		This completes the proof.
	\end{proof}
	
	\begin{rem}\label{rem-1}
\begin{enumerate}		
\item Let $\mathcal{Q}$ be a joint invariant subspace of $\textbf{T} = (T_1, \dots, T_7).$ Then for any polynomial any $p \in \mathbb{C}[z_1, \dots, z_7],$ we notice that
		\begin{equation*}
			\begin{aligned}
				||p(T_1|_{\mathcal{Q}}, \dots, T_7|_{\mathcal{Q}})||
				&= ||p(T_1, \dots, T_7)|_{\mathcal{Q}}||\\
				&\leqslant ||p(T_1, \dots, T_7)||\\
				&\leqslant ||p||_{\infty, \Gamma_{E(3; 3; 1, 1, 1)}}.
			\end{aligned}
		\end{equation*}
This indicates that a $\Gamma_{E(3; 3; 1, 1, 1)}$-contraction remains a  $\Gamma_{E(3; 3; 1, 1, 1)}$-contraction when it is restricted to a joint invariant subspace.  Similarily, a $\Gamma_{E(3; 2; 1, 2)}$-contractions remains a $\Gamma_{E(3; 2; 1, 2)}$-contractions when it is restricted to a joint invariant subspace.

\item From Proposition \ref{prop-1}, it follows that the adjoint $\textbf{T}^* = (T_1^*, \dots, T_7^*)$ is a $\Gamma_{E(3; 3; 1, 1, 1)}$-contraction. Similarly, from Proposition \ref{prop-2}, we deduce that $\textbf{S}^*= (S_1^*, S_2^*, S_3^*, \tilde{S}_1^*, \tilde{S}_2^*)$ is a $\Gamma_{E(3; 2; 1, 2)}$-contraction.

\end{enumerate}

	\end{rem}
	
	\subsection{Various properties of $\Gamma_{E(3; 3; 1, 1, 1)}$-Contractions and $\Gamma_{E(3; 2; 1, 2}$-Contractions}

Let $T$ be a contraction on a Hilbert space $\mathcal H.$ Define the defect operator $D_{T}=(I-T^*T)^{\frac{1}{2}}$ associated with $T$. The closure of the range of $ D_{T}$ is denoted by $\mathcal D_{T}$.  In this subsection, we discuss the fundamental equations for $\Gamma_{E(3; 3; 1, 1,1 )}$-contractions and $\Gamma_{E(3; 2; 1, 2)}$-contractions. The definition of operator functions $\rho_{G_{E(2; 1; 2)}} $ and $\rho_{G_{E(2; 2; 1,1)}} $ for symmetrized bidisc and tetrablock are defined as follows:
$$\rho_{G_{E(2; 1; 2)}} (S,P)=2(I-P^*P)-(S-S^*P)-(S^*-P^*S)$$ and 
$$\rho_{G_{E(2; 2; 1,1)}} (T_1,T_2,T_3)=(I-T_3^*T_3)-(T_2^*T_2-T_1^*T_1)-2\Re {T_2-T_1^*T_3}$$  where $P,T_3$ are contractions and $S,P$ and $T_1,T_2,T_3$ are commuting bounded operators defined on Hilbert spaces $\mathcal H_1$ and $\mathcal H_2$, respectively. The above operator functions are crucial for characterizing $\Gamma_{E(2; 1; 2)}$-contraction and $\Gamma_{E(2; 2; 1,1)}$-contraction, respectively. Similar to the operator functions $\rho_{G_{E(2; 1; 2)}}$ and $\rho_{G_{E(2; 2; 1,1)}} $, we now introduce three operator functions for a $7$ commuting bounded operators $T_1, \dots, T_7$, with $\|T_7\|\leq 1$ , which are crucial for characterization of the $\Gamma_{E(3; 3; 1, 1, 1)}$-contractions:
\begin{equation}
\begin{aligned}
				\rho^{(1)}_{G_{E(3; 3; 1, 1, 1)}} (T_1, \dots, T_7)
				&= D^2_{T_7} + (T^*_2T_2 + T^*_4T_4 + T^*_6T_6 - T^*_1T_1 - T^*_3T_3 - T^*_5T_5)\\
				&\hspace{0.4cm} - 2Re ~ (T_2 - T^*_1T_3) - 2Re ~ (T_4 - T^*_1T_5) + 2Re ~ (T_6 - T^*_1T_7)\\
				&\hspace{0.4cm} - 2Re ~ (T^*_4T_6 - T^*_5T_7) - 2Re ~ (T^*_2T_6 - T^*_3T_7) -2Re ~ (T^*_5T_3 - T^*_4T_2), 
			\end{aligned}
		\end{equation}
		
		\begin{equation}
			\begin{aligned}
				\rho^{(2)}_{G_{E(3; 3; 1, 1, 1)}} (T_1, \dots, T_7)
				&= D^2_{T_7} + (T^*_1T_1 + T^*_4T_4 + T^*_5T_5 - T^*_2T_2 - T^*_3T_3 - T^*_6T_6)\\
				&\hspace{0.4cm} - 2Re ~ (T_1 - T^*_2T_3) - 2Re ~ (T_4 - T^*_2T_6) + 2Re ~ (T_5 - T^*_2T_7)\\ 
				&\hspace{0.4cm} -2Re ~ (T^*_4T_5 - T^*_6T_7) - 2Re ~ (T^*_1T_5 - T^*_3T_7) -2Re ~ (T^*_6T_3 - T^*_4T_1)
			\end{aligned}
		\end{equation}
and		
		\begin{equation}
			\begin{aligned}
				\rho^{(3)}_{G_{E(3; 3; 1, 1, 1)}} (T_1, \dots, T_7)
				&= D^2_{T_7} + (T^*_1T_1 + T^*_2T_2 + T^*_3T_3 - T^*_4T_4 - T^*_5T_5 - T^*_6T_6)\\
				&\hspace{0.4cm} - 2Re ~ (T_1 - T^*_4T_5) - 2Re ~ (T_2 - T^*_4T_6) + 2Re ~ (T_3 - T^*_4T_7)\\
				&\hspace{0.4cm} -2Re ~ (T^*_2T_3 - T^*_6T_7) - 2Re ~ (T^*_1T_3 - T^*_5T_7) -2Re ~ (T^*_6T_5 - T^*_2T_1). 
			\end{aligned}
		\end{equation}
Let us assume that $T_1, T_2$ and $T_4$ are contractions. We introduce three triples of bounded operators as follows:
		\begin{equation*}
			\begin{aligned}
				(A_{z_1}, B_{z_1}, P_{z_1})
				&= \{ ((T_2 - z_1T_3)(I - z_1T_1)^{-1}, (T_4 - z_1T_5)(I - z_1T_1)^{-1},\\
				&~~~~ \,\,\,\,\,\,\,\,(T_6 - z_1T_7)(I - z_1T_1)^{-1}) : z_1 \in \mathbb{D}\},
			\end{aligned}
		\end{equation*}
		
		\begin{equation*}
			\begin{aligned}
				(A_{z_2}, B_{z_2}, P_{z_2})
				&= \{ ((T_1 - z_2T_3)(I - z_2T_2)^{-1}, (T_4 - z_2T_6)(I - z_2T_2)^{-1},\\
				&~~~~ \,\,\,\,\,\,\,\,(T_5 - z_2T_7)(I - z_2T_2)^{-1}) : z_2 \in \mathbb{D} \}
			\end{aligned}
		\end{equation*}
and		
		\begin{equation*}
			\begin{aligned}
				(A_{z_3}, B_{z_3}, P_{z_3})
				&= \{ ((T_1 - z_3T_5)(I - z_3T_4)^{-1}, (T_2 - z_3T_6)(I - z_3T_4)^{-1},\\
				&~~~~ \,\,\,\,\,\,\,\, (T_3 - z_3T_7)(I - z_3T_4)^{-1}) : z_3 \in \mathbb{D} \}.
			\end{aligned}
		\end{equation*}

	\begin{defn}\label{fundamental}
		Let $(T_1, \dots, T_7)$ be a $7$-tuples of commuting contractions on a Hilbert space $\mathcal{H}$ and $\{(A_{z_i}, B_{z_i}, P_{z_i}) : z_i \in \mathbb{D}\}$ are as above with $\|P_{z_i}\|\leq 1$ for $1\leq i\leq 3$. For fixed but arbitrary $z_i\in \mathbb D$, the equations
		\begin{equation}\label{Fundamental 1}
			\begin{aligned}
				&A_{z_i} - B^*_{z_i}P_{z_i} = D_{P_{z_i}}F^{(1)}_{z_i}D_{P_{z_i}}, B_{z_i} - A^*_{z_i}P_{z_i} = D_{P_{z_i}}F^{(2)}_{z_i}D_{P_{z_i}},
			\end{aligned}
		\end{equation}
where $F_{z_i}^{(j)}\in \mathcal{B}(\mathcal{D}_{P_{z_i}}), 1\leq j\leq2$, are called the fundamental equations for $(A_{z_i}, B_{z_i},P_{z_i}),\; 1\leq i\leq 3$.
			\end{defn}
We recall the definition of tetrablock contraction from \cite{Bhattacharyya}.
\begin{defn}
Let $(A,B,P)$ be a commuting triple of bounded operators on a Hilbert space $\mathcal H.$ We say that $(A,B,P)$  is a tetrablock contraction if $\Gamma_{E(2;2;1,1)}$ is a spectral set for $(A,B,P)$ .
\end{defn}
	
\begin{thm}\label{thm-1}
		Let $\textbf{T}=(T_1, \dots, T_7)$ denote a $7$-tuple of  commuting contractions on some Hilbert space $\mathcal{H}$. Then the following statements are valid
 $(1) \Rightarrow (2) \Rightarrow (3) \Rightarrow (4) \Rightarrow (5), (1) \Rightarrow (2^{'}) \Rightarrow (3^{'}) \Rightarrow (4^{'}) \Rightarrow (5^{'})$ and $(1) \Rightarrow (2^{''}) \Rightarrow
		(3^{''}) \Rightarrow (4^{''}) \Rightarrow (5^{''})$:
		\begin{enumerate}
			\item $(T_1, \dots, T_7)$ is a $\Gamma_{E(3; 3; 1, 1, 1)}$-contraction.
			
			\item $\{ (A_{z_2}, B_{z_2}, P_{z_2}) : z_2 \in \mathbb{D} \}$ is a family of tetrablock contractions.
			
			\item[($2^{'}$)] $\{ (A_{z_3}, B_{z_3}, P_{z_3}) : z_3 \in \mathbb{D} \}$ is a family of tetrablock contractions.
			
			\item[($2^{''}$)] $\{ (A_{z_1}, B_{z_1}, P_{z_1}) : z_1 \in \mathbb{D} \}$ is a family of tetrablock contractions.
			
			\item For all $\omega \in \mathbb{T}$ and $z_2 \in \mathbb{D},$
			\begin{equation*}
				\begin{aligned}
					\rho_{G_{E(2; 2; 1, 1)}}(A_{z_2}, \omega B_{z_2}, \omega P_{z_2}) &\geqslant 0,
					\rho_{G_{E(2; 2; 1, 1)}}(B_{z_2}, \omega A_{z_2}, \omega P_{z_2}) \geqslant 0.
				\end{aligned}
			\end{equation*}
and the spectral radius of 	$S^{(z_2)}$ is no bigger than $2.$	
			\item[($3^{'}$)] For all $\omega \in \mathbb{T}$ and $z_3 \in \mathbb{D},$
			\begin{equation*}
				\begin{aligned}
					\rho_{G_{E(2; 2; 1, 1)}}(A_{z_3}, \omega B_{z_3}, \omega P_{z_3}) &\geqslant 0,
					\rho_{G_{E(2; 2; 1, 1)}}(B_{z_3}, \omega A_{z_3}, \omega P_{z_3}) \geqslant 0.
				\end{aligned}
			\end{equation*}
and the spectral radius of 	$S^{(z_3)}$ is no bigger than $2.$			
			\item[($3^{''}$)] For all $\omega \in \mathbb{T}$ and $z_1 \in \mathbb{D},$
			\begin{equation*}
				\begin{aligned}
					\rho_{G_{E(2; 2; 1, 1)}}(A_{z_1}, \omega B_{z_1}, \omega P_{z_1}) &\geqslant 0,
					\rho_{G_{E(2; 2; 1, 1)}}(B_{z_1}, \omega A_{z_1}, \omega P_{z_1}) \geqslant 0.
				\end{aligned}
			\end{equation*}
and the spectral radius of 	$S^{(z_1)}$ is no bigger than $2.$			
			\item For each fixed but arbitrary $z_2 \in \mathbb{D}$, the pair $(S^{(z_2)}(\omega), P^{(z_2)}(\omega)) = (A_{z_2} + \omega B_{z_2}, \omega P_{z_2})$ is a $\Gamma_{E(2; 1; 2)}$-contraction for every $\omega \in \mathbb{T}$.
			
			\item[($4^{'}$)] For each fixed but arbitrary $z_3 \in \mathbb{D}$, the pair $(S^{(z_3)}(\omega), P^{(z_3)}(\omega)) = (A_{z_3} + \omega B_{z_3}, \omega P_{z_3})$ is a $\Gamma_{E(2; 1; 2)}$-contraction for every $\omega \in \mathbb{T}$.
			
			\item[($4^{''}$)]  For each fixed but arbitrary $z_1 \in \mathbb{D}$, the pair $(S^{(z_1)}(\omega), P^{(z_1)}(\omega)) = (A_{z_1} + \omega B_{z_1}, \omega P_{z_1})$is a $\Gamma_{E(2; 1; 2)}$-contraction for every $\omega \in \mathbb{T}$.
			
			\item For each fixed but arbitrary  $z_2 \in \mathbb{D}$ the  fundamental equations  (\ref{Fundamental 1}) possess unique solutions $F^{(1)}_{z_2}, F^{(2)}_{z_2}$ in $\mathcal{B}(\mathcal{D}_{P_{z_2}})$. Furthermore, the numerical radius of the operator $F^{(1)}_{z_2} + w_1F^{(2)}_{z_2}$ is less than or equal to $1$ for every $w_1 \in \mathbb{T}$.
			
			\item[($5^{'}$)] For each fixed but arbitrary  $z_3 \in \mathbb{D}$ the  fundamental equations  (\ref{Fundamental 1}) possess unique solutions $F^{(1)}_{z_3}, F^{(2)}_{z_3}$ in $\mathcal{B}(\mathcal{D}_{P_{z_3}})$. Furthermore, the numerical radius of the operator $F^{(1)}_{z_3} + w_2F^{(2)}_{z_3}$ is less than or equal to $1$ for every $w_2 \in \mathbb{T}$.
			
			\item[($5^{''}$)] For each fixed but arbitrary  $z_1 \in \mathbb{D}$ the  fundamental equations  (\ref{Fundamental 1}) possess unique solutions $F^{(1)}_{z_1}, F^{(2)}_{z_1}$ in $\mathcal{B}(\mathcal{D}_{P_{z_1}})$. Furthermore, the numerical radius of the operator $F^{(1)}_{z_1} + w_3F^{(2)}_{z_1}$ is less than or equal to $1$ for every $w_3 \in \mathbb{T}$.
		\end{enumerate}
	\end{thm}
	
	\begin{proof}
		We demonstrate the implications $(1) \Rightarrow (2) \Rightarrow (3) \Rightarrow (4) \Rightarrow (5)$, as the proofs for the remaining two implications are analogous. We first prove $(1)$ implies $(2).$  Let $\textbf{T}=(T_1, \dots, T_7)$ be a $\Gamma_{E(3; 3; 1, 1, 1)}$-contraction. 
%
For each fixed but arbitrary $z_2 \in \mathbb{D}$ we define the map $\pi_{z_2} : \Gamma_{E(3; 3; 1, 1, 1)} \to \mathbb{C}^3$ as follows:
		\begin{equation*}
			\begin{aligned}
				\pi_{z_2}(x_1, \dots, x_7) &= \Big(\frac{x_1 - z_2x_3}{1 - z_2x_2}, \frac{x_4 - z_2x_6}{1 - z_2x_2}, \frac{x_5 - z_2x_7}{1 - z_2x_2}\Big).
			\end{aligned}
		\end{equation*}

		Note that $p \circ \pi_{z_2}$ is a rational function defined on $\Gamma_{E(3; 3; 1, 1, 1)} $ for any $p \in \mathbb{C}[z_1, z_2, z_3].$ Thus we have 		\begin{equation*}
			\begin{aligned}
				||p(A_{z_2}, B_{z_2}, P_{z_2})||
				&= ||p \circ \pi_{z_2} (T_1, \dots, T_7)||\\
				&\leqslant ||p \circ \pi_{z_2}||_{\infty, \Gamma_{E(3; 3; 1, 1, 1)}}\\
				&= ||p||_{\infty, \pi_{z_2}(\Gamma_{E(3; 3; 1, 1, 1)})}\\
				&\leqslant ||p||_{\infty, \Gamma_{E(2; 2; 1, 1)}}.
			\end{aligned}
		\end{equation*}
This shows that  $\{(A_{z_2}, B_{z_2}, P_{z_2}) : z_2 \in \mathbb{D}\}$ is a family of tetrablock contractions.
%

We can derive the implication $(2) \Rightarrow (3) \Rightarrow (4) \Rightarrow (5)$ from [Theorem $3.4$, \cite{Bhattacharyya}]. 
This completes the proof.	\end{proof}
\begin{rem}\label{remm1}
Let $T=(T_1, \dots, T_7)$ be a $\Gamma_{E(3; 3; 1, 1, 1)}$-contraction and $p_i$ be a polynomial in $\mathbb{C}[x_1,\ldots,x_7]$ which is defined as $p_i(x_1,x_2,\ldots,x_7)=x_i$ for $1\leq i\leq 7$. By  Lemma \ref{lemm-2} and Proposition \ref{prop-1}, it follows that $\|T_i\|=\|p_i(T)\|\leq \|p_i\|_{\infty, \Gamma_{E(3; 3; 1, 1, 1)}}= \|x_i\|_{\infty, \Gamma_{E(3; 3; 1, 1, 1)}}\leq 1$. For $i=1,2,4$, if $\|T_i\|<1$, then $(I-z_iT_i)$ is invertible for all $z_i\in \bar{\mathbb{D}}$.  By using the similar argument as in the Theorem \ref{thm-1}, it yields that if $(T_1, \dots, T_7)$ is a $\Gamma_{E(3; 3; 1, 1, 1)}$-contraction, then  $\{ (A_{z_2}, B_{z_2}, P_{z_2}) : z_2 \in \mathbb{\bar{D}} \}$ is a family of tetrablock contractions. Similarly, it can be demonstrated that if  $(T_1, \dots, T_7)$ is a $\Gamma_{E(3; 3; 1, 1, 1)}$-contraction, then  $\{ (A_{z_1}, B_{z_1}, P_{z_1}) : z_1 \in \mathbb{\bar{D}} \}$ as well as  $\{ (A_{z_3}, B_{z_3}, P_{z_3}) : z_3 \in \mathbb{\bar{D}} \}$ is a family of tetrablock contractions.
\end{rem}	
	\begin{thm}\label{thm-2}
		Let $\textbf{T}=(T_1, \dots, T_7)$ be a $7$-tuple of commuting contractions on some Hilbert space $\mathcal{H}$. Then what follows is true:
		$(1) \Rightarrow (2), (1) \Rightarrow (2'),
		(1) \Rightarrow (2'')$:		\begin{enumerate}
			\item $(T_1, \dots, T_7)$ is a $\Gamma_{E(3; 3; 1, 1, 1)}$-contraction.
			
			\item  For all $z_2, z_3 \in \overline{\mathbb{D}}$,
			\begin{equation*}
				\begin{aligned}
					\rho^{(1)}_{G_{(3; 3; 1, 1, 1)}} (T_1, z_2T_2, z_2T_3, z_3T_4, z_3T_5, z_2z_3T_6, z_2z_3T_7) &\geqslant 0.
				\end{aligned}	
			\end{equation*}
			
			\item[($2^{'}$)] For all $z_1, z_3 \in \overline{\mathbb{D}}$.
			\begin{equation*}
				\begin{aligned}
					\rho^{(2)}_{G_{(3; 3; 1, 1, 1)}} (z_1T_1, T_2, z_1T_3, z_3T_4, z_1z_3T_5, z_3T_6, z_1z_3T_7) &\geqslant 0.
				\end{aligned}
			\end{equation*}
			
			\item[($2^{''}$)] For all $z_1, z_2 \in \overline{\mathbb{D}}$,
			\begin{equation*}
				\begin{aligned}
					\rho^{(3)}_{G_{(3; 3; 1, 1, 1)}} (z_1T_1, z_2T_2, z_1z_2T_3, T_4, z_1T_5, z_2T_6, z_1z_2T_7) &\geqslant 0.
				\end{aligned}
			\end{equation*}
		\end{enumerate}
	
	\end{thm}
	
	\begin{proof}
We only prove $(1) \Rightarrow (2)$, as the proofs for $(1) \Rightarrow (2')$ and $(1) \Rightarrow (2'')$ are similar. Suppose that $(T_1, \dots, T_7)$ is a $\Gamma_{E(3; 3; 1, 1, 1)}$-contraction. Consider the function $\Psi^{(1)}$ as defined in \eqref{psi11}. For $z_2,z_3 \in \mathbb{D},$ the function $\Psi^{(1)}(z_2,z_3,\cdot)$ is holomorphic function on $G_{E(3; 3; 1, 1, 1)}$ with
$$\Psi^{(1)}(z_2, z_3, (T_1, \dots, T_7))=(T_1 - z_2T_3 - z_3T_5 + z_2z_3T_7)
				(I - z_2T_2 - z_3T_4 + z_2z_3T_6)^{-1}.$$
As  $\Gamma_{E(3; 3; 1, 1, 1)}$ is a spectral set for $\textbf{T}, $ we have $||\Psi^{(1)}(z_2, z_3, (T_1, \dots, T_7))|| \leqslant 1$ which implies that \begin{equation*}
			\begin{aligned}
				&(I - \overline{z}_2T^*_2 - \overline{z}_3T^*_4 + \overline{z}_2\overline{z}_3T^*_6)^{-1}
				(T^*_1 - \overline{z}_2T^*_3 - \overline{z}_3T^*_5 + \overline{z}_2\overline{z}_3T^*_7)\\
				&(T_1 - z_2T_3 - z_3T_5 + z_2z_3T_7)
				(I - z_2T_2 - z_3T_4 + z_2z_3T_6)^{-1}
				\leqslant I.
			\end{aligned}
		\end{equation*}
		This corresponds  to
		\begin{equation*}
			\begin{aligned}
				&(T^*_1 - \overline{z}_2T^*_3 - \overline{z}_3T^*_5 + \overline{z}_2\overline{z}_3T^*_7)
				(T_1 - z_2T_3 - z_3T_5 + z_2z_3T_7)\\
				&\leqslant (I - \overline{z}_2T^*_2 - \overline{z}_3T^*_4 + \overline{z}_2\overline{z}_3T^*_6)
				(I - z_2T_2 - z_3T_4 + z_2z_3T_6)
			\end{aligned}
		\end{equation*}
		and cosequently, we obtain
		\begin{equation*}
			\begin{aligned}
				&(I - |z_2z_3|^2T^*_7T_7) + (|z_2|^2T^*_2T_2 + |z_3|^2T^*_4T_4 + |z_2z_3|^2T^*_6T_6 - T^*_1T_1
				- |z_2|^2T^*_3T_3 - |z_3|^2T^*_5T_5)\\
				&- 2Re\, z_2(T_2 - T^*_1T_3)
				- 2Re\, z_3(T_4 - T^*_1T_5)
				+ 2Re\, z_2z_3(T_6 - T^*_1T_7) - 2Re\, z_2|z_3|^2(T^*_4T_6 - T^*_5T_7)\\
				&- 2Re\, z_3|z_2|^2(T^*_2T_6 - T^*_3T_7) - 2Re\, z_2\overline{z}_3(T^*_4T_2 - T^*_5T_3)
				\geqslant 0,
			\end{aligned}
		\end{equation*}
		which is equivalent to
		\begin{equation*}
			\begin{aligned}
				\rho^{(1)}_{G_{(3; 3; 1, 1, 1)}} (T_1, z_2T_2, z_2T_3, z_3T_4, z_3T_5, z_2z_3T_6, z_2z_3T_7) &\geqslant 0.
			\end{aligned}
		\end{equation*}
		Since this is true for all $z_2, z_3 \in \mathbb{D}$ and $\rho^{(1)}_{G_{(3; 3; 1, 1, 1)}} $ is a continuous function of $z_2, z_3$, then the stated inequality is valid for all $z_2, z_3 \in \overline{\mathbb{D}}$ as well. 
This completes the proof.
	\end{proof}

	\begin{thm}\label{thm-3}
		Let $\textbf{T}=(T_1, \dots, T_7)$ be a $7$-tuple commuting contractions on some Hilbert space $\mathcal{H}$ and $(A_{z_i}, B_{z_i}, P_{z_i})$, $i = 1, 2, 3$ as above. Then in the following, $(1) \Rightarrow (2), (1^{'}) \Rightarrow (2^{'}), (1^{''}) \Rightarrow (2^{''})$:
		\begin{enumerate}
			\item $\{ (A_{z_2}, B_{z_2}, P_{z_2}) : z_2 \in \mathbb{D} \}$ is a family of tetrablock contractions.
			
			\item[($1^{'}$)] $\{ (A_{z_3}, B_{z_3}, P_{z_3}) : z_3 \in \mathbb{D}\}$ is a family of tetrablock contractions.
			
			\item[($1^{''}$)] $\{ (A_{z_1}, B_{z_1}, P_{z_1}) : z_1 \in \mathbb{D} \}$ is a family of tetrablock contractions.
			
			\item[($2$)]  For all $z_2, z_3 \in \overline{\mathbb{D}}$,
			\begin{equation*}
				\begin{aligned}
					\rho^{(1)}_{G_{(3; 3; 1, 1, 1)}} (T_1, z_2T_2, z_2T_3, z_3T_4, z_3T_5, z_2z_3T_6, z_2z_3T_7) &\geqslant 0.
				\end{aligned}	
			\end{equation*}
			
			\item[($2^{'}$)] For all $z_1, z_3 \in \overline{\mathbb{D}}$.
			\begin{equation*}
				\begin{aligned}
					\rho^{(2)}_{G_{(3; 3; 1, 1, 1)}} (z_1T_1, T_2, z_1T_3, z_3T_4, z_1z_3T_5, z_3T_6, z_1z_3T_7) &\geqslant 0.
				\end{aligned}
			\end{equation*}
			
			\item[($2^{''}$)] For all $z_1, z_2 \in \overline{\mathbb{D}}$,
			\begin{equation*}
				\begin{aligned}
					\rho^{(3)}_{G_{(3; 3; 1, 1, 1)}} (z_1T_1, z_2T_2, z_1z_2T_3, T_4, z_1T_5, z_2T_6, z_1z_2T_7) &\geqslant 0.
				\end{aligned}
			\end{equation*}
		\end{enumerate}
	\end{thm}
	
	\begin{proof}
		We only prove $(1) \Rightarrow (2),$ because the proofs for $(1^{'}) \Rightarrow (2^{'}), (1^{''}) \Rightarrow (2^{''})$ are analogous. Let $\{ (A_{z_2}, B_{z_2}, P_{z_2}) : |z_2| < 1 \}$ be a family of tetrablock contractions. For each $z_3\in \mathbb D,$  the function $\Psi(z_3, \cdot)$ is holomorphic on $G_{E(2; 2; 1, 1)}$ with
		\begin{equation}\label{Az2}
			\begin{aligned}
				\Psi(z_3, ((A_{z_2}, B_{z_2}, P_{z_2})))
				&= (A_{z_2} - z_3P_{z_2})(I - z_3B_{z_2})^{-1}.
			\end{aligned}
		\end{equation}
		As $\Gamma_{E(2; 2; 1, 1)}$ is a spectral set for $(A_{z_2}, B_{z_2}, P_{z_2})$ for fixed but arbitrary $z_2\in \mathbb D,$ it yields that  $$||\Psi(z_3, (A_{z_2}, B_{z_2}, P_{z_2}))|| \leqslant 1$$ which implies to 
		\begin{equation}\label{Az22}
			\begin{aligned}
				\Psi(z_3, ((A_{z_2}, B_{z_2}, P_{z_2})))^*\Psi(z_3, ((A_{z_2}, B_{z_2}, P_{z_2})))
				&= ((I - z_3B_{z_2})^{-1})^*(A_{z_2} - z_3P_{z_2})^*(A_{z_2} - z_3P_{z_2})(I - z_3B_{z_2})^{-1}\\&\leq I.
			\end{aligned}
		\end{equation}
and hence from \eqref{Az22} and \eqref{psi111546}, we have
		\begin{equation*}
			\begin{aligned}
				||\Psi^{(1)}(z_2, z_3, (T_1, \dots, T_7))|| \leqslant 1.
			\end{aligned}
		\end{equation*}
It follows from Theorem \ref{thm-2} that
		\begin{equation*}
			\begin{aligned}
				\rho^{(1)}_{G_{(3; 3; 1, 1, 1)}} (T_1, z_2T_2, z_2T_3, z_3T_4, z_3T_5, z_2z_3T_6, z_2z_3T_7) \geqslant 0
			\end{aligned}
		\end{equation*} for all $z_2,z_3\in \mathbb{\bar{D}}.$
This completes the proof.
	\end{proof}

Let $(S_1, S_2, S_3, \tilde{S}_1, \tilde{S}_2)$ be a $5$-tuple of commuting bounded operators defined on some Hilbert space $\mathcal H_1.$ Similar to the operator functions $\rho_{G_{E(2; 1; 2)}},$ $\rho_{G_{E(2; 2; 1,1)}} $ and  $\rho_{G_{E(3; 3; 1,1,1)}},$ we introduce an operator function for the commuting tuple $(S_1, S_2, S_3, \tilde{S}_1, \tilde{S}_2)$ of bounded operators with $\|S_3\|\leq 1$, which plays an important role in characterization of the $\Gamma_{E(3; 2; 1, 2)}$-contractions:
		\begin{equation}
			\begin{aligned}
				\rho_{G_{E(3; 2; 1, 2)}}(S_1, S_2, S_3, \tilde{S}_1, \tilde{S}_2)
				&= D^2_{S_3} + (\tilde{S}^*_1\tilde{S}_1 - S^*_2S_2 - S^*_1S_1 + \tilde{S}^*_2\tilde{S}_2) - 2Re ~ (\tilde{S}_1 - S^*_1S_2)\\
				&~~~~ + 2Re ~ (\tilde{S}_2 - S^*_1S_3) - 2Re ~ (\tilde{S}^*_1\tilde{S}_2 - S^*_2S_3).
			\end{aligned}
		\end{equation}
Let us assume that $\|\tilde{S}_1\|\leq 2$. We consider a family of triples of bounded operators as follows:
		\begin{equation}
			\begin{aligned}
				(P^{(1)}_z, P^{(2)}_z, P^{(3)}_z) 
				&= \{ ((2S_1 - zS_2)(2I - z\tilde{S}_1)^{-1}, (\tilde{S}_1 - 2z\tilde{S}_2)(2I - z\tilde{S}_1)^{-1},\\
				&\hspace{1cm} (S_2 - 2zS_3)(2I - z\tilde{S}_1)^{-1}) : z \in \mathbb{D} \}.
			\end{aligned}
		\end{equation}

\begin{defn}
		Let $(S_1, S_2, S_2, \tilde{S}_1, \tilde{S}_2)$ be a $5$-tuples of commuting operators on a Hilbert space $\mathcal{H}$ and $\{(P^{(1)}_z, P^{(2)}_z, P^{(3)}_z) : z \in \mathbb{D}\}$ are as above. For fixed but arbitrary $z\in \mathbb D,$ the equations 
		\begin{equation}\label{Fundamental 2}
			\begin{aligned}
				P^{(1)}_z - P^{(2)}_zP^{(3)}_z
				&= D_{P^{(3)}_z}\tilde{F}^{(1)}_zD_{P^{(3)}_z} ~\text{and}~ P^{(2)}_z - P^{(1)}_zP^{(3)}_z
				= D_{P^{(3)}_z}\tilde{F}^{(2)}_zD_{P^{(3)}_z}
			\end{aligned}
			\end{equation}
			
where $F_{z}^{(j)}\in \mathcal{B}(\mathcal{D}_{P_{z}^{(3)}}), 1\leq j\leq2$, are called the fundamental equations for $(S_1, S_2, S_3, \tilde{S}_1, \tilde{S}_2)$.
	\end{defn}
	
	\begin{thm}\label{thm-3}
		Let $(S_1, S_2, S_3, \tilde{S}_1, \tilde{S}_2)$ denote a $5$-tuple of commuting bounded operators defined on a Hilbert space $\mathcal{H}$. Then  the following implications hold  $(1) \Rightarrow (2) 
		\Rightarrow (3)\Rightarrow (4) \Rightarrow (5):$

		\begin{enumerate}
			\item $(S_1, S_2, S_3, \tilde{S}_1, \tilde{S}_2)$ is a $\Gamma_{E(3; 2; 1, 2)}$-contraction.
			
			\item $\{(P^{(1)}_z, P^{(2)}_z, P^{(3)}_z) : z \in \mathbb{D}\}$ is a family of $\Gamma_{E( 2; 2;1, 1)}$-contractions.
			
			\item For fixed but arbitrary $z\in \mathbb D$ and for all $\omega \in \mathbb{T}$,
			\begin{equation*}
				\begin{aligned}
					\rho_{G_{E(2; 2; 1, 1)}}(P^{(1)}_z, \omega P^{(2)}_z, \omega P^{(3)}_z)
					&\geqslant 0, \rho_{G_{E(2; 2; 1, 1)}}(P^{(2)}_z, \omega P^{(1)}_z, \omega P^{(3)}_z) \geqslant 0.
				\end{aligned}
			\end{equation*}
			
			\item For each fixed but arbitrary $z \in \mathbb{D}$, the pair $(\tilde{S}_{\omega}(z), \tilde{P}_{\omega}(z)) = (P^{(1)}_z + \omega P^{(2)}_z, \omega P^{(3)}_z)$ is a $\Gamma_{E(2; 1; 2)}$-contraction for all $\omega \in \mathbb{T}$.
			
			\item For each fixed but arbitrary $z \in \mathbb{D}$ the fundamental equations in \eqref{Fundamental 2} have unique solutions $\tilde{F}^{(1)}_z$ and $\tilde{F}^{(2)}_z$ in $\mathcal{B}(\mathcal{D}_{{P^{(3)}_z}})$. Moreover, the numerical radius of the operator $\tilde{F}^{(1)}_z + \omega\tilde{F}^{(2)}_z$ is not bigger than $1$ for every $\omega \in \mathbb{T}$.
		\end{enumerate}
	\end{thm}
	
	\begin{proof}
We first show that $(1) \Rightarrow (2)$. For each $z \in \mathbb{D},$ we define the map $\pi_z :\Gamma_{E(3; 2; 1, 2)} \to \mathbb{C}^3$ as follows:
		\begin{equation*}
			\begin{aligned}
				\pi_z(x_1, x_2, x_3, y_1, y_2)
				&= \Big(\frac{2x_1 - zx_2}{2 - zy_1}, \frac{y_1 - 2zy_2}{2 - zy_1}, \frac{x_2 - 2zx_3}{2 - zy_1}\Big).
			\end{aligned}
		\end{equation*}
 Observe that  $p \circ \pi_z$ is a rational function on $\Gamma_{E(3; 2; 1, 2)}$ for any $p \in \mathbb{C}[z_1, z_2, z_3].$ Since $(S_1, S_2, S_3, \tilde{S}_1, \tilde{S}_2)$ is a $\Gamma_{E(3; 2; 1, 2)}$-contraction, so $\Gamma_{E(3; 2; 1, 2)}$ is a spectral set for $(S_1, S_2, S_3, \tilde{S}_1, \tilde{S}_2)$. Thus we have
		\begin{equation*}
			\begin{aligned}
				||p(P^{(1)}_z, P^{(2)}_z, P^{(3)}_z)||
				&= ||p \circ \pi_z(S_1, S_2, S_3, \tilde{S}_1, \tilde{S}_2)||\\
				&\leqslant ||p \circ \pi_z||_{\infty, \Gamma_{E(3; 2; 1, 2)}}\\
				&= ||p||_{\infty, \pi_z(\Gamma_{E(3; 2; 1, 2)})}\\
				&\leqslant ||p||_{\infty, \Gamma_{E(2; 2; 1, 1)}}.
			\end{aligned}
		\end{equation*}
This implies that $\{(P^{(1)}_z, P^{(2)}_z, P^{(3)}_z) : z \in \mathbb{D}\}$ is a family of tetrablock contractions.

The proof of the implication $(2) \Rightarrow (3) \Rightarrow (4) \Rightarrow (5)$ is analogous to [Theorem $3.4$, \cite{Bhattacharyya}]. This completes the proof.	
%
%
%
%
%
	\end{proof}
	
	\begin{thm}\label{thm-4}
		Let $\textbf{S} = (S_1, S_2, S_3, \tilde{S}_1, \tilde{S}_2)$ be a $5$-tuple of commuting bounded operators acting on a Hilbert space $\mathcal{H}$ and $\textbf{S}$ be a $\Gamma_{E(3; 2; 1, 2)}$-contraction. Then
		\begin{equation*}
			\begin{aligned}
				\rho_{G_{E(3; 2; 1, 2)}}(S_1, zS_2, z^2S_3, z\tilde{S}_1, z^2\tilde{S}_2) &\geqslant 0~{\rm{for~ all~}}z \in \overline{\mathbb{D}}.
			\end{aligned}
		\end{equation*}
	\end{thm}
	
	\begin{proof}
		Let $\textbf{S}$ be a $\Gamma_{E(3; 2; 1, 2)}$-contraction. Consider the function $\Psi_3$ as described in \eqref{psi1}. For $z \in \mathbb{D},$ the function $\Psi_3(z,\cdot)$ is holomorphic function on $G_{E(3; 2; 1,2)}$ with
\begin{equation*}
			\begin{aligned}
		\Psi_3(z, S_1, S_2, S_3, \tilde{S}_1, \tilde{S}_2)=(S_1 - zS_2 + z^2S_3)(I - z\tilde{S}_1 + z^2\tilde{S}_2)^{-1}.
			\end{aligned}
		\end{equation*}
As $\Gamma_{E(3; 2; 1, 2)}$ is a spectral set for $\textbf{S}$, by Theorem \ref{bhch1}, it follows that 

\begin{eqnarray}\label{S1}
&& ||\Psi_3(z, S_1, S_2, S_3, \tilde{S}_1, \tilde{S}_2)||
				\leqslant 1 \nonumber\\
				&\Leftrightarrow& (S^*_1 - \overline{z}S^*_2 + \overline{z}^2S^*_3)(S_1 - zS_2 + z^2S_3)
				\leqslant
				(I - \overline{z}\tilde{S}^*_1 + \overline{z}^2\tilde{S}^*_2)(I - z\tilde{S}_1 + z^2\tilde{S}_2) \nonumber\\
				&\Leftrightarrow & 				
(I - |z|^4S^*_3S_3) + (|z|^2\tilde{S}^*_1\tilde{S}_1 - |z|^2S^*_2S_2 - S^*_1S_1 + |z|^4\tilde{S}^*_2\tilde{S}_2) \nonumber\\
&&
				- 2Re\, z(\tilde{S}_1 - S^*_1S_2)
				+ 2Re\, z^2(\tilde{S}_2 - S^*_1S_3) - 2 Re\, z|z|^2(\tilde{S}^*_1\tilde{S}_2 - S^*_2S_3) \geqslant 0	\nonumber \\ &\Leftrightarrow & 					
				\rho_{G_{E(3; 2; 1, 2)}}(S_1, zS_2, z^2S_3, z\tilde{S}_1, z^2\tilde{S}_2) \geqslant 0.							
\end{eqnarray}

Thus, we have
$\rho_{G_{E(3; 2; 1, 2)}}(S_1, zS_2, z^2S_3, z\tilde{S}_1, z^2\tilde{S}_2) \geqslant 0~{\rm{for~all}}~z\in \mathbb D$. Since this is true for all $z \in \mathbb{D}$, by continuity of $\rho_{G_{E(3; 2; 1, 2)}},$ we conclude that 
		\begin{equation*}
			\begin{aligned}
				\rho_{G_{E(3; 2; 1, 2)}}(S_1, zS_2, z^2S_3, z\tilde{S}_1, z^2\tilde{S}_2) &\geqslant 0~{\rm{for~all}}~z\in \overline{\mathbb D}.
			\end{aligned}
		\end{equation*}
This completes the proof.
	\end{proof}

\begin{lem}\label{psi12}
Let  $\tilde{\textbf{x}}=(x_1,x_2,x_3,y_1,y_2)$ be an element of $\mathbb{C}^{5}.$ Then  $\tilde{\textbf{x}}\in \Gamma_{E(3;2;1,2)}$ if and only if $$\sup_{z\in \mathbb D}\|\Psi(z,(p_1(z),p_2(z),p_3(z)))\|\leq 1, $$ where  $p_1(z)=\frac{2x_1-zx_2}{2-y_1z},p_2(z)=\frac{y_1-2zy_2}{2-y_1z}~{\rm{and}}~p_3(z)=\frac{x_2-2zx_3}{2-y_1z}.$
\end{lem}

\begin{proof}
By Proposition \ref{bhch1}, we observe that $\tilde{\textbf{x}}\in \Gamma_{E(3;2;1,2)}$ if and only if $y_2z_2^2-y_1z_2+1\neq 0$ for all $z\in {\mathbb{D}}$ and $\|\Psi_{3}(\cdot,\tilde{\textbf{x}})\|_{H^{\infty}({\mathbb{D}})}:=\sup_{z\in {\mathbb{D}}}|\Psi_{3}(z,\tilde{\textbf{x}})|\leq 1$. Note that
\begin{align}\label{psii}
\Psi_{3}(z,\tilde{\textbf{x}})\nonumber &=\frac{x_3z^2-x_2z+x_1}{y_2z^2-y_1z+1}\\&=
\frac {p_1(z)-zp_3(z)}{1-zp_2(z)} \nonumber \\&= \Psi(z,(p_1(z),p_2(z),p_3(z))).
\end{align}
Thus, from \eqref{psii} and  Proposition \ref{bhch1}, we deduce that $\tilde{\textbf{x}}\in \Gamma_{E(3;2;1,2)}$ if and only if $$\sup_{z\in \mathbb D}\|\Psi(z,(p_1(z),p_2(z),p_3(z)))\|\leq 1.$$ This completes the proof.
\end{proof}

	\begin{thm}\label{them-2}
		Let $\{(P^{(1)}_z, P^{(2)}_z, P^{(3)}_z) : z \in \mathbb{D}\}$ be a family of tetrablock contraction. Then
		\begin{equation*}
			\begin{aligned}
				\rho_{G_{E(3; 2; 1, 2)}}(S_1, zS_2, z^2S_3, z\tilde{S}_1, z^2\tilde{S}_2) &\geqslant 0~{\rm{for~all}}~z\in \mathbb {\bar{D}}.
			\end{aligned}
		\end{equation*}
	\end{thm}
	
	\begin{proof}
Suppose that  $\{(P^{(1)}_z, P^{(2)}_z, P^{(3)}_z) : z \in \mathbb{D}\}$ is a family of tetrablock contractions. For each $w\in \mathbb D,$  the function $\Psi(w, \cdot)$ is holomorphic on $G_{E(2; 2; 1, 1)}$ with
\begin{equation*}
			\begin{aligned}
				\Psi(w, (P^{(1)}_z, P^{(2)}_z, P^{(3)}_z)) &= (P^{(1)}_z - wP^{(3)}_z)(I - wP^{(2)}_z)^{-1}.
			\end{aligned}
		\end{equation*}

As $\Gamma_{E(2; 2; 1, 1)}$ is a spectral set for $(P^{(1)}_z, P^{(2)}_z, P^{(3)}_z)$ for fixed but arbitrary $z\in \mathbb D,$ by [Therem 2.4, \cite {Abouhajar}], we have $$||\Psi(w, (P^{(1)}_z, P^{(2)}_z, P^{(3)}_z))|| \leqslant 1, \;\;\;\mbox{for all}\;\;w\in \mathbb{D}$$ In particular, for  $w=z$, we get
		\begin{equation*}
			\begin{aligned}
				\|\Psi(z, (P^{(1)}_z, P^{(2)}_z, P^{(3)}_z))\| &\leqslant 1
			\end{aligned}
		\end{equation*}
		and consequently by \eqref{psii} we derive the following inequality
		\begin{equation*}
			\begin{aligned}
				||\Psi_3(z, S_1, S_2, S_3, \tilde{S}_1, \tilde{S}_2)||
				&\leqslant 1.
			\end{aligned}
		\end{equation*}
By \eqref{S1} and continuity of $\rho_{G_{E(3; 2; 1, 2)}}$ in the variable $z$,  we conclude that
		\begin{equation*}
			\begin{aligned}
				\rho_{G_{E(3; 2; 1, 2)}}(S_1, zS_2, z^2S_3, z\tilde{S}_1, z^2\tilde{S}_2) &\geqslant 0 ~{\rm{for~all}}~z\in \mathbb {\bar{D}}.
			\end{aligned}
		\end{equation*}
		This completes the proof.
	\end{proof}
	
\begin{rem}\label{rem-5}

Let $\textbf{S} = (S_1, S_2, S_3, \tilde{S}_1, \tilde{S}_2)$ be a $\Gamma_{E(3; 2; 1, 2)}$-contraction.
Let $p$ be the polynomial in $\mathbb{C}[x_1,x_2,x_3,y_1,y_2]$ defined by $p(x_1,x_2,x_3,y_1,y_2)=y_1$. By Lemma \ref{lem-3} and Proposition \ref{prop-1}, it follows that $\|\tilde{S}_1\|=\|p(\textbf{S})\|\leq \|p\|_{\infty,\Gamma_{E(3; 2; 1, 2)}}=\|y_1\|_{\infty, \Gamma_{E(3; 2; 1, 2)}}\leq 2$. If $\|\tilde{S}_1\|=\|p(\textbf{S})\|<2$, then $(2I-z\tilde{S}_1)$ is invertible for all $z\in \mathbb {\bar{D}}$.
By using the similar argument as in the Theorem \ref{thm-3}, we deduce that if $\textbf{S}$ is a $\Gamma_{E(3; 2; 1, 2)}$-contraction, then  $\{(P^{(1)}_z, P^{(2)}_z, P^{(3)}_z) : z \in \mathbb{\bar{D}}\}$ is a family of tetrablock contractions. 
\end{rem}

\section{$\Gamma_{E(3; 3; 1, 1, 1)}$-Unitaries and $\Gamma_{E(3; 2; 1, 2)}$-Unitaries}
In this section, we investigate the $\Gamma_{E(3; 3; 1, 1, 1)}$-unitaries and $\Gamma_{E(3; 2; 1, 2)}$-unitaries and elaborate on the relationship between them. The following lemma provides the various characterization of the set $K$. As the proof of the lemma follows easily, therefore we omit the proof.

\begin{lem}\label{lem}
		Let $\textbf{x} = (x_1, \dots, x_7) \in \mathbb{C}^7$. Then the followings are equivalent:
		\begin{enumerate}
			\item $\textbf{x} \in K$. 
			
			\item $(\omega x_1, x_2, \omega x_3, \omega x_4, \omega^2 x_5, \omega x_6, \omega^2 x_7) \in K$ for all $\omega \in \mathbb T$.
			
			\item[$(2^{'})$] $ (\omega x_1, \omega x_2, \omega^2 x_3, x_4, \omega x_5, \omega x_6, \omega^2 x_7) \in K$ for all $\omega \in \mathbb T$.
			
			\item[$(2^{''})$] $ (x_1, \omega x_2, \omega x_3, \omega x_4, \omega x_5, \omega^2 x_6, \omega^2 x_7) \in K$ for all $\omega \in \mathbb T$.
		\end{enumerate}
	\end{lem}

	\begin{thm}\label{thm-5}
		Let $\textbf{N} = (N_1, \dots, N_7)$ denote a commuting $7$-tuple of bounded operators defined on a Hilbert space $\mathcal{H}$. Then, the following conditions are equivalent:
		\begin{enumerate}		
			\item $\textbf{N}$ is a $\Gamma_{E(3; 3; 1, 1, 1)}$-unitary.
			
			\item For $1\leq i\leq 6,$ $N_i$'s are contractions, $N_i = N^*_{7-i} N_7$ and $N_7$ is unitary.
			\item For $1\leq i\leq 6,$ $(N_i,N_{7-i},N_7)$ are $\Gamma_{E(2; 2; 1, 1)}$-unitary.
			
			\item $N_7$ is a unitary and $\textbf{N}$ is a $\Gamma_{E(3; 3; 1, 1, 1)}$-contraction.
					
			\item $||N_2|| < 1$, $\|N_4\|<1,$ $\{((N_1 - z_2N_3)(I - z_2N_2)^{-1}, (N_4 - z_2N_6)(I - z_2N_2)^{-1}, (N_5 - z_2N_7)(I - z_2N_2)^{-1}) : |z_2| = 1\}$  and  $\{((N_1 - z_3N_5)(I - z_3N_4)^{-1},  (N_2 - z_3N_6)(I - z_3N_4)^{-1}, (N_3 - z_3N_7)(I - z_3N_4)^{-1}) : |z_3| = 1\}$ are  commuting family of tetrablock unitaries.
				
				\item[$(5)^{\prime}$]  $N_2$ is unitary and $\{((N_1 - z_2N_3)(I - z_2N_2)^{-1}, (N_4 - z_2N_6)(I - z_2N_2)^{-1}, (N_5 - z_2N_7)(I - z_2N_2)^{-1}) : |z_2| < 1\}$ is a commuting family of tetrablock unitaries.
			\item[$(5)^{\prime\prime}$] $N_4$ is unitary and $ \{((N_1 - z_3N_5)(I - z_3N_4)^{-1},  (N_2 - z_3N_6)(I - z_3N_4)^{-1}, (N_3 - z_3N_7)(I - z_3N_4)^{-1}) : |z_3| < 1\}$ is a commuting family of tetrablock unitaries.
						
				\item  $||N_1|| < 1$, $\|N_2\|<1,$  $\{((N_2 - z_1N_3)(I - z_1N_1)^{-1},  (N_4 - z_1N_5)(I - z_1N_1)^{-1}, (N_6 - z_1N_7)(I - z_1N_1)^{-1}) : |z_1| = 1\}$ and  $\{((N_1 - z_2N_3)(I - z_2N_2)^{-1}, (N_4 - z_2N_6)(I - z_2N_2)^{-1}, (N_5 - z_2N_7)(I - z_2N_2)^{-1}) : |z_2| = 1\}$ are  commuting family of tetrablock unitaries.
				
				\item [$(6)^{\prime}$]  $N_1$ is unitary and $ \{((N_2 - z_1N_3)(I - z_1N_1)^{-1},  (N_4 - z_1N_5)(I - z_1N_1)^{-1}, (N_6 - z_1N_7)(I - z_1N_1)^{-1}) : |z_1| < 1\}$ is a commuting family of tetrablock unitaries.
			
			\item[$(6)^{\prime\prime}$]  $N_2$ is unitary and $\{((N_1 - z_2N_3)(I - z_2N_2)^{-1}, (N_4 - z_2N_6)(I - z_2N_2)^{-1}, (N_5 - z_2N_7)(I - z_2N_2)^{-1}) : |z_2| < 1\}$ is a commuting family of tetrablock unitaries.
					
			\item  $\|N_1\|<1,$ $\|N_4\| < 1$, $ \{((N_1 - z_3N_5)(I - z_3N_4)^{-1},  (N_2 - z_3N_6)(I - z_3N_4)^{-1}, (N_3 - z_3N_7)(I - z_3N_4)^{-1}) : |z_3| = 1\}$ and $\{((N_2 - z_1N_3)(I - z_1N_1)^{-1},  (N_4 - z_1N_5)(I - z_1N_1)^{-1}, (N_6 - z_1N_7)(I - z_1N_1)^{-1}) : |z_1| = 1\}$ are  commuting family of tetrablock unitaries.				
				\item [$(7)^{\prime}$] $N_1$ is unitary and $ \{((N_2 - z_1N_3)(I - z_1N_1)^{-1},  (N_4 - z_1N_5)(I - z_1N_1)^{-1}, (N_6 - z_1N_7)(I - z_1N_1)^{-1}) : |z_1| < 1\}$ is a commuting family of tetrablock unitaries.
				\item [$(7)^{\prime\prime}$]  $N_4$ is unitary and $\mathcal{F}^{'}_2 = \{((N_1 - z_3N_5)(I - z_3N_4)^{-1},  (N_2 - z_3N_6)(I - z_3N_4)^{-1}, (N_3 - z_3N_7)(I - z_3N_4)^{-1}) : |z_3| < 1\}$ is a commuting family of tetrablock unitaries.

		\end{enumerate}
	\end{thm}
	
	\begin{proof}
We show the equivalence of $(1), (2), (3),(4),(5), (5)^{\prime},(5)^{\prime\prime},(6), (6)^{\prime},(6)^{\prime\prime}, (7) (7)^{\prime}$ and $(7)^{\prime\prime}$  by demonstrating the following relationships: 
\\
\small{$(1) \Leftrightarrow (2) \Leftrightarrow (3), (1) \Rightarrow (4),(4)  \Rightarrow (3), (4)  \Rightarrow (5), (4)  \Rightarrow (5)^{\prime}, (4)  \Rightarrow (5)^{\prime \prime},(5) \Rightarrow(1), (5)^{\prime} \Rightarrow(1), $} \small{$ (5)^{\prime\prime} \Rightarrow(1), (4)  \Rightarrow (6),(4)  \Rightarrow (6)^{\prime}, (4)  \Rightarrow (6)^{\prime \prime}, (6) \Rightarrow(1), (6)^{\prime} \Rightarrow(1), (6)^{\prime\prime} \Rightarrow(1), (4)\Rightarrow (7),$}
$ (4)  \Rightarrow (7)^{\prime},  (4)  \Rightarrow (7)^{\prime \prime}, (7) \Rightarrow (1), (7)^{\prime} \Rightarrow(1) ~{\rm{and}}~  (7)^{\prime\prime} \Rightarrow(1).$
\\
$(1) \Rightarrow (2):$ Suppose that $\textbf{N} = (N_1, \dots, N_7)$ is a $\Gamma_{E(3; 3; 1, 1, 1)}$-unitary. By the definition of $\Gamma_{E(3; 3; 1, 1, 1)}$-unitary, it follows that $N_i$ for $1\leq i\leq 7$ are commuting normal operators, and their Taylor joint spectrum $\sigma(\textbf{N})$ is contained in the set  $ K$. Let $P_7:\mathbb{C}^7 \to \mathbb{C}^7$ be the projection map onto the $7$th coordinate. It follows from the \textit{Spectral mapping theorem} that $\sigma(N_7) =\sigma (P_7(\textbf{N})) =P_7(\sigma(\textbf{N}))$. Since $\sigma(\textbf{N})$ is a subset of the set  $ K$, it implies that $|\mu|=1$ for every $\mu \in \sigma(N_7) .$ Thus, $N_7$ is a normal operator with  $\sigma(N_7) \subseteq \mathbb{T}$; consequently,  $N_7$ is a unitary operator.
				
Consider the $C^*$-algebra $\mathcal A$ generated by the commuting normal operators $I,N_1, \dots, N_7$. By \textit{Gelfand-Naimark theorem}, this commutative $C^*$-algebra $\mathcal A$ is isometrically $^*$-isomorphic to $C(\sigma(\textbf{N}))$ via the \textit{Gelfand map.}  The \textit{Gelfand map}  corresponds  $N_i$ to the coordinate function $ x_i$ for $ i = 1, \dots, 7$. The coordinate functions satisfy the conditions $x_1 = \overline{x}_6 x_7, x_2 = \overline{x}_5 x_7, ~\text{and}~ x_3 = \overline{x}_4 x_7$ on the set $K$ and consequently on  $\sigma(\textbf{N})$, which is a subset of $ K$. Thus, we have $N_1 = N^*_6 N_7, N_2 = N^*_5 N_7, N_3 = N^*_4 N_7$ and $N_1, \dots, N_6$ are contractions.

\vspace{0.3cm}
$(2) \Rightarrow (1):$ We first show that $N_{7-i}$'s are normal for $i=1,\dots,6$. For $1\leq i\leq 6,$ we have $N_iN_{7-i}=N^*_{7-i}N_7N_{7-i}=N^*_{7-i}N_{7-i}N_{7}.$ On the other hand,  for $1\leq i\leq 6,$ we get $N_{7-i}N_{i}=N_{7-i}N^*_{7-i}N_7.$ Since for $1\leq i\leq 6,$ $N_i$ and $N_{7-i}$ are commute, it follows from the above observation that  $N^*_{7-i}N_{7-i}N_{7}=N_{7-i}N^*_{7-i}N_7.$ Multiplying on the right by $N_7^*$, we deduce that  $N_{7-i}$'s are normal for $i=1,\dots,6$. Let  $\mathcal A$ be the $C^*$-algebra  generated by the commuting normal operators $I,N_1, \dots, N_7$. By the  \textit{Gelfand-Naimark theorem}, this commutative $C^*$-algebra $\mathcal A$ is isometrically $^*$-isomorphic to $C(\sigma(\textbf{N}))$ via the \textit{Gelfand map.} The inverse of the \textit{Gelfand map} takes the co-ordinate functions $x_j$ to $N_j$ for $1\leq j\leq 7$. In order to complete the proof, it is sufficient to show that the Taylor joint spectrum $\sigma(\textbf{N})$ is contained in the set  $ K$. Let $(x_1,\dots,x_7)\in \sigma(\textbf{N}).$
Since $N_i$'s are contractions, $N_i = N^*_{7-i} N_7$ for $1\leq i\leq 6$, and $N_7$ is unitary; the fact that the inverse \textit{Gelfand map} sends $x_i$ to $N_i$ for $i=1,\dots,7$, we deduce that $x_i=\bar{x}_{7-i}x_7$ for $1\leq i\leq 6$, $~{\rm{and}}~|x_7|=1$, which implies that $(x_1,\dots,x_7)\in K$.

		$(2) \Leftrightarrow (3):$ 
The equivalence of $(2)$ and $(3)$ follows from [Theorem $5.4$,\cite{Bhattacharyya}].

$(1) \Rightarrow (4):$ Suppose that $\textbf{N}$ is a $\Gamma_{E(3; 3; 1, 1, 1)}$-unitary. By equivalence of $(1)$ and $(2)$, we deduce that $N_7$ is a unitary. Let $p$ be a polynomial of seven variables. Since $N_1,N_2,\ldots,N_7$ are commuting normal operators, so $p(N_1,\dots,N_7)$ is a normal operator. Then 
\begin{align*} \|p(N_1,\dots,N_7)\|&=\sup\{|p(\lambda_1,\dots,\lambda_7)|: (\lambda_1,\dots,\lambda_7)\in \sigma(\textbf{N})\}\\&\leq
\sup\{|p(\lambda_1,\dots,\lambda_7)|: (\lambda_1,\dots,\lambda_7)\in K\}
\\&\leq\sup\{|p(\lambda_1,\dots,\lambda_7)|: (\lambda_1,\dots,\lambda_7)\in \Gamma_{E(3; 3; 1, 1, 1)}\}\\
&= \|p\|_{\infty, \Gamma_{E(3; 3; 1, 1, 1)}},
\end{align*}
by Proposition \ref{prop-1}, it follows that $\textbf{N}$ is a $\Gamma_{E(3; 3; 1, 1, 1)}$-contraction.

$(4) \Rightarrow (3):$  As $\textbf{N}$ is a $\Gamma_{E(3; 3; 1, 1, 1)}$-contraction, by Proposition \ref{props-1}, it yields that $(N_1,N_6,N_7),$ $ (N_2,N_5,N_7)$ and $(N_3,N_4,N_7)$ are  $\Gamma_{E(2; 2; 1,1)}$-contractions. Since $N_7$ is a unitary, it follows from  [Theorem $5.4$,\cite{Bhattacharyya}] that  the triplet $(N_i,N_{7-i},N_7)$ is $\Gamma_{E(2; 2; 1, 1)}$-unitary for $1\leq i\leq 6$.
\newline We demonstrate only $(1)\Leftrightarrow(5)$ and $(1)\Leftrightarrow(5)^{\prime}$, other equivalences, namely, $(1)\Leftrightarrow(5)^{\prime\prime}, (1)\Leftrightarrow(6), (1)\Leftrightarrow(6)^{\prime},(1) \Leftrightarrow(6)^{\prime\prime}, (1)\Leftrightarrow(7), (1)\Leftrightarrow(7)^{\prime}, (1)\Leftrightarrow(7)^{\prime\prime}$ are established in a similar manner.

$(1) \Rightarrow (5):$  Since $\mathbf{N}$ is a $\Gamma_{E(3; 3; 1, 1, 1)}$ unitary so  $N_7$ is a unitary operator and $\textbf{N}$ is a $\Gamma_{E(3; 3; 1, 1, 1)}$-contraction. By equivalence of $(1),(2)$ and $(4)$, we observe that $N_i$'s are normal operators with $\|N_i\|\leq 1$ and $N_i = N^*_{7-i} N_7$ for $1\leq i\leq 6$.  Consider the $C^*$-algebra $\mathcal A$ generated by the commuting normal operators $I,N_1, \dots, N_7$. By \textit{Gelfand-Naimark theorem}, $\mathcal A$ is isometrically $^*$-isomorphic to $C(\sigma(\textbf{N}))$ via the \textit{Gelfand map.} Let $(x_1,\dots,x_7)\in \sigma(\textbf{N}).$ Since Taylor joint spectrum $\sigma(\textbf{N})$ is contained in the set  $ K$,  we deduce from Theorem \ref{relation bw} that  $(x_1,\dots,x_7)\in K$ if and only if $$ \tilde{\textbf{z}}^{(z_3)}\in b\Gamma_{E(2;2;1,1)}~\text{and}~\tilde{\textbf{y}}^{(z_2)}\in b\Gamma_{E(2;2;1,1)}~\text{for all}~z_2,z_3\in \mathbb{T}~\text{with}~|x_4|<1 ~\text{and}~|x_2|<1.$$

The inverse of the \textit{Gelfand map} takes the co-ordinate functions $x_j$ to $N_j$  for $1\leq j\leq 7$ and hence we have $\|N_2\|<1$ and $\|N_4\|<1.$
As $\textbf{N}$ is a $\Gamma_{E(3; 3; 1, 1, 1)}$-contraction,  the Theorem \ref{thm-1} and Remark \ref{remm1} imply  that \[\{((N_1 - z_2N_3)(I - z_2N_2)^{-1}, (N_4 - z_2N_6)(I - z_2N_2)^{-1}, (N_5 - z_2N_7)(I - z_2N_2)^{-1}) : z_2 \in \mathbb {\bar{D}}\}\]  and  \[\{((N_1 - z_3N_5)(I - z_3N_4)^{-1},  (N_2 - z_3N_6)(I - z_3N_4)^{-1}, (N_3 - z_3N_7)(I - z_3N_4)^{-1}) : z_3 \in \mathbb {\bar{D}}\}\] are  commuting family of tetrablock contractions. In particular, \[\{((N_1 - z_2N_3)(I - z_2N_2)^{-1}, (N_4 - z_2N_6)(I - z_2N_2)^{-1}, (N_5 - z_2N_7)(I - z_2N_2)^{-1}) : |z_2| = 1\}\]  and  \[\{((N_1 - z_3N_5)(I - z_3N_4)^{-1},  (N_2 - z_3N_6)(I - z_3N_4)^{-1}, (N_3 - z_3N_7)(I - z_3N_4)^{-1}) : |z_3| = 1\}\] are  commuting family of tetrablock contractions. By [Theorem $5.4$,\cite{Bhattacharyya}], to complete the proof it is sufficient to demonstrate that  $(N_5 - z_2N_7)(I - z_2N_2)^{-1}$ and $(N_3 - z_3N_7)(I - z_3N_4)^{-1})$ are unitary  operators for all $z_2,z_3 \in \mathbb T$. We show that $(N_5 - z_2N_7)(I - z_2N_2)^{-1}$ is a unitary for all $z_2\in \mathbb T,$  the proof of $(N_3 - z_3N_7)(I - z_3N_4)^{-1})$ is unitary for all $z_3\in \mathbb T$ follows in a similar manner. Note that 
$N_2=N_5^*N_7$, $N_5=N_2^*N_7$ 
\begin{eqnarray}
N_5^*N_5&=& N_7^*N_2N_2^*N_7\nonumber\\
&=& N_7^*N_2^*N_2N_7 \nonumber\\
&=& N_2^*N_7^*N_7N_2\nonumber\\
&=& N_2^*N_2 \nonumber
\end{eqnarray}
 and 	\begin{align}\label{n-5}
(N^*_5 - \overline{z}_2N^*_7)(N_5 - z_2N_7)
\nonumber&=N^*_5N_5 - \overline{z}_2N^*_7N_5 - z_2N^*_5N_7 + N^*_7N_7
\\\nonumber&=N^*_2N_2 - \overline{z}_2N^*_2 - z_2N_2 + I
\\\nonumber&=N_2N^*_2 - \overline{z}_2N^*_2 - z_2N_2 + I 
\\ &=(I-z_2N_2)(I-\overline{z}_2N_2^*).
\end{align}

From \eqref{n-5}, we conclude that $(N_5 - z_2N_7)(I - z_2N_2)^{-1}$  is a unitary for all $z_2\in \mathbb T$.

We now demonstrate that $(1) \Rightarrow (5)^{\prime},$  the proof of $(1) \Rightarrow (5)^{\prime \prime}$ follows  similarly. 

As $\mathbf{N}$ is a $\Gamma_{E(3; 3; 1, 1, 1)}$ unitary, it follows that $N_7$ is a unitary operator and $\textbf{N}$ is a $\Gamma_{E(3; 3; 1, 1, 1)}$-contraction. From the equivalence of $(1),(2)$ and $(4)$, it yields that $N_i$'s are normal operators with $\|N_i\|\leq 1$ and $N_i = N^*_{7-i} N_7$ for $1\leq i\leq 6$.   Let $\mathcal A$ be the $C^*$-algebra generated by the commuting normal operators $I,N_1, \dots, N_7$. The \textit{Gelfand-Naimark theorem} states that $\mathcal A$ is isometrically $^*$-isomorphic to $C(\sigma(\textbf{N}))$ via the \textit{Gelfand map.} Let $(x_1,\dots,x_7)\in \sigma(\textbf{N}).$ As Taylor joint spectrum $\sigma(\textbf{N})$ is contained in the set  $ K$,  it follows from Theorem \ref{relation bw} that $$\tilde{\textbf{y}}^{(z_2)}\in b\Gamma_{E(2;2;1,1)}~\text{for all}~z_2\in \mathbb{D}~\text{with}~|x_2|=1.$$ Since $|x_2|=1$ and $|x_7|=1,$ we have $|x_5|=1.$  Let $P_2, P_5$ and $P_7$ denote the projections onto the $2$nd, $5$th and $7$th coordinates, respectively. Since the commutative $C^*$-algebra $\mathcal A$ is isometrically $^*$-isomorphic to $C(\sigma(\textbf{N}))$ via the \textit{Gelfand map} and applying  Spectral mapping theorem, we conclude that $N_2$, $N_5$ and $N_7$ are unitary operators.	 Since $\textbf{N}$ is a $\Gamma_{E(3; 3; 1, 1, 1)}$-contraction,  it follows from the Theorem \ref{thm-1}  that $\{((N_1 - z_2N_3)(I - z_2N_2)^{-1}, (N_4 - z_2N_6)(I - z_2N_2)^{-1}, (N_5 - z_2N_7)(I - z_2N_2)^{-1}) : z_2 \in \mathbb {D}\}$  is a commuting family of tetrablock contractions. Observe that
\begin{eqnarray*}
(N_5^*-\overline{z}_2N_7^*)(N_5-z_2N_7)&=& N_5^*N_5-\overline{z}_2N_7^*N_5-z_2N_5^*N_7+|z_2|^2N_7*N_7\\
&=&I-\overline{z}_2N_7^*N_5-z_2N_5^*N_7+|z_2|^2I\\
&=&|z_2|^2N_2^*N_2-\overline{z}_2N_7^*N_5-z_2N_5^*N_7+I\\
&=&(I-\overline{z}_2N_2^*)(I-z_2N_2),
\end{eqnarray*}
 therefore  $(N_5 - z_2N_7)(I - z_2N_2)^{-1}$ is a unitary operator for all $z_2$ in $\mathbb D$. Thus by  [Theorem $5.4$,\cite{Bhattacharyya}]
we deduce that \[\{((N_1 - z_2N_3)(I - z_2N_2)^{-1}, (N_4 - z_2N_6)(I - z_2N_2)^{-1}, (N_5 - z_2N_7)(I - z_2N_2)^{-1}) : |z_2| < 1\}\] is a commuting family of tetrablock unitaries.		
		
$(5) \Rightarrow (1):$		
 Suppose that $||N_2|| < 1$, $\|N_4\|<1,$ 
\begin{eqnarray}\label{e1}
\{((N_1 - z_2N_3)(I - z_2N_2)^{-1}, (N_4 - z_2N_6)(I - z_2N_2)^{-1}, (N_5 - z_2N_7)(I - z_2N_2)^{-1}) : |z_2| = 1\}\end{eqnarray}  and  \begin{eqnarray}\label{e2}
\{((N_1 - z_3N_5)(I - z_3N_4)^{-1},  (N_2 - z_3N_6)(I - z_3N_4)^{-1}, (N_3 - z_3N_7)(I - z_3N_4)^{-1}) : |z_3| = 1\}\end{eqnarray} are commuting family of tetrablock unitaries. We first prove that $N_i$'s are normal operator. 
For $z_2\in \mathbb{T}$, it follows from  \eqref{e1} and
[Theorem $5.4$,\cite{Bhattacharyya}] that 
\begin{equation}\label{N-1-N-3}
\begin{aligned}
(I-\bar{z}_2N^*_2)(N_1-z_2N_3)&=(N_4^*-\bar{z}_2N_6^*)(N_5-z_2N_7),\\
(I-\bar{z}_2N^*_2)(N_4-z_2N_6)&=(N_1^*-\bar{z}_2N_3^*)(N_5-z_2N_7),\\
\rm{and}\\
(N_5^*-\bar{z}_2N_7^*)(N_5-z_2N_7)&=(I-\bar{z}_2N_2^*)(I-z_2N_2),\\
(N_5-z_2N_7)(N_5^*-\bar{z}_2N_7^*)&=(I-z_2N_2)(I-\bar{z}_2N_2^*).
\end{aligned}
\end{equation}	
Also, it yields from  \eqref{e2} and
[Theorem $5.4$,\cite{Bhattacharyya}] 
\begin{equation}\label{N-1-N-31}
\begin{aligned}
(I-\bar{z}_3N^*_4)(N_1-z_3N_5)&=(N_2^*-\bar{z}_3N_6^*)(N_3-z_3N_7), \\
(I-\bar{z}_3N^*_4)(N_2-z_3N_6)&=(N_1^*-\bar{z}_3N_5^*)(N_3-z_3N_7).
\end{aligned}
\end{equation}	
Thus, from \eqref{N-1-N-3}, we get
\begin{equation}\label{e3}
\begin{aligned}
N_1+N_2^*N_3=N_4^*N_5+N_6^*N_7, N_3=N_4^*N_7 ~{\rm{and}}~N_2^*N_1=N_6^*N_5,
\end{aligned}
\end{equation}
\begin{equation}\label{e4}
\begin{aligned}
N_4+N_2^*N_6=N_1^*N_5+N_3^*N_7, N_6=N_1^*N_7 ~{\rm{and}}~N_2^*N_4=N_3^*N_5,
\end{aligned}
\end{equation}
\begin{equation}\label{e5}
\begin{aligned}
N_5^*N_5+N_7^*N_7=I+N_2^*N_2, N_2^*=N_7^*N_5,
\end{aligned}
\end{equation}
\begin{equation}\label{e6}
\begin{aligned}
N_5N_5^*+N_7N_7^*=I+N_2N_2^*, N_2^*=N_5N_7^*
\end{aligned}
\end{equation}
We also deduce from \eqref{N-1-N-31} that
\begin{equation}\label{e7}
\begin{aligned}
N_1+N_4^*N_5=N_2^*N_3+N_6^*N_7, N_5=N_2^*N_7 ~{\rm{and}}~N_4^*N_1=N_6^*N_3,
\end{aligned}
\end{equation}

\begin{equation}\label{e8}
\begin{aligned}
N_2+N_4^*N_6=N_1^*N_3+N_5^*N_7, N_6=N_1^*N_7 ~{\rm{and}}~N_4^*N_2=N_5^*N_3.
\end{aligned}
\end{equation} 

From \eqref{e5} and \eqref{e6}, we get $N_2N_2^*=N_7N_5^*N_5N_7^*=N_2^*N_2.$ This indicates that $N_2$ is normal operator.  

From \eqref{e5}, \eqref{e6} and \eqref{e7},  we observe  that 
\begin{equation}\label{n5normal}
N_5^*N_5=N_5^*N_2^*N_7=N_2^*N_5^*N_7=N_2^*N_7N_5^*=N_5N_5^*\end{equation}
and 
\begin{equation}\label{n5n2}
N_5^*N_5=N_5^*N_2^*N_7= N_2^*N_5^*N_7=N_2^*N_7N_5^*=N_2^*N_2.
\end{equation} 
Thus, from \eqref{e5}, \eqref{e6} and \eqref{n5n2}, we conclude that $N_7$ is a unitary operator. We now show that $N_1$ is normal operator. We notice from  \eqref{e4} that $N_1N_6=N_1N_1^*N_7$ and $N_6N_1=N_1^*N_7N_1=N_1^*N_1N_7.$ Since $N_1$ and $N_6$ commute, we have $N_1N_1^*N_7=N_1^*N_1N_7.$ Since $N_7$ is a unitary operator, we deduce that $N_1$ is a normal operator. 
From \eqref{e3} and \eqref{e4}, we have 
\begin{equation}\label{e9}
N_7^*N_3=N_7^*N_4^*N_7=N_4^*\end{equation}
and 
\begin{equation}\label{e10}
N_7^*N_6=N_7^*N_1^*N_7=N_1^*.\end{equation}
 We observe from \eqref{e10} that $N_1N_6=N_6^*N_7N_6=N_6^*N_6N_7$ and $N_6N_1=N_6N_6^*N_7$. As $N_1$ and $N_6$ commute, we get $N_6^*N_6N_7=N_6N_6^*N_7$. Because $N_7$ is a unitary operator, we conclude that $N_6$ is a normal operator. As $N_7$ is a unitary operator and $N_3$ and $N_4$ commute, we deduce from \eqref{e3} and \eqref{e9} that $N_3$ and $N_4$ are normal operators.
Let $\mathcal A$ denote the $C^*$-algebra generated by the commuting normal operators $I, N_1, \dots, N_7$. By the Gelfand-Naimark theorem, the commutative $C^*$-algebra $\mathcal A$ is isometrically $^*$-isomorphic to $C(\sigma(\textbf{N}))$ via the Gelfand map. The \textit{Gelfand map} sends  $N_i$ to the coordinate function $x_i$ for $i = 1, \dots, 7$.  As $||N_2|| < 1$ and $\|N_4\|<1$, it follows that $|x_2|<1$ and $|x_4|<1.$
It remains to check the joint spectrum $\sigma(\textbf{N})$ is contained in $K.$		

Let $(x_1,\dots,x_7)\in \sigma(\textbf{N}).$ Since \[\{((N_1 - z_2N_3)(I - z_2N_2)^{-1}, (N_4 - z_2N_6)(I - z_2N_2)^{-1}, (N_5 - z_2N_7)(I - z_2N_2)^{-1}) : |z_2| = 1\}\]  and  \[\{((N_1 - z_3N_5)(I - z_3N_4)^{-1},  (N_2 - z_3N_6)(I - z_3N_4)^{-1}, (N_3 - z_3N_7)(I - z_3N_4)^{-1}) : |z_3| = 1\}\] are a commuting family of tetrablock unitaries, by [Theorem $5.4$,\cite{Bhattacharyya}], the Taylor joint spectrum of \[\{((N_1 - z_2N_3)(I - z_2N_2)^{-1}, (N_4 - z_2N_6)(I - z_2N_2)^{-1}, (N_5 - z_2N_7)(I - z_2N_2)^{-1}) \}\]  and  \[\{((N_1 - z_3N_5)(I - z_3N_4)^{-1},  (N_2 - z_3N_6)(I - z_3N_4)^{-1}, (N_3 - z_3N_7)(I - z_3N_4)^{-1}) \}\] are contained in $b\Gamma_{E(2;2;1,1)},$  for every $z_2,z_3\in \mathbb T$.  Thus, from the spectral mapping theorem, we have
$$ \tilde{\textbf{z}}^{(z_3)}\in b\Gamma_{E(2;2;1,1)}~\text{and}~\tilde{\textbf{y}}^{(z_2)}\in b\Gamma_{E(2;2;1,1)}~\text{for all}~z_2,z_3\in \mathbb{T}~\text{with}~|x_4|<1 ~\text{and}~|x_2|<1.$$
By Theorem \ref{relation bw}, we conclude that  $(x_1,\dots,x_7)\in K$. This shows that $\textbf{N} = (N_1, \dots, N_7)$ is a $\Gamma_{E(3; 3; 1, 1, 1)}$-unitary. 

 $(5)^{\prime} \Rightarrow (1)$: Suppose that $N_2$ is unitary and \begin{equation}\label{e100}\{((N_1 - z_2N_3)(I - z_2N_2)^{-1}, (N_4 - z_2N_6)(I - z_2N_2)^{-1}, (N_5 - z_2N_7)(I - z_2N_2)^{-1}) : |z_2| < 1\}\end{equation} is a commuting family of tetrablock unitaries. We first show that $N_i$'s are normal operator. 
For $z_2\in \mathbb{D}$, it implies from  \eqref{e100} and
[Theorem $5.4$,\cite{Bhattacharyya}] that 
\begin{equation}\label{11N-1-N-3}
\begin{aligned}
(I-\bar{z}_2N^*_2)(N_1-z_2N_3)&=(N_4^*-\bar{z}_2N_6^*)(N_5-z_2N_7),\\
(I-\bar{z}_2N^*_2)(N_4-z_2N_6)&=(N_1^*-\bar{z}_2N_3^*)(N_5-z_2N_7),\\
(N_5^*-\bar{z}_2N_7^*)(N_5-z_2N_7)&=(I-\bar{z}_2N_2^*)(I-z_2N_2),\\
(N_5-z_2N_7)(N_5^*-\bar{z}_2N_7^*)&=(I-z_2N_2)(I-\bar{z}_2N_2^*).
\end{aligned}
\end{equation}	
Therefore, from \eqref{11N-1-N-3}, we get
\begin{equation}\label{e31}
\begin{aligned}
N_1=N_4^*N_5, N_2^*N_3=N_6^*N_7, N_3=N_4^*N_7 ~{\rm{and}}~N_2^*N_1=N_6^*N_5,
\end{aligned}
\end{equation}
\begin{equation}\label{e41}
\begin{aligned}
N_4=N_1^*N_5, N_2^*N_6=N_3^*N_7, N_6=N_1^*N_7 ~{\rm{and}}~N_2^*N_4=N_3^*N_5,
\end{aligned}
\end{equation}
\begin{equation}\label{e51}
\begin{aligned}
N_7^*N_7=N_2^*N_2, N_5^*N_5=I,N_2^*=N_7^*N_5,
\end{aligned}
\end{equation}
\begin{equation}\label{e61}
\begin{aligned}
N_7N_7^*=N_2N_2^*, N_5N_5^*=I,N_2^*=N_5N_7^*
\end{aligned}
\end{equation}
Since $N_2$ is unitary, it follows from \eqref{e51} and \eqref{e61} that $N_5$ and $N_7$ are also unitaries. Furthermore, because $N_7$ is unitary, it indicates from \eqref{e31}, \eqref{e41}, \eqref{e51} and \eqref{e61} that $N_1=N_6^*N_7,N_4=N_3^*N_7$ and  $N_5=N^*_2N_7.$ By using the commutativity of $N_1,N_3,N_4$ and $N_6,$ we conclude that $N_1,N_3,N_4$ and $N_6$ are normal operators.
Let $\mathcal A$ represent the $C^*$-algebra generated by the commuting normal operators $I, N_1, \dots, N_7$. By the Gelfand-Naimark theorem, the commutative $C^*$-algebra $\mathcal A$ is isometrically $^*$-isomorphic to $C(\sigma(\textbf{N}))$ via the Gelfand map. The \textit{Gelfand map} sends  $N_i$ to the coordinate function $x_i$ for $i = 1, \dots, 7$.  Since $N_2$ is unitary, it implies that $|x_2|=1.$ In order to complete the proof, it remains to check the joint spectrum $\sigma(\textbf{N})$ is contained in $K.$		

Let $(x_1,\dots,x_7)\in \sigma(\textbf{N}).$ Since \[\{((N_1 - z_2N_3)(I - z_2N_2)^{-1}, (N_4 - z_2N_6)(I - z_2N_2)^{-1}, (N_5 - z_2N_7)(I - z_2N_2)^{-1}) : |z_2| < 1\}\] is a commuting family of tetrablock unitary, by [Theorem $5.4$,\cite{Bhattacharyya}], the Taylor joint spectrum of \[\{((N_1 - z_2N_3)(I - z_2N_2)^{-1}, (N_4 - z_2N_6)(I - z_2N_2)^{-1}, (N_5 - z_2N_7)(I - z_2N_2)^{-1}) \}\]   is contained in $b\Gamma_{E(2;2;1,1)},$  for every $z_2\in \mathbb D$.  Thus, from the spectral mapping theorem, we have
$$ \tilde{\textbf{z}}^{(z_3)}\in b\Gamma_{E(2;2;1,1)}~\text{for all}~z_2\in \mathbb{D}~\text{with}~|x_2|=1.$$
Hence, from Theorem \ref{relation bw}, we deduce that  $(x_1,\dots,x_7)\in K$. This shows that $\textbf{N} = (N_1, \dots, N_7)$ is a $\Gamma_{E(3; 3; 1, 1, 1)}$-unitary. Similarly, we can also show that  $(5)^{\prime \prime} \Rightarrow (1)$. This completes the proof.

		\end{proof}
As a consequence of the above Theorem, we state the following corollary. 
	
	\begin{cor}\label{prop-6}
		Let $\textbf{N} = (N_1, \dots, N_7)$ be a $7$-tuples of commuting bounded operators. Then the followings are equivalent:
		\begin{enumerate}
			\item $\textbf{N}$ is a $\Gamma_{E(3; 3; 1, 1, 1)}$-unitary.
			
			\item $(\omega N_1, N_2, \omega N_3, \omega N_4, \omega^2 N_5, \omega N_6, \omega^2 N_7)$ is a family of $\Gamma_{E(3; 3; 1, 1, 1)}$-unitaries for all $\omega \in \mathbb{T}.$
			
			\item[$(2^{'})$] $(\omega N_1, \omega N_2, \omega^2 N_3, N_4, \omega N_5, \omega N_6, \omega^2 N_7)$ is a family of $\Gamma_{E(3; 3; 1, 1, 1)}$-unitaries for all $\omega \in \mathbb{T}.$
			
			\item[$(2^{''})$] $(N_1, \omega N_2, \omega N_3, \omega N_4, \omega N_5, \omega^2 N_6, \omega^2 N_7)$ is a family of $\Gamma_{E(3; 3; 1, 1, 1)}$-unitaries for all $\omega \in \mathbb{T}.$
		\end{enumerate}
	\end{cor}
\begin{thm}\label{unita}
Let  $\textbf{U} = ((U_{ij}))$ be a $3 \times 3$ unitary block operator matrix, where $U_{ij}$ are commuting normal operators. Set 			\small{\begin{equation*}
				\begin{aligned}
					\textbf{N} &= (N_1=U_{11}, N_2=U_{22}, N_3=U_{11}U_{22} - U_{12}U_{21}, N_4=U_{33}, N_5=U_{11}U_{33} - U_{13}U_{31},\\&\,\,\,\,\,\,N_6=U_{22}U_{33} - U_{23}U_{32}, N_7=
					 U_{11}(U_{22}U_{33} - U_{23}U_{32}) - U_{12}(U_{21}U_{33} - U_{31}U_{23}) + U_{13}(U_{21}U_{32} - U_{31}U_{22})).
				\end{aligned}
\end{equation*}}
Then $\textbf{N}$ is a $\Gamma_{E(3; 3; 1, 1, 1)}$-unitary.
\end{thm}	
\begin{proof}
By Theorem \ref{thm-5}, it is sufficient to show that $N_7$ is unitary and $N_1=N_6^*N_7, N_2=N_5^*N_7~{\rm{and}}~N_3=N_4^*N_7.$  We first observe that since $\textbf{U}$ is a unitary, it follows that  $$\left( \begin{smallmatrix} I & 0 &0\\0 & I &0\\0 & 0 & I\end{smallmatrix}\right)=\textbf{U}^*\textbf{U}=\textbf{U}\textbf{U}^*$$  Consequently, we have
		\begin{equation}\label{Eq}
			\begin{aligned}
				&U^*_{11}U_{11} + U^*_{21}U_{21} + U^*_{31}U_{31} = I,\\
				&U^*_{12}U_{12} + U^*_{22}U_{22} + U^*_{32}U_{32} = I,\\
				&U^*_{13}U_{13} + U^*_{23}U_{23} + U^*_{33}U_{33}= I,
			\end{aligned}
		\end{equation}
		and
		\begin{equation}\label{Eq'}
			\begin{aligned}
				&U^*_{11}U_{12} + U^*_{21}U_{22} + U^*_{31}U_{32} = 0,\\
				&U^*_{11}U_{13} + U^*_{21}U_{23} + U^*_{31}U_{33} = 0,\\
				&U^*_{12}U_{13} + U^*_{22}U_{23} + U^*_{32}U_{33} = 0.
			\end{aligned}
		\end{equation}
and 	\begin{equation}\label{Eq1}
			\begin{aligned}
				&U_{11} U^*_{11}+ U_{12}U^*_{12} + U_{13}U^*_{13} = I,\\
				&U_{21}U^*_{21} + U_{22}U^*_{22} + U_{23}U^*_{23} = I,\\
				&U_{31}U^*_{31} + U_{32}U^*_{32} + U_{33}U^*_{33}= I,
			\end{aligned}
		\end{equation}
		and
		\begin{equation}\label{Eq'1}
			\begin{aligned}
				&U_{11}U^*_{21} + U_{12}U^*_{22} + U_{13}U^*_{23} = 0,\\
				&U_{11}U^*_{31} + U_{12}U^*_{32} + U_{13}U^*_{33} = 0,\\
				&U_{21}U^*_{31} + U_{22}U^*_{32} + U_{23}U^*_{33} = 0.
			\end{aligned}
		\end{equation}
We first show that $N_7$ is unitary. Set $A=(U_{21}U_{33} - U_{31}U_{23})$ and $B=(U_{21}U_{32} - U_{31}U_{22})).$ Note that 
\begin{align}\label{unitaryN7}
N^*_7N_7\nonumber&=(N^*_6U_{11}^*-A^*U_{12}^*+B^*U_{13}^*)(U_{11}N_6-U_{12}A+U_{13}B)\\\nonumber&=N^*_6U_{11}^*(U_{11}N_6-U_{12}A+U_{13}B)-A^*U_{12}^*(U_{11}N_6-U_{12}A+U_{13}B)\\&+B^*U_{13}^*(U_{11}N_6-U_{12}A+U_{13}B)
\end{align}
 By using \eqref{Eq}, \eqref{Eq'}, \eqref{Eq1} and \eqref{Eq'1}, it follows that
\begin{equation}
\begin{aligned} 
N^*_6U_{11}^*(U_{11}N_6-U_{12}A+U_{13}B)&=N^*_6N_6\\
A^*U_{12}^*(U_{11}N_6-U_{12}A+U_{13}B)&=-A^*A\\
B^*U_{13}^*(U_{11}N_6-U_{12}A+U_{13}B)&=B^*B
\end{aligned}
\end{equation}
From  \eqref{Eq}, \eqref{Eq'}, \eqref{Eq1} and \eqref{Eq'1}, we also notice that 
\begin{equation}\label{identity}
N^*_6N_6+A^*A+B^*B=I.
\end{equation}
Thus, from \eqref{unitaryN7} we conclude that $N^*_7N_7=I.$ Similarly, we can also show that $N_7N_7^*=I.$ We now show that $N_6=N_1^*N_7.$ Note that 
\begin{align}\label{n-1n-7}
N^*_1N_7&\nonumber=U_{11}^*U_{11}N_6+U_{11}^*U_{12}A+U_{11}^*U_{13}B\\&=N_6-U_{21}^*U_{21}N_6-U_{31}^*U_{31}N_6+U_{11}^*U_{12}A+U_{11}^*U_{13}B
\end{align}
By  \eqref{Eq}, \eqref{Eq'}, \eqref{Eq1} and \eqref{Eq'1}, we deduce that 
\begin{equation}\label{n6=n_1n_7}
-U_{21}^*U_{21}N_6-U_{31}^*U_{31}N_6+U_{11}^*U_{12}A+U_{11}^*U_{13}B=0
\end{equation}
Thus, from \eqref{n-1n-7} and \eqref{n6=n_1n_7}, we conclude that $N^*_1N_7=N_6.$ Similarly, we can also show that $N_5=N_2^*N_7$ and $N_4=N_3^*N_7.$ This completes the proof.

\end{proof}

%
%
%
	
We now study the various characterization of $\Gamma_{E(3; 2; 1, 2)}$-unitary. The following theorem plays an important role for characterizing  $\Gamma_{E(3; 2; 1, 2)}$-unitary.
	
\begin{thm}\label{thm-6}
		Let $\tilde{\textbf{x}} = (x_1, x_2, x_3, y_1, y_2) $ be in $\mathbb{C}^5$. Then the following are equivalent:
		\begin{enumerate}
			\item $\tilde{\textbf{x}} \in K_1 = \{\tilde{\textbf{x}} \in \mathbb{C}^5 : x_1 = \overline{y}_2 x_3, x_2 = \overline{y}_1 x_3, |x_3| = 1 \}.$
			
			\item 
			$\begin{cases}
				\Big(\frac{y_1 - zx_2}{1 - zx_1}, \frac{y_2 - zx_3}{1 - zx_1}\Big) \in b\Gamma_{E(2; 1; 2)} &\text{for all}~ z \in \mathbb{D} ~\text{with}~  |x_1| = 1\\
				\Big(\frac{y_1 - zx_2}{1 - zx_1}, \frac{y_2 - zx_3}{1 - zx_1}\Big) \in b\Gamma_{E(2; 1; 2)} &\text{for all}~ z \in \mathbb{T}~\text{with}~   |x_1| < 1 ~\text{and}~|x_3|=1.
			\end{cases}$
		\end{enumerate}
	\end{thm}
	
	\begin{proof}
We first show that $(1) \Rightarrow (2).$ In order to show $(1) \Rightarrow (2),$ we take into consideration the following cases:	
\newline\textbf{Case $1$:} We assume that  $|x_1| < 1.$  Then  $1 - zx_1 \neq 0$ for any $z \in \mathbb{T}.$  By [Theorem $2.4$,\cite{JAgler}], it follows that  a point $(s,p) \in b\Gamma_{E(2; 1; 2)}$ if and only if $s=\bar{s}p$ and $|p|=1.$ Because $\tilde{\textbf{x}} \in K_1$, we note that 		\begin{equation*}
			\begin{aligned}
				\Big|\frac{y_2 - zx_3}{1 - zx_1}\Big|
				&=\Big |\frac{x_3(\bar{x}_1-z)}{1 - zx_1}\Big| \\&=1
							\end{aligned}\end{equation*}
							and
\begin{equation*}
	\begin{aligned}
				\overline{\Big(\frac{y_2 - zx_3}{1 - zx_1}\Big)}\frac{y_1 - zx_2}{1 - zx_1}
				&= \frac{\overline{y}_2y_1 - \overline{z}\overline{x}_3y_1 - z\overline{y}_2x_2 + |z|^2\overline{x}_3x_2}{(1 - zx_1)(1 - \overline{z}\overline{x}_1)}\\
				&= \frac{x_1\overline{x}_2 - \overline{z}\overline{x}_2 - zx_1\overline{x}_3\overline{y}_1x_3+ |z|^2\overline{y}_1}{(1 - zx_1)(1 - \overline{z}\overline{x}_1)}\\
				&= \frac{\overline{y}_1 - \overline{z}\overline{x}_2 - zx_1\overline{y}_1 + |z|^2x_1\overline{x}_2}{(1 - zx_1)(1 - \overline{z}\overline{x}_1)}\\
				&= \frac{(1 - zx_1)(\overline{y}_1 - \overline{z}\overline{x}_2)}{(1 - zx_1)(1 - \overline{z}\overline{x}_1)}\\
				&= \overline{\Big(\frac{y_1 - zx_2}{1 - zx_1}\Big)}
			\end{aligned}
		\end{equation*}
This shows that  $\Big(\frac{y_1 - zx_2}{1 - zx_1}, \frac{y_2 - zx_3}{1 - zx_1}\Big) \in b\Gamma_{E(2; 1; 2)}$  for all $|z| = 1$ with for $|x_1| < 1$.
\newline \textbf{Case $2:$} Assume that  $|x_1| = 1$ then for any $z \in \mathbb{D}, 1 - zx_1 \ne 0$. By using the similar argument as in  Case $1$, we conclude that $\Big(\frac{y_1 - zx_2}{1 - zx_1}, \frac{y_2 - zx_3}{1 - zx_1}\Big) \in b\Gamma_{E(2; 1; 2)}$  for all $ z \in \mathbb{D} $ with $ |x_1| = 1$.

We now demonstrate that $(2) \Rightarrow (1)$. To show this, we need to analyse the following cases:
\newline\textbf{Case $1:$} Assume that $\Big(\frac{y_1 - zx_2}{1 - zx_1}, \frac{y_2 - zx_3}{1 - zx_1}\Big) \in b\Gamma_{E(2; 1; 2)}$ for all $z\in \mathbb T$ with $|x_1| < 1$  and $|x_3|=1.$ By [Theorem $2.4$,\cite{JAgler}], we have  \begin{equation}
			\begin{aligned}\label{gamma}
				\overline{\Big(\frac{y_2 - zx_3}{1 - zx_1}\Big)}\frac{y_1 - zx_2}{1 - zx_1} = \overline{\Big(\frac{y_1 - zx_2}{1 - zx_1}\Big)}~{\rm{and}}~1 = \Big|\frac{y_2 - zx_3}{1 - zx_1}\Big|~{\rm{for~ all}} ~z \in \mathbb T.
			\end{aligned}
		\end{equation}
From \eqref{gamma},  we deduce that 
$x_1 = \overline{y}_2 x_3, x_2 = \overline{y}_1 x_3.$
\newline \textbf{Case $2:$} Suppose that  $|x_1| = 1,$ then for any $z \in \mathbb{D}, 1 - zx_1 \ne 0$. By using the similar argument as in  Case $1$, we conclude that $\tilde{\textbf{x}} \in K_1.$ This completes the proof.
%
%
	\end{proof}
	
	The proof of the following 	Proposition is straightforward. Therefore, we skip the proof.
	\begin{prop}\label{cor-3}
		Let $\textbf{x} = (x_1, x_2, x_3, y_1, y_2) \in \mathbb{C}^5$. Then the following are equivalent:
		\begin{enumerate}
			\item $\textbf{x} \in K_1$.
			
			\item For all $\omega \in \mathbb{T}, (x_1, \omega x_2, \omega^2x_3, \omega y_1, \omega^2 y_2) \in K_1$.
		\end{enumerate}
	\end{prop}
	
The following theorem provides the characterization of $\Gamma_{E(3; 2; 1, 2)}$-unitaries. 

	\begin{thm}\label{thm-7}
		Let $\textbf{M} = (M_1, M_2, M_3, \tilde{M}_1, \tilde{M}_2)$ be a commuting  $5$-tuple of bounded operators on a Hilbert space $\mathcal{H}$. Then the following are equivalent:
		\begin{enumerate}
			\item $\textbf{M}$ is a $\Gamma_{E(3; 2; 1, 2)}$-unitary.
			
			\item $M_3$ is a unitary, $M_1, \frac{M_2}{2}, \frac{\tilde{M}_1}{2}~{\rm{and}}~\tilde{M}_2$ are contractions and $M_1 = \tilde{M}^*_2M_3, M_2 = \tilde{M}^*_1M_3$.
			
\item $(M_1, \tilde{M}_2, M_3), ( \frac{\tilde{M}_1}{2},\frac{M_2}{2},  M_3)$ and $(\frac{M_2}{2}, \frac{\tilde{M}_1}{2}, M_3)$ are $\Gamma_{E(2; 2; 1,1)}$-unitary.		
			
			\item $M_3$ is a unitary and $\textbf{M}$ is a $\Gamma_{E(3; 2; 1, 2)}$-contraction.
			
\item $||\tilde{M}_1|| < 2$ and $ \{ ((2M_1 - zM_2)(2I - z\tilde{M}_1)^{-1}, (\tilde{M}_1 - 2z\tilde{M}_2)(2I - z\tilde{M}_1)^{-1}, (M_2 - 2zM_3)(2I - z\tilde{M}_1)^{-1}) : |z| = 1 \}$ is a family of tetrablock unitaries.
				
\item[$(5^{\prime})$] $||\tilde{M}_1|| = 2$ and $ \{ ((2M_1 - zM_2)(2I - z\tilde{M}_1)^{-1}, (\tilde{M}_1 - 2z\tilde{M}_2)(2I - z\tilde{M}_1)^{-1}, (M_2 - 2zM_3)(2I - z\tilde{M}_1)^{-1}) : |z| < 1 \}$ is a family of tetrablock unitaries.
			\end{enumerate}
	\end{thm}
	
	\begin{proof}
We show the equivalence of $(1), (2), (3),(4)$ and $(5)$ by proving the following relationships: \small{$$(1) \Leftrightarrow (2) \Leftrightarrow (3), (1) \Rightarrow (4),(4)  \Rightarrow (3), (4)  \Rightarrow (5), (5) \Rightarrow(1), (4)  \Rightarrow (5^{\prime})~{\rm{and}}~ (5^{\prime}) \Rightarrow(1).$$}

Equivalence of $(1)$ and $(2)$ follows by using the same methods as in Theorem \ref{thm-5}. Therefore, we skip the proof.

$(2) \Leftrightarrow (3):$ 
The equivalence of $(2)$ and $(3)$ is deduced from [Theorem $5.4$, \cite{Bhattacharyya}].	
		
$(1) \Rightarrow (4):$ Assume that $\textbf{M}$ is a $\Gamma_{E(3; 2; 1, 2)}$-unitary. By equivalence of $(1)$ and $(2)$, we deduce that $M_3$ is a unitary. Let $p$ be a polynomial of five variables. As $M_1, M_2, M_3, \tilde{M}_1, \tilde{M}_2$ are normal operators, it yields that $p (M_1, M_2, M_3, \tilde{M}_1, \tilde{M}_2)$ is normal operator. Then 
\begin{align*} \|p (M_1, M_2, M_3, \tilde{M}_1, \tilde{M}_2)\|&=\sup\{|p(\lambda_1,\dots,\lambda_5)|: (\lambda_1,\dots,\lambda_5)\in \sigma(\textbf{M}) \}\\&\leq\sup\{|p(\lambda_1,\dots,\lambda_5)|: (\lambda_1,\dots,\lambda_5)\in K_1\}\\&\leq\sup\{|p(\lambda_1,\dots,\lambda_5)|: (\lambda_1,\dots,\lambda_5)\in \Gamma_{E(3; 2; 1, 2)}\}\\
&= \|p\|_{\infty, \Gamma_{E(3; 2; 1, 2)}}.
\end{align*}
This demonstrates that $\textbf{M}$ is a $\Gamma_{E(3; 2; 1, 2)}$-contraction.

$(4) \Rightarrow (3):$  Suppose that  $\textbf{M}$ is a $\Gamma_{E(3; 2; 1, 2)}$-contraction. By Proposition \ref{prop-41}, it follows that $(M_1, \tilde{M}_2, M_3),$ $ ( \frac{\tilde{M}_1}{2},\frac{M_2}{2},  M_3)$ and $(\frac{M_2}{2}, \frac{\tilde{M}_1}{2}, M_3)$ are $\Gamma_{E(2; 2; 1,1)}$-contraction. As $M_3$ is a unitary, it yields from [Theorem $5.4$,\cite{Bhattacharyya}] that $(M_1, \tilde{M}_2, M_3), ( \frac{\tilde{M}_1}{2},\frac{M_2}{2},  M_3)$ and $(\frac{M_2}{2}, \frac{\tilde{M}_1}{2}, M_3)$ are $\Gamma_{E(2; 2; 1,1)}$-unitaries.

$(4) \Rightarrow (5):$ Assume that $M_3$ is a unitary and $\textbf{M}$ is a $\Gamma_{E(3; 2; 1, 2)}$-contraction. By equivalence of $(1),(2)$ and $(4)$, we note that $M_1, M_2, M_3, \tilde{M}_1, \tilde{M}_2$ are normal operators,  $M_1, \frac{M_2}{2}, \frac{\tilde{M}_1}{2}~{\rm{and}}~\tilde{M}_2$ are contractions and $M_1 = \tilde{M}^*_2M_3, M_2 = \tilde{M}^*_1M_3$.  Let $\mathcal A_1$ be the $C^*$-algebra generated by the commuting normal operators $I,M_1, M_2, M_3, \tilde{M}_1, \tilde{M}_2$. By \textit{Gelfand-Naimark theorem},  $\mathcal A_1$ is isometrically $^*$-isomorphic to $C(\sigma(\textbf{M}))$ via the \textit{Gelfand map.}
		
 Let $\tilde{\textbf{x}}=(x_1,x_2,x_3,y_1,y_2)\in \sigma(\textbf{M}).$ As the Taylor joint spectrum $\sigma(\textbf{M})$ is contained in the set  $ K_1$,  it follows from Theorem \ref{relation bw1} that $\tilde{\textbf{x}}\in K_1$ if and only if $$(p_1(z),p_2(z),p_3(z))\in b\Gamma_{E(2;2;1,1)}~\text{for ~all}~z\in \mathbb{T}~\text{with}~|y_1|<2.$$  The inverse of the \textit{Gelfand map} takes the co-ordinate functions $x_j$ to $M_j$ for $j=1,2,3$ and $y_i$ to $\tilde{M}_i$ for $i=1,2.$ Thus, we have $\|\tilde{M}_1\|<2.$ 
As $\textbf{M}$ is a $\Gamma_{E(3; 2; 1, 2)}$-contraction, by Theorem \ref{thm-3} and Remark \ref{rem-5} imply that  $$ \{ ((2M_1 - zM_2)(2I - z\tilde{M}_1)^{-1}, (\tilde{M}_1 - 2z\tilde{M}_2)(2I - z\tilde{M}_1)^{-1}, (M_2 - 2zM_3)(2I - z\tilde{M}_1)^{-1}) : z\in \mathbb {\bar{D}} \}$$ is a family of tetrablock  contractions. In particular, $$ \{ ((2M_1 - zM_2)(2I - z\tilde{M}_1)^{-1}, (\tilde{M}_1 - 2z\tilde{M}_2)(2I - z\tilde{M}_1)^{-1}, (M_2 - 2zM_3)(2I - z\tilde{M}_1)^{-1}) : |z|=1 \}$$ is a family of tetrablock contractions. In order to complete the proof, by  [Theorem $5.4$,\cite{Bhattacharyya}], we need to show that $(M_2 - 2zM_3)(2I - z\tilde{M}_1)^{-1}$ are unitary for all $z\in \mathbb T$. Observe that $M_2 = \tilde{M}^*_1M_3, \tilde{M}_1=M_2^*M_3,$
\begin{align}\label{n-21}
 M_2^*M_2\nonumber&=M_3^*\tilde{M}_1 \tilde{M}^*_1M_3\\\nonumber &=M_3^* \tilde{M}^*_1\tilde{M}_1M_3\\\nonumber &=\tilde{M}_1^*M_3^*M_3\tilde{M}_1\\&=\tilde{M}_1^*\tilde{M}_1
\end{align}  	
and 	
\begin{align}\label{n-22351}
(M^*_2 - 2\overline{z}M^*_3)(M_2 - 2zM_3)\nonumber&=(M^*_2M_2 - 2\overline{z}M^*_3M_2 - 2zM^*_2M_3 + 4M^*_3M_3)\\\nonumber&=(4I- 2\overline{z}\tilde{M}^*_1-2z\tilde{M}_1+\tilde{M}_1^*\tilde{M}_1)\\&=(2I - z\tilde{M}_1) (2I - \overline{z}\tilde{M}^*_1).
\end{align}	
It follows from \eqref{n-21} and \eqref{n-22351} that $(M_2 - 2zM_3)(2I - z\tilde{M}_1)^{-1}$ is a unitary for all $z\in \mathbb T$. 

$(5) \Rightarrow (1):$ Suppose that $||\tilde{M}_1|| < 2$ and $$ \{ ((2M_1 - zM_2)(2I - z\tilde{M}_1)^{-1}, (\tilde{M}_1 - 2z\tilde{M}_2)(2I - z\tilde{M}_1)^{-1}, (M_2 - 2zM_3)(2I - z\tilde{M}_1)^{-1}) : |z| = 1 \}$$ is a family of tetrablock unitaries. To demonstrate that $\textbf{M}$ is a $\Gamma_{E(3; 2; 1, 2)}$-unitary,  we first that establish that $M_1, M_2, M_3, \tilde{M}_1, \tilde{M}_2$ are normal operators.  Since $$\{ ((2M_1 - zM_2)(2I - z\tilde{M}_1)^{-1}, (\tilde{M}_1 - 2z\tilde{M}_2)(2I - z\tilde{M}_1)^{-1}, (M_2 - 2zM_3)(2I - z\tilde{M}_1)^{-1}) : |z| = 1 \}$$ is a family of tetrablock unitaries, by [Theorem $5.4$,\cite{Bhattacharyya}], we have
\begin{equation}\label{N-1-N-322}
\begin{aligned}
(2I-\bar{z}\tilde{M}_1^*)(2M_1-zM_2)&=(\tilde{M}_1^*-2\bar{z}\tilde{M}_2^*)(M_2-2zM_3), \\
(2I-\bar{z}\tilde{M}_1^*)(\tilde{M}_1-2z\tilde{M}_2)&=(2M_1^*-\bar{z}M_2^*)(M_2-2zM_3)\\
\end{aligned}
\end{equation}	
and
\begin{equation}\label{N-1-N-3122}
\begin{aligned}
(M_2^*-2\bar{z}M_3^*)(M_2-2zM_3)= (2I-\bar{z}\tilde{M}_1^*)(2I-z\tilde{M}_1),\\
(M_2-2zM_3)(M_2^*-2\bar{z}M_3^*)=(2I-z\tilde{M}_1)(2I-\bar{z}\tilde{M}_1^*),
\end{aligned}
\end{equation}	
 for all $z\in \mathbb T.$ Therefore, it follows  from \eqref{N-1-N-322}  that
\begin{equation}\label{N-1zN-223}
\begin{aligned}
M_1=\tilde{M}_2^*M_3, \tilde{M}_1^*M_1=\tilde{M}_2^*M_2 ~{\rm{and}}~M_2=\tilde{M}_1^*M_3,
\end{aligned}
\end{equation} and
\begin{equation}\label{eq112}
\tilde{M}_1+\tilde{M}_1^*\tilde{M}_2=M_1^*M_2+M_2^*M_3, \tilde{M}_2=M_1^*M_3 ~{\rm{and}}~\tilde{M}_1^*\tilde{M}_1=M_2^*M_2,
\end{equation}
Also, from \eqref{N-1-N-3122}, we have
\begin{equation}\label{N-1zN-2231}
\begin{aligned}
M_2^*M_2+4M_3^*M_3=4I+\tilde{M}_1^*\tilde{M}_1, \tilde{M}_1=M_2^*M_3,
\end{aligned}
\end{equation} 
and 
\begin{equation}\label{eq122}
M_2M_2^*+4M_3M_3^*=4I+\tilde{M}_1\tilde{M}_1^*, \tilde{M}_1=M_3M_2^*.
\end{equation}
It implies from \eqref{N-1zN-223}, \eqref{N-1zN-2231} and \eqref{eq122} that 

\begin{equation}\label{M22}M_2M_2^*=\tilde{M}_1^*M_3M_2^*=\tilde{M}_1^*M_2^*M_3=M_2^*\tilde{M}_1^*M_3=M_2^*M_2.\end{equation}  From \eqref{N-1zN-2231} and \eqref{eq122},  it yields  that 

\begin{equation}\label{m1normal}\tilde{M}_1\tilde{M}_1^*=M_3M_2^*M_3^*M_2=M_3M_3^*M_2^*M_2=M_3M_3^*M_2M_2^*=M_3\tilde{M}_1^*M_2^*=M_3M_2^*\tilde{M}_1^*=\tilde{M}_1\tilde{M}_1^*.\end{equation} 
Since $\tilde{M}_1$, $M_2$ are normal operators, so from \eqref{eq112}, \eqref{N-1zN-2231} and \eqref{eq122}, we deduce  that $M_3$ is a unitary operator. We now demonstrate that $M_1$ is a normal operator. We observe from \eqref{eq112} that \begin{equation}\tilde{M}_2M_1=M_1^*M_3M_1=M_1^*M_1M_3~{\rm{ and }}~M_1\tilde{M}_2=M_1M_1^*M_3.\end{equation} Since $M_1$ and $\tilde{M}_2$ commute, it follows that $M_1^*M_1M_3 = M_1M_1^*M_3.$ Since $M_3$ is a unitary, so we get $M_1^*M_1=M_1M_1^*$. Similarly, we can also show that  $\tilde{M}_2$ is a normal operator. Let $\mathcal A_1$ denote the $C^*$-algebra generated by the commuting normal operators $I,M_1, M_2, M_3, \tilde{M}_1, \tilde{M}_2$. The \textit{Gelfand-Naimark theorem} states that  the commutative $C^*$-algebra $\mathcal A_1$ is isometrically $^*$-isomorphic to $C(\sigma(\textbf{M}))$ via the \textit{Gelfand map.} The \textit{Gelfand map} sends  $M_i$ to the coordinate function $x_i$ for $i = 1, 2,3$ and $\tilde{M}_i$ to the coordinate function $y_i$ for $i = 1, 2.$ As  $||\tilde{M}_1|| < 2$, it implies that $|y_1|<2$. To complete the proof, we need to verify that the joint spectrum $\sigma(\textbf{M})$ is contained in $K_1.$ The proof of this claim depends on the Theorem \ref{relation bw1}, which states that $(x_1,x_2,x_3,y_1,y_2)\in K_1$ if and only if  $$(p_1(z),p_2(z),p_3(z))\in b\Gamma_{E(2;2;1,1)}~\text{for ~all}~z\in \mathbb{T}~\text{with}~|y_1|<2,$$ where $p_1(z)=\frac{2x_1-zx_2}{2-y_1z},p_2(z)=\frac{y_1-2zy_2}{2-y_1z}~{\rm{and}}~p_3(z)=\frac{x_2-2zx_3}{2-y_1z}.$	 Let $(x_1,x_2,x_3,y_1,y_2)\in \sigma(\textbf{M}).$ Since $$ \{ ((2M_1 - zM_2)(2I - z\tilde{M}_1)^{-1}, (\tilde{M}_1 - 2z\tilde{M}_2)(2I - z\tilde{M}_1)^{-1}, (M_2 - 2zM_3)(2I - z\tilde{M}_1)^{-1}) : |z| = 1 \}$$ is a commuting family of tetrablock unitaries, the Taylor joint spectrum of $$  ((2M_1 - zM_2)(2I - z\tilde{M}_1)^{-1}, (\tilde{M}_1 - 2z\tilde{M}_2)(2I - z\tilde{M}_1)^{-1}, (M_2 - 2zM_3)(2I - z\tilde{M}_1)^{-1}) $$ is contained in $b\Gamma_{E(2;2;1,1)}$ for every $z\in \mathbb T$ with $\|\tilde{M}_1\|<2$.  Thus, from spectral mapping theorem, we have
$$(p_1(z),p_2(z),p_3(z))\in b\Gamma_{E(2;2;1,1)}~\text{for ~all}~z\in \mathbb{T}~\text{with}~|y_1|<2.$$
By Theorem \ref{relation bw1}, we conclude that  $(x_1,x_2,x_3,y_1,y_2)\in K_1$. This shows that $\textbf{M} $ is a $\Gamma_{E(3; 2; 1, 2)}$-unitary. 		
		
We now prove that $(4) \Rightarrow (5^{\prime}).$ Let $\tilde{\textbf{x}}=(x_1,x_2,x_3,y_1,y_2)\in \sigma(\textbf{M}).$  Since the Taylor joint spectrum $\sigma(\textbf{M})$ is contained in the set  $ K_1$,  it follows from Theorem \ref{relation bw1} that  $$(p_1(z),p_2(z),p_3(z))\in b\Gamma_{E(2;2;1,1)}~\text{for~all}~z\in \mathbb{D}~\text{with}~|y_1|=2.$$ Since $|y_1|=2$ and $|x_3|=1,$ we have $|x_2|=2.$   Since the commutative $C^*$-algebra $\mathcal A$ is isometrically $^*$-isomorphic to $C(\sigma(\textbf{M}))$ via the \textit{Gelfand map}, we conclude that $\frac{M_2}{2}$, $M_3$ and $\frac{\tilde{M}_1}{2}$ are unitary operators, respectively. Because $\textbf{M}$ is a $\Gamma_{E(3; 2; 1, 2)}$-contraction,  it yields  from the Theorem \ref{thm-3}  that $$\{ ((2M_1 - zM_2)(2I - z\tilde{M}_1)^{-1}, (\tilde{M}_1 - 2z\tilde{M}_2)(2I - z\tilde{M}_1)^{-1}, (M_2 - 2zM_3)(2I - z\tilde{M}_1)^{-1}) : |z| < 1 \}$$ is a family of tetrablock contractions.
 For $z\in \mathbb D,$ we have
\begin{align}\label{n-223512}
(M^*_2 - 2\overline{z}M^*_3)(M_2 - 2zM_3)\nonumber&=(M^*_2M_2 - 2\overline{z}M^*_3M_2 - 2zM^*_2M_3 + 4|z|^2M^*_3M_3)\\
\nonumber&=(4I- 2\overline{z}\tilde{M}^*_1-2z\tilde{M}_1+4|z|^2I) \\
\nonumber&=(4I- 2\overline{z}\tilde{M}^*_1-2z\tilde{M}_1+|z|^2\tilde{M}_1^*\tilde{M}_1) \\
&= (2I - \overline{z}\tilde{M}^*_1) (2I - z\tilde{M}_1).
\end{align}
Thus, we conclude that
$(M_2 - 2zM_3)(2I - z\tilde{M}_1)^{-1}$ is a family of unitary operators for all $z\in \mathbb D$.

The proof of $(5^{\prime}) \Rightarrow (1)$ is obtained by using argument similar to the proof of  $(5) \Rightarrow (1).$ This completes the proof.

	\end{proof}
As a consequence of the above Theorem,  we state the following corollary. The proof of the corollary is easy to verify. Therefore, we skip the proof.
	\begin{cor}
		Let $\textbf{M} = (M_1, M_2, M_3, \tilde{M}_1, \tilde{M}_2)$ be a commuting $5$-tuple of bounded operators on a Hilbert space $\mathcal{H}$. Then $\textbf{M}$ is a $\Gamma_{E(3; 2; 1, 2)}$-unitary if and only if
		\[\{(M_1, \omega M_2, \omega^2 M_3, \omega \tilde{M}_1, \omega^2 \tilde{M}_2) : \omega \in \mathbb{T}\}\]
		is a family of $\Gamma_{E(3; 2; 1, 2)}$-unitaries.
	\end{cor}
The proof of the following theorem follows  by using the similar technique as in Theorm \ref{unita}. Therefore, we skip the proof.
\begin{thm}
Let  $\textbf{U} = ((U_{ij}))$ be a $3 \times 3$ unitary block operator matrix, where $U_{ij}$ are commuting normal operators. Set 			\small{\begin{equation*}
				\begin{aligned}
					\textbf{M} &= \Big(M_1=U_{11}, M_2=U_{11}U_{22} - U_{12}U_{21}+U_{11}U_{33} - U_{13}U_{31}\\&\,\,\,\,\,\,M_3=
					 U_{11}(U_{22}U_{33} - U_{23}U_{32}) - U_{12}(U_{21}U_{33} - U_{31}U_{23}) + U_{13}(U_{21}U_{32} - U_{31}U_{22})
			\\&\,\,\,\,\,\,\tilde{M}_1=U_{22}+U_{33},N_6=U_{22}U_{33} - U_{23}U_{32}\Big)	\end{aligned}
\end{equation*}}
Then $\textbf{M}$ is a $\Gamma_{E(3; 3; 1, 1, 1)}$-unitary.
\end{thm}

	\begin{thm}\label{thm-8}
		Let $\textbf{M} = (M_1, M_2, M_3, \tilde{M}_1, \tilde{M}_2)$ be a commuting $5$-tuple of bounded operators on a Hilbert space $\mathcal{H}$. Then the following are equivalent:
		\begin{enumerate}
			\item $\textbf{M}$ is a $\Gamma_{E(3; 2; 1, 2)}$-unitary.
			
			\item  $M_3$ is a unitary, $||M_1|| < 1$ and  $ \{((\tilde{M}_1 - z M_2)(I - zM_1)^{-1}, (\tilde{M}_2 - z M_3)(I - zM_1)^{-1}) : z \in \mathbb{T} \}$ is a family of $\Gamma_{E(2; 1; 2)}$-unitaries.
				\item[$(2^{\prime})$] $M_3$ and $M_1$ are unitaries and $\{((\tilde{M}_1 - z M_2)(I - zM_1)^{-1}, (\tilde{M}_2 - z M_3)(I - zM_1)^{-1}) : z \in \mathbb{D} \}$ is a family of $\Gamma_{E(2; 1; 2)}$-unitaries.
			
		\end{enumerate}
	\end{thm}
	
	\begin{proof}
We demonstrate the equivalence of $(1)$ and $(2)$. The equivalence of $(1)$ and $(2^{\prime})$ is established in a similar manner.

$(1) \Rightarrow (2):$ 
Since $\textbf{M}$ is a $\Gamma_{E(3; 2; 1, 2)}$-unitary, so by Theorem \ref{thm-7},  it follows that $M_3$ is a unitary operator and $\textbf{M}$ is a $\Gamma_{E(3; 2; 1, 2)}$-contraction. Hence, by Proposition \ref{prop-5}, we deduce that $$ \{((\tilde{M}_1 - z M_2)(I - zM_1)^{-1}, (\tilde{M}_2 - z M_3)(I - zM_1)^{-1}) : z \in \mathbb{T} \}$$ represents a family of $\Gamma_{E(2; 1; 2)}$-contractions. It follows from [\cite{Roy},Theorem $2.5$] that a $2$-tuple of commuting bounded operators $(S,P)$ on a Hilbert space $\mathcal H$ is a $\Gamma_{E(2; 1; 2)}$-unitary if and only if $(S,P)$ is a $\Gamma_{E(2; 1; 2)}$-contraction and $P$ is a unitary. To complete the proof, it    remains to demonstrate that $(\tilde{M}_2 - z M_3)(I - zM_1)^{-1}$  is a unitary operator for all $z \in \mathbb{T}.$ Since $(M_1, M_2, M_3, \tilde{M}_1, \tilde{M}_2)$ is a $\Gamma_{E(3; 2; 1, 2)}$-unitary,  so $M_1, M_2, M_3, \tilde{M}_1, \tilde{M}_2$ are normal operators.  Let $\mathcal A_1$ represent the $C^*$-algebra generated by the commuting normal operators $I,M_1, M_2, M_3, \tilde{M}_1, \tilde{M}_2$.  By the \textit{Gelfand-Naimark theorem}, $\mathcal A_1$ is isometrically $^*$-isomorphic to $C(\sigma(\textbf{M}))$   via the \textit{Gelfand map.} The Gelfand map sends $M_i$ to $x_i$ and $\tilde{M}_j$ to $y_j$ for $1\leq i\leq 3, 1\leq j\leq 2$. The cordinate functions satisfy the conditions $x_1=\overline{y}_2 x_3$ and $x_2= \overline{y}_1 x_3$  on the set $K_1$ and consequently on $\sigma(\textbf{M})$. Thus we have 
$M_1 =\tilde{M}^*_2M_3$ and $\tilde{M}_2=M_1^*M_3$. Note that
\begin{equation}\label{sym}
\begin{aligned}
&\tilde{M}^*_2\tilde{M_2}\\&=M_3^*M_1 M_1^*M_3\\&=M_3^* M^*_1M_1M_3\\
&= M_1^*M_3^*M_3M_1\\
&= M_1^*M_1
\end{aligned}
\end{equation}
and 
\begin{equation}\label{sym1}
			\begin{aligned}
				&(\tilde{M}^*_2 - \overline{z}M_3^*)(\tilde{M}_2 - zM_3)\\
				&= (\tilde{M}^*_2\tilde{M_2} - \overline{z}M^*_3\tilde{M}_2 - z\tilde{M}^*_2M_3 + M^*_3M_3)\\
				&= (I -\bar{z}M_1^*-zM_1+M_1^*M_1)\\
				&= (I -\bar{z}M_1^*-zM_1+M_1M_1^*)\\
				&= (I - zM_1)(I - \overline{z}M^*_1)
				\end{aligned}
		\end{equation}
Hence, it follows from \eqref{sym} and \eqref{sym1} that  $(\tilde{M}_2 - zM_3)(I - zM_1)^{-1}$ is a family of unitary for all $z\in \mathbb T.$

%
%
%

$(2) \Rightarrow (1):$ Suppose that $M_3$ is a unitary, $||M_1|| < 1$ and $$ \{((\tilde{M}_1 - z M_2)(I - zM_1)^{-1}, (\tilde{M}_2 - z M_3)(I - zM_1)^{-1}) : z \in \mathbb{T} \}$$ is a family of $\Gamma_{E(2; 1; 2)}$-unitaries. In order to show that $\textbf{M}$ is a $\Gamma_{E(3; 2; 1, 2)}$-unitary, it is  sufficient to establish that $M_1, M_2, M_3, \tilde{M}_1, \tilde{M}_2$ are normal operators.  As $$\{((\tilde{M}_1 - z M_2)(I - zM_1)^{-1}, (\tilde{M}_2 - z M_3)(I - zM_1)^{-1}) : z \in \mathbb{T} \}$$ is a family of $\Gamma_{E(2; 1; 2)}$-unitaries, it follows from [Theorem $2.5$,\cite{Roy}] that
\begin{equation}\label{N-1-N-11322}
\begin{aligned}
(I-\bar{z}M_1^*)(\tilde{M}_1 - z M_2)&=(\tilde{M}_1^*-\bar{z}M_2^*)(\tilde{M}_2 - z M_3)
\end{aligned}
\end{equation}	
and
\begin{equation}\label{N-1-N-113122}
\begin{aligned}
(I-\bar{z}M_1^*)(I-zM_1)=(\tilde{M}_2^* - \bar{z} M_3^*)(\tilde{M}_2 - z M_3),\\
(I-zM_1)(I-\bar{z}M_1^*)=(\tilde{M}_2 - z M_3)(\tilde{M}_2^* - \bar{z} M_3^*),
\end{aligned}
\end{equation}	
for all $z\in \mathbb T.$ Therefore, it follows  from \eqref{N-1-N-11322} 
 that
\begin{equation}\label{N-1zN-22113}
\begin{aligned}
\tilde{M}_1+M_1^*M_2=\tilde{M}_1^*\tilde{M}_2+M_2^*M_3, M_2=\tilde{M}_1^*M_3 ~{\rm{and}}~M_1^*\tilde{M}_1=M_2^*\tilde{M}_2
\end{aligned}
\end{equation} 
Also, from  \eqref{N-1-N-113122}, we have
\begin{equation}\label{N-1zN-221131}
\begin{aligned}
\tilde{M}_2^*\tilde{M}_2+M_3^*M_3=I+M_1^*M_1, M_1=\tilde{M}_2^*M_3,
\end{aligned}
\end{equation} 
and \begin{equation}\label{eq20}
\tilde{M}_2\tilde{M}_2^*+M_3M_3^*=I+M_1M_1^*, M_1=M_3\tilde{M}_2^*.
\end{equation}
From \eqref{N-1zN-221131} and \eqref{eq20}, we have 
$M_1M_1^*=\tilde{M}_2^*M_3\tilde{M}_2M_3^*=\tilde{M}_2^*\tilde{M}_2M_3M_3^*=\tilde{M}_2^*\tilde{M}_2=M_1^*M_1.$ This shows that $M_1$ is normal. Also, from \eqref{N-1zN-221131} and \eqref{eq20}, we deduce that $\tilde{M}_2$ is a normal operator.  We notice from \eqref{N-1zN-22113} that 
\begin{equation}\label{eq21}M_3^*M_2=M_3^*\tilde{M}_1^*M_3=\tilde{M}_1^*.\end{equation}
Thus, from \eqref{N-1zN-22113} and \eqref{eq21}, it follows that $\tilde{M}_1M_2=	\tilde{M}_1\tilde{M}_1^*M_3$ and $M_2\tilde{M}_1=\tilde{M}_1^*M_3M_1=\tilde{M}_1^*M_1M_3.$ Since $\tilde{M}_1$  and $M_2$ commute, we have $\tilde{M}_1\tilde{M}_1^*M_3=
\tilde{M}_1^*M_1M_3$. Since $M_3$ is a unitary, we conclude that $\tilde{M}_1$ is a normal operator.  By using \eqref{eq21} and  argument similar to as above, we deduce that $M_2$ is a normal operator. Let $\mathcal A_1$ represent the $C^*$-algebra generated by the commuting normal operators $I,M_1, M_2, M_3, \tilde{M}_1, \tilde{M}_2$. By the \textit{Gelfand-Naimark theorem}, $\mathcal A_1$ is isometrically $^*$-isomorphic to $C(\sigma(\textbf{M}))$ via the \textit{Gelfand map.}  The \textit{Gelfand map} sends $M_i$ to the coordinate function $x_i$  and sends $\tilde{M}_j$ to the coordinate function $y_j$ for $1\leq i\leq 3, 1\leq j \leq 2.$ Since  $||M_1|| < 1$ and $M_3$ is unitary, it implies that $|x_1|<1$ and $|x_3|=1.$ In order to complete the proof, it is sufficient to verify that the joint spectrum $\sigma(\textbf{M})$ is contained in $K_1.$

Let $(x_1,x_2,x_3,y_1,y_2)$ be  in $\sigma(\textbf{M})$. As $$ \{((\tilde{M}_1 - z M_2)(I - zM_1)^{-1}, (\tilde{M}_2 - z M_3)(I - zM_1)^{-1}) : z \in \mathbb{T} \}$$ is a family of $\Gamma_{E(2; 1; 2)}$-unitary,  the Taylor joint spectrum of $$ \{((\tilde{M}_1 - z M_2)(I - zM_1)^{-1}, (\tilde{M}_2 - z M_3)(I - zM_1)^{-1}) : z \in \mathbb{T} \}$$ lies in $b\Gamma_{E(2;1;2)}.$ So, by spectral mapping thorem, we get $$\Big(\frac{y_1 - zx_2}{1 - zx_1}, \frac{y_2 - zx_3}{1 - zx_1}\Big) \in b\Gamma_{E(2; 1; 2)} ~\text{for all}~ z \in \mathbb{T}~\text{with}~   |x_1| < 1 ~\text{and}~|x_3|=1$$ and hence by Theorem \ref{thm-6}, we conclude that   $(x_1,x_2,x_3,y_1,y_2)\in K_1$. This completes the proof.

	\end{proof}
	
The following theorems establish the relationship between the $\Gamma_{E(3; 3; 1, 1, 1)}$-unitary and $\Gamma_{E(3; 2; 1, 2)}$-unitary. 	
	\begin{thm}\label{thm-9}
		Let $\textbf{N} = (N_1, \dots, N_7)$ be a $7$-tuple of commuting bounded operators on a Hilbert space $\mathcal{H}$. Then the following are equivalent:
		\begin{enumerate}
			\item $\textbf{N}$ is $\Gamma_{E(3; 3; 1, 1, 1)}$-unitary.
			
			\item $\|N_i\|\leq 1, 1\leq i \leq 6$ and  $\{(N_1, N_3 + \eta N_5, \eta N_7, N_2 + \eta N_4, \eta N_6) : \eta \in \mathbb{T}\}$ is a family of $\Gamma_{E(3; 2; 1, 2)}$-unitary
			
			\item $\|N_i\|\leq 1, 1\leq i \leq 6$ and $\{(N_1, N_5 + \eta N_3, \eta N_7, N_4 + \eta N_2, \eta N_6) : \eta \in \mathbb{T}\}$ is a family of $\Gamma_{E(3; 2; 1, 2)}$-unitary.
			
			\item $\|N_i\|\leq 1, 1\leq i \leq 6$ and $ \{(N_2, N_3 + \eta N_6, \eta N_7, N_1 + \eta N_4, \eta N_5) : \eta \in \mathbb{T}\}$ is a family of $\Gamma_{E(3; 2; 1, 2)}$-unitary.
			
			\item $\|N_i\|\leq 1, 1\leq i \leq 6$ and $ \{(N_2, N_6 + \eta N_3, \eta N_7, N_4 + \eta N_1, \eta N_5) : \eta \in \mathbb{T}\}$ is a family of $\Gamma_{E(3; 2; 1, 2)}$-unitary.
			
			\item $\|N_i\|\leq 1, 1\leq i \leq 6$ and $ \{(N_4, N_5 + \eta N_6, \eta N_7, N_1 + \eta N_2, \eta N_3) : \eta \in \mathbb{T}\}$ is a family of $\Gamma_{E(3; 2; 1, 2)}$-unitary.
			
			\item $\|N_i\|\leq 1, 1\leq i \leq 6$ and $ \{(N_4, N_6 + \eta N_5, \eta N_7, N_2 + \eta N_1, \eta N_3) : \eta \in \mathbb{T}\}$ is a family of $\Gamma_{E(3; 2; 1, 2)}$-unitary.
		\end{enumerate}
	\end{thm}
	
	\begin{proof}
		We only demonstrate that $(1) \Leftrightarrow(2)$ holds.  The equivalences $(1) \Leftrightarrow (3)$, $(1) \Leftrightarrow (4)$, $(1) \Leftrightarrow (5)$, $(1) \Leftrightarrow (6)$, and $(1) \Leftrightarrow (7)$  is established in a similar manner.
		
$(1) \Rightarrow (2):$ Let $\textbf{N}$ be a $\Gamma_{E(3; 3; 1, 1, 1)}$-unitary. Then, Proposition  \ref{prop-3} implies that $$\{(N_1, N_3 + \eta N_5, \eta N_7, N_2 + \eta N_4, \eta N_6) : \eta \in \mathbb{T}\}$$ forms a family of $\Gamma_{E(3; 2; 1, 2)}$-contractions. For fixed but arbitrary $\eta \in \mathbb T,$ set $$M_1=N_1, M_2=N_3 + \eta N_5, M_3=\eta N_7, \tilde{M}_1=N_2 + \eta N_4~{\rm{ and}}~\tilde{M}_2=\eta N_6.$$ Clearly, $M_3$ is unitary. Hence, it follows from Theorem \ref{thm-7} that $$\{(N_1, N_3 + \eta N_5, \eta N_7, N_2 + \eta N_4, \eta N_6) : \eta \in \mathbb{T}\}$$ is a family of $\Gamma_{E(3; 2; 1, 2)}$-unitary.
		
$(2) \Rightarrow (1):$ Suppose that $\{(N_1, N_3 + \eta N_5, \eta N_7, N_2 + \eta N_4, \eta N_6) : \eta \in \mathbb{T}\}$ is a family of $\Gamma_{E(3; 2; 1, 2)}$-unitary. It yields from Theorm \ref{thm-7} that 
\begin{equation} \label{Muni} N_1=\bar{\eta}N_6^*\eta N_7, N_3+\eta N_5=(N_2^*+\bar{\eta}N_4^*)\eta N_7\end{equation} and $\eta N_7$ is a unitary for all $\eta \in \mathbb T.$ For $\eta =1,$ we have $N_7$ is a unitary. From \eqref{Muni}, we also deduce that $N_1=N_6^*N_7$,  $N_2=N_5^*N_7$ and $N_3=N_4^*N_7.$  Hence, by Theorem \ref{thm-5}, we conclude that  $\textbf{N}$ is $\Gamma_{E(3; 3; 1, 1, 1)}$-unitary. This completes the proof.
	\end{proof}
	
	\begin{thm}
		Let $(M_1, M_2, M_3, \tilde{M}_1, \tilde{M}_2)$ be a commuting $5$-tuple of bounded operators on Hilbert space $\mathcal{H}$. Then $(M_1, M_2, M_3, \tilde{M}_1, \tilde{M}_2)$ is a $\Gamma_{E(3; 2; 1, 2)}$-unitary if and only if $\Big(M_1, \frac{\tilde{M}_1}{2}, \frac{M_2}{2}, \frac{\tilde{M}_1}{2}, \frac{M_2}{2}, \tilde{M}_2, M_3\Big)$ is a $\Gamma_{E(3; 3; 1, 1, 1)}$-unitary.
	\end{thm}
	
	\begin{proof}
		Let $(M_1, M_2, M_3, \tilde{M}_1, \tilde{M}_2)$ be a $\Gamma_{E(3; 2; 1, 2)}$-unitary. Then, it follows from Theorem \ref{thm-7} that $M_3$ is a unitary. According to Proposition \ref{prop-4}, we have$\Big(M_1, \frac{\tilde{M}_1}{2}, \frac{M_2}{2}, \frac{\tilde{M}_1}{2}, \frac{M_2}{2}, \tilde{M}_2, M_3\Big)$ is a $\Gamma_{E(3; 3; 1, 1, 1)}$-contraction. Consequently, using Theorem \ref{thm-5}, we conclude that $\Big(M_1, \frac{\tilde{M}_1}{2}, \frac{M_2}{2}, \frac{\tilde{M}_1}{2}, \frac{M_2}{2}, \tilde{M}_2, M_3\Big)$ is a $\Gamma_{E(3; 3; 1, 1, 1)}$-unitary.
		
Conversely, suppose that $\Big(M_1, \frac{\tilde{M}_1}{2}, \frac{M_2}{2}, \frac{\tilde{M}_1}{2}, \frac{M_2}{2}, \tilde{M}_2, M_3\Big)$ is a $\Gamma_{E(3; 3; 1, 1, 1)}$-unitary. It yields from Theorem \ref{thm-5} that 
$$M_1 = \tilde{M}^*_2M_3, M_2 = \tilde{M}^*_1M_3, \frac{\tilde{M}^*_1}{2} = \frac{M_2}{2}M_3, \frac{M_2}{2} = \frac{\tilde{M}^*_1}{2}M_3$$
 and $M_3$ is unitary. Thus, by using Theorem \ref{thm-7}, we deduce that  $\Big(M_1, \frac{\tilde{M}_1}{2}, \frac{M_2}{2}, \frac{\tilde{M}_1}{2}, \frac{M_2}{2}, \tilde{M}_2, M_3\Big)$ is a $\Gamma_{E(3; 3; 1, 1, 1)}$-unitary. This completes the proof.
	\end{proof}

\section{$\Gamma_{E(3; 3; 1, 1, 1)}$-isometries and $\Gamma_{E(3; 2; 1, 2)}$-isometries}
In the previous section we discussed the the special types of $\Gamma_{E(3; 3; 1, 1, 1)}$-contractions and $\Gamma_{E(3; 2; 1, 2)}$-contractions, namely, $\Gamma_{E(3; 3; 1, 1, 1)}$-unitaries and $\Gamma_{E(3; 2; 1, 2)}$-unitaries, respectively. In this section we describe the $\Gamma_{E(3; 3; 1, 1, 1)}$-isometries and $\Gamma_{E(3; 2; 1, 2)}$-isometries. A $\Gamma_{E(3; 3; 1, 1, 1)}$-isometry (respectively, $\Gamma_{E(3; 2; 1, 2)}$-isometry) is defined as a $7$-tuple (respectively, $5$-tuple) of commuting bounded operators that possesses a simultaneous extension  to a $\Gamma_{E(3; 3; 1, 1, 1)}$-unitary (respectively, $\Gamma_{E(3; 2; 1, 2)}$-unitary). Therefore, a $\Gamma_{E(3; 3; 1, 1, 1)}$-isometry $\textbf{V} = (V_1, \dots, V_7)$ and a $\Gamma_{E(3; 2; 1, 2)}$-isometry $\textbf{W} = (W_1, W_2, W_3, \tilde{W}_1, \tilde{W}_2)$  consist of commuting sub-normal operators. Furthermore, by definition of $\Gamma_{E(3; 3; 1, 1,1)}$-isometry (respectively, $\Gamma_{E(3; 2; 1, 2)}$-isometry), it follows that $V_7$ (respectively, $W_3$) is an isometry. We call $\textbf{V} = (V_1, \dots, V_7)$ is a pure $\Gamma_{E(3; 3; 1, 1, 1)}$-isometry (respectively,  $\textbf{W} = (W_1, W_2, W_3, \tilde{W}_1, \tilde{W}_2)$ is a pure $\Gamma_{E(3; 2; 1, 2)}$-isometry) if $V_7$ (respectively, $W_3$) is a pure isometry, that is, a shift of some multiplicity.  This section discusses the Wold Decomposition for $\Gamma_{E(3; 3; 1, 1, 1)}$-isometry and $\Gamma_{E(3; 2; 1, 2)}$-isometry. We investigate the various properties of $\Gamma_{E(3; 3; 1, 1, 1)}$-isometries and $\Gamma_{E(3; 2; 1, 2)}$-isometries. 
\begin{thm}[Wold Decomposition for a $\Gamma_{E(3; 3; 1, 1, 1)}$-Isometry]\label{thm-10}
 Let $\textbf{V} = (V_1, \dots, V_7)$ be a $\Gamma_{E(3; 3; 1, 1, 1)}$-isometry on a Hilbert space $\mathcal{H}$. Then, there exists a decomposition of $\mathcal{H}$ into a direct sum $\mathcal{H} = \mathcal{H}_1 \oplus \mathcal{H}_2$ that satisfies the following conditions:
		
		\begin{enumerate}
			\item $\mathcal{H}_1$ and $\mathcal{H}_2$ are reducing subspaces for all $V_i, 1\leq i \leq  7$.
			
			\item  If $N_i = V_i|_{\mathcal{H}_1}$ and $X_i = V_i|_{\mathcal{H}_2}$ for $1 \leq i\leq 7$, then $\textbf{N} = (N_1, \dots, N_7)$ is a $\Gamma_{E(3; 3; 1, 1, 1)}$-unitary and $\textbf{X} = (X_1, \dots, X_7)$ is a \textit{pure $\Gamma_{E(3; 3; 1, 1, 1)}$-isometry}.
		\end{enumerate}
	\end{thm}
	
	\begin{proof}
Since $\textbf{V} = (V_1, \dots, V_7)$ is a $\Gamma_{E(3; 3; 1, 1, 1)}$-isometry, so there exists a Hilbert space $\mathcal K$ containing $\mathcal H$ and a $\Gamma_{E(3; 3; 1, 1, 1)}$-unitary $(U_1,U_2,\dots,U_7)$ on $\mathcal{K}$ such that $V_i=U_i|_{\mathcal H}$ for $1\leq i\leq 7$. For  $h_1,h_2 \in \mathcal H$ and  $1\leq i\leq 6,$ it follows from Theorem \ref{thm-5} that
\begin{equation}\label{v11}
\begin{aligned}\langle V_i h_1,h_2 \rangle&=\langle U_i h_1,h_2 \rangle\\&=\langle U_{7-i}^*U_7h_1,h_2\rangle \\&=\langle U_7h_1,U_{7-i}h_2\rangle\\&=\langle V_7h_1,V_{7-i}h_2\rangle\\&=\langle V_{7-i}^*V_7h_1,h_2\rangle
\end{aligned}
\end{equation}
and 	
\begin{equation}\label{v12}
\begin{aligned}\langle V_7^*V_7h_1,h_2 \rangle&=\langle V_7 h_1,V_7h_2 \rangle\\&=\langle U_7h_1,U_7h_2\rangle \\&=\langle U_7^*U_7h_1,h_2\rangle\\&=\langle h_1,h_2\rangle.
\end{aligned}
\end{equation}
Thus, from \eqref{v11} and \eqref{v12}, we deduce that $V_i=V_{7-i}^*V_7$, $1\leq i\leq 6$,  and $V_7$ is an isometry. Since $V_7$ is an isometry, by Wold type decomposition theorem, there exists an orthogonal decomposition $\mathcal H_1\oplus \mathcal H_2$ of $\mathcal H$ such that $V_7$ has the following decomposition \begin{equation}\label{dec}
			\begin{aligned}
				V_7 &= 
				\begin{pmatrix}
					N_7 & 0\\
					0 & X_7
				\end{pmatrix},
			\end{aligned}
		\end{equation}
where $N_7$ is an unitary and  $X_7$ is a shift of some multiplicity. Let $V_i$ have the form \begin{equation*}
\begin{aligned}
				V_i &= \begin{pmatrix}
A_{11}^{(i)} & A_{12}^{(i)}\\
A_{21}^{(i)} & A_{22}^{(i)}
\end{pmatrix}
\end{aligned}
\end{equation*}
with respect to the decomposition $\mathcal H=\mathcal H_1\oplus \mathcal H_2, 1\leq i\leq 6$.
As $V_i$ commutes with $V_7$ the off-diagonal elements $A_{12}^{(i)}$ and $ A_{21}^{(i)}$ satisfy the following properties:		\begin{equation*}
			\begin{aligned}
				N_7A_{12}^{(i)} = A_{12}^{(i)}X_7 ~\text{and}~ X_7A_{21}^{(i)} = A_{21}^{(i)}N_7,~~~~~1\leq i\leq 6.
			\end{aligned}
		\end{equation*}
It follows from [Lemma $2.5$, \cite{ay1}] that  $A_{21}^{(i)}=0$ for $1\leq i\leq 6.$ For $1\leq i\leq 6,$ note that
\begin{equation}\label{matrixeq}
\begin{aligned}
 \begin{pmatrix}
A_{11}^{(i)} & A_{12}^{(i)}\\
0 & A_{22}^{(i)}
\end{pmatrix} &=V_i \\&=V_{7-i}^*V_7\\&= \begin{pmatrix}
{A_{11}^{(7-i)}}^* & 0\\
{A_{12}^{(7-i)}}^* & {A_{22}^{(7-i)}}^*
\end{pmatrix}\begin{pmatrix}
					N_7 & 0\\
					0 & X_7
				\end{pmatrix}\\&= \begin{pmatrix}
{A_{11}^{(7-i)}}^*N_7 & 0\\
{A_{12}^{(7-i)}}^*N_7& {A_{22}^{(7-i)}}^*X_7
\end{pmatrix}.
		\end{aligned}
\end{equation}
Thus, from \eqref{matrixeq}, we have  $A_{12}^{(i)}=0$ for $1\leq i\leq 6$ and hence we get $V_i=N_i\oplus X_i, 1\leq i \leq 7$, where $N_i=A_{11}^{(i)}$ and $X_i=A_{22}^{(i)}, 1\leq i \leq 6$. As $V_i=V_{7-i}^*V_7$, it follows that 
\begin{equation} \label{unit}
N_i=N_{7-i}^*N_7, X_i=X_{7-i}^*X_7~{\rm{for}}~1\leq i \leq 6.
\end{equation}
Since $N_7$ is a unitary and $N_i=N_{7-i}^*N_7, 1 \leq i \leq 6,$ we conclude from Theorem \ref{thm-5} that $\textbf{N} = (N_1, \dots, N_7)$ is a $\Gamma_{E(3; 3; 1, 1, 1)}$-unitary. Furthermore, by definition of \textit{pure $\Gamma_{E(3; 3; 1, 1, 1)}$-isometry}, we deduce that $\textbf{X} = (X_1, \dots, X_7)$ is a \textit{pure $\Gamma_{E(3; 3; 1, 1, 1)}$-isometry}. This completes the proof.
	\end{proof}
We state the  Wold Decomposition theorem for $\Gamma_{E(3; 2; 1, 2)}$-Isometry. Its proof is analogous to that of the previous theorem. Therefore, we skip the proof.
	
	\begin{thm}[Wold Decomposition for $\Gamma_{E(3; 2; 1, 2)}$-Isometry] \label{thm-11}
		Let $\textbf{W} = (W_1, W_2, W_3, \tilde{W}_1, \tilde{W}_2)$ be a $\Gamma_{E(3; 2; 1, 2)}$-isometry on a Hilbert space $\mathcal{H}$. Then $\mathcal{H}$ decomposes into a direct sum $\mathcal{H} = \mathcal{H}_1 \oplus \mathcal{H}_2$ of subspaces $\mathcal{H}_1$ and $\mathcal{H}_2$ such that
		\begin{enumerate}
			\item $\mathcal{H}_1, \mathcal{H}_2$ are reducing subspaces for $W_1, W_2, W_3, \tilde{W}_1, \tilde{W}_2$.
			
			\item If $M_i = W_i|_{\mathcal{H}_1}, \tilde{M}_j = \tilde{W}_j|_{\mathcal{H}_1}$ and $L_i = W_i|_{\mathcal{H}_2}, \tilde{L}_j = \tilde{W}_j|_{\mathcal{H}_2}$ for $i = 1, 2, 3$ and $j = 1, 2,$ then $\textbf{M} = (M_1, M_2, M_3, \tilde{M}_1, \tilde{M}_2)$ is a $\Gamma_{E(3; 2; 1, 2)}$-unitary and $\textbf{L} = (L_1, L_2, L_3, \tilde{L}_1, \tilde{L}_2)$ is a pure $\Gamma_{E(3; 2; 1, 2)}$-isometry.
		\end{enumerate}
	\end{thm}	
The following lemma gives the characterization of tetrablock isometry.
\begin{lem}\label{tetiso}
Let $(A,B,P)$ be a commuting triple of bounded operator on some Hilbert space $\mathcal H.$ Then $(A,B,P)$ is a $\Gamma_{E(2;2;1,1)}$-isometry if and only if $A$ is a contraction and  $$\rho_{G_{E(2; 2; 1,1)}} (A,zB,zP)=0 ~~\textit{and}~~\rho_{G_{E(2; 2; 1,1)}} (B,zA,zP)=0 ~\textit{for~all}~z\in \mathbb T.$$
\end{lem}
\begin{proof}
Suppose that $(A,B,P)$ is a $\Gamma_{E(2;2;1,1)}$-isometry. It follows from  [Theorem $5.7$, \cite{Bhattacharyya}] that 
\begin{equation}\label{isom}
A=B^*P, B=A^*P~~\textit{and}~P~\textit{is~an~isometry}.
\end{equation}
By definition of $\rho_{G_{E(2; 2; 1,1)}}$, we deduce that $$\rho_{G_{E(2; 2; 1,1)}} (A,zB,zP)=0 ~~\textit{and}~~\rho_{G_{E(2; 2; 1,1)}} (B,zA,zP)=0 ~\textit{for~all}~z\in \mathbb T.$$

Conversely, suppose that $$\rho_{G_{E(2; 2; 1,1)}} (A,zB,zP)=0 ~~\textit{and}~~\rho_{G_{E(2; 2; 1,1)}} (B,zA,zP)=0 ~\textit{for~all}~z\in \mathbb T.$$ Then, by definition of $\rho_{G_{E(2; 2; 1,1)}}$, we have
\begin{equation}\label{abp}
I-P^*P+(B^*B-A^*A)-2Re z(B-A^*P)=0
\end{equation}
and 
\begin{equation}\label{bap}
I-P^*P+(A^*A-B^*B)-2Re z(A-B^*P)=0
\end{equation}
for all $z\in \mathbb T.$ By putting $z=1$ and $z=-1$ in \eqref{abp} and \eqref{bap}, we get
\begin{equation}\label{abp1}
I-P^*P+(B^*B-A^*A)=0 ~\textit{and}~B=A^*P
\end{equation}
and 
\begin{equation}\label{bap1}
I-P^*P+(A^*A-B^*B)=0 ~\textit{and}~A=B^*P.
\end{equation}
From \eqref{abp1} and \eqref{bap1}, we conclude that $P^*P=I, B=A^*P$ and $A=B^*P$. Thus, by [Theorem $5.7$, \cite{Bhattacharyya}], we deduce that $(A,B,P)$ is a $\Gamma_{E(2;2;1,1)}$-isometry. This completes the proof.
\end{proof}

The following theorems provide a necessary and sufficient condition for $\Gamma_{E(3; 3; 1, 1, 1)}$-isometry and $\Gamma_{E(3; 2; 1, 2)}$-isometry. 


\begin{thm}\label{thm-12}
Let $\textbf{V} = (V_1, \dots, V_7)$ be a $7$-tuple of commuting bounded operators on a Hilbert space $\mathcal{H}$. Then the following are equivalent:
\begin{enumerate}
\item $\textbf{V}$ is a $\Gamma_{E(3; 3; 1, 1, 1)}$-isometry.

\item $\textbf{V}$ is a $\Gamma_{E(3; 3; 1, 1, 1)}$-contraction and $V_7$ is an isometry.

\item For $1\leq i\leq 6,$ $V_i$'s are contractions, $V_i = V^*_{7-i} V_7$ and $V_7$ is isometry.
\item For $1\leq i\leq 3,$ $(V_i,V_{7-i},V_7)$ are $\Gamma_{E(2; 2; 1, 1)}$-isometries.
\item  $V_i$'s are contractions and $\rho_{G_{E(2; 2; 1,1)}} (V_i,zV_{7-i},zV_7)=0, 1\leq i \leq 6,$ for all $z\in \mathbb T.$

\item $V_7$ is an isometry, $r(V_i)\leqslant 1$ for $1\leq i \leq 6$ and $V_1 = V^*_6V_7, V_2 = V^*_5V_7, V_3 = V^*_4V_7$.
\item $||V_2|| < 1$, $\|V_4\|<1,$ $\{((V_1 - z_2V_3)(I - z_2V_2)^{-1}, (V_4 - z_2V_6)(I - z_2V_2)^{-1}, (V_5 - z_2V_7)(I - z_2V_2)^{-1}) : |z_2| = 1\}$  and  $\{((V_1 - z_3V_5)(I - z_3V_4)^{-1},  (V_2 - z_3V_6)(I - z_3V_4)^{-1}, (V_3 - z_3V_7)(I - z_3V_4)^{-1}) : |z_3| = 1\}$ are  commuting family of tetrablock isometries.
				
				\item[$(7)^{\prime}$]  $V_2$ is isometry and $\{((V_1 - z_2V_3)(I - z_2V_2)^{-1}, (V_4 - z_2V_6)(I - z_2V_2)^{-1}, (V_5 - z_2V_7)(I - z_2V_2)^{-1}) : |z_2| < 1\}$ is a commuting family of tetrablock isometries.
			\item[$(7)^{\prime\prime}$] $V_4$ is isometry and $ \{((V_1 - z_3V_5)(I - z_3V_4)^{-1},  (V_2 - z_3V_6)(I - z_3V_4)^{-1}, (V_3 - z_3V_7)(I - z_3V_4)^{-1}) : |z_3| < 1\}$ is a commuting family of tetrablock isometries.
						
				\item  $||V_1|| < 1$, $\|V_2\|<1,$  $\{((V_2 - z_1V_3)(I - z_1V_1)^{-1},  (V_4 - z_1V_5)(I - z_1V_1)^{-1}, (V_6 - z_1V_7)(I - z_1V_1)^{-1}) : |z_1| = 1\}$ and  $\{((V_1 - z_2V_3)(I - z_2V_2)^{-1}, (V_4 - z_2V_6)(I - z_2V_2)^{-1}, (V_5 - z_2V_7)(I - z_2V_2)^{-1}) : |z_2| = 1\}$ are  commuting family of tetrablock isometries.
				
				\item [$(8)^{\prime}$]  $V_1$ is isometry and $ \{((V_2 - z_1V_3)(I - z_1V_1)^{-1},  (V_4 - z_1V_5)(I - z_1V_1)^{-1}, (V_6 - z_1V_7)(I - z_1V_1)^{-1}) : |z_1| < 1\}$ is a commuting family of tetrablock isometries.
			
			\item[$(8)^{\prime\prime}$]  $V_2$ is isometry and $\{((V_1 - z_2V_3)(I - z_2V_2)^{-1}, (V_4 - z_2V_6)(I - z_2V_2)^{-1}, (V_5 - z_2V_7)(I - z_2V_2)^{-1}) : |z_2| < 1\}$ is a commuting family of tetrablock isometries.
					
			\item  $\|V_1\|<1,$ $\|V_4\| < 1$, $ \{((V_1 - z_3V_5)(I - z_3V_4)^{-1},  (V_2 - z_3V_6)(I - z_3V_4)^{-1}, (V_3 - z_3V_7)(I - z_3V_4)^{-1}) : |z_3| = 1\}$ and $\{((V_2 - z_1V_3)(I - z_1V_1)^{-1},  (V_4 - z_1V_5)(I - z_1V_1)^{-1}, (V_6 - z_1V_7)(I - z_1V_1)^{-1}) : |z_1| = 1\}$ are  commuting family of tetrablock isometries.				
				\item [$(9)^{\prime}$] $V_1$ is isometry and $ \{((V_2 - z_1V_3)(I - z_1V_1)^{-1},  (V_4 - z_1V_5)(I - z_1V_1)^{-1}, (V_6 - z_1V_7)(I - z_1V_1)^{-1}) : |z_1| < 1\}$ is a commuting family of tetrablock isometries.
				\item [$(9)^{\prime\prime}$]  $V_4$ is isometry and $ \{((V_1 - z_3V_5)(I - z_3V_4)^{-1},  (V_2 - z_3V_6)(I - z_3V_4)^{-1}, (V_3 - z_3V_7)(I - z_3V_4)^{-1}) : |z_3| < 1\}$ is a commuting family of tetrablock isometries.
			
\end{enumerate}
\end{thm}

\begin{proof}
We demonstrate the equivalence of $(1), (2), (3),(4)$,$(5)$,$(6)$, $(7)$ $(7)^{\prime}$,$(7)^{\prime\prime}$,$(8)$ $(8)^{\prime}$,$(8)^{\prime\prime}$,$(9)$,$(9)^{\prime}$ and $(9)^{\prime\prime}$ by proving the following relationships: 

\small{$(1) \Rightarrow (2) \Rightarrow (3)\Leftrightarrow (4)  \Rightarrow (1), (4)  \Rightarrow (5), (5) \Rightarrow(4), (4)  \Rightarrow (6), (6) \Rightarrow(4), (1)\Leftrightarrow (7),(1)\Leftrightarrow (7)^{\prime}, (1)\Leftrightarrow (7)^{\prime\prime},(1)\Leftrightarrow (8),(1)\Leftrightarrow (8)^{\prime},(1)\Leftrightarrow (8)^{\prime\prime},(1)\Leftrightarrow (9),(1)\Leftrightarrow (9)^{\prime}, (1)\Leftrightarrow (9)^{\prime\prime}.$}			

$(1) \Rightarrow (2):$ Suppose that $\textbf{V}$ is a $\Gamma_{E(3; 3; 1, 1, 1)}$-isometry. Then, by definition of $\Gamma_{E(3; 3; 1, 1, 1)}$-isometry, there exists a Hilbert space $\mathcal{K}$  containing $\mathcal{H}$ and a $\Gamma_{E(3; 3; 1, 1, 1)}$-unitary operator $\textbf{N} = (N_1, \dots, N_7)$ on $\mathcal K,$ such that $\mathcal H$ is an invariant subspace of $N_i$s and $V_i=N_i|_{\mathcal H}$, $1\leq i \leq 7.$  For $h_1,h_2\in \mathcal H,$ note that
\begin{equation}\label{iso}
\begin{aligned}
\langle V^*_7V_7h_1,h_2\rangle&=\langle V_7h_1,V_7h_2\rangle\\&=\langle N_7h_1,N_7h_2\rangle\\&=\langle N^*_7N_7h_1,h_2\rangle\\&=\langle h_1,h_2\rangle.
\end{aligned}
\end{equation}
It implies from  \eqref{iso} that $V_7$ is an isometry. Clearly, $\textbf{V}$ is a $\Gamma_{E(3; 3; 1, 1, 1)}$-contraction because it is the restriction of a $\Gamma_{E(3; 3; 1, 1, 1)}$-unitary $\textbf{N}$ to the invariant subspace $\mathcal{H}$.

$(2) \Rightarrow (3):$ It follows from Proposition \ref{props-1} that $(V_1,V_6,V_7),(V_2,V_5,V_7)$ and $(V_3,V_4,V_7)$ are tetrablock contractions, as $\textbf{V}$ is a $\Gamma_{E(3; 3; 1, 1, 1)}$-contraction. Since $V_7$ is an isometry, by [\cite{Bhattacharyya},Theorem $5.7$], we conclude that the triples $(V_1,V_6,V_7),(V_2,V_5,V_7)$ and $(V_3,V_4,V_7)$ are tetrablock isometries and hence we obtain from  [\cite{Bhattacharyya},Theorem $5.7$] that
 $V_i$'s are contractions and $V_i = V^*_{7-i} V_7$ for $1\leq i \leq 6.$	
 
 $(3) \Leftrightarrow (4):$ 
The equivalence of $(3)$ and $(4)$ is derived from [Theorem $5.7$, \cite{Bhattacharyya}].

$(4) \Rightarrow (1):$ Suppose that  $(V_i,V_{7-i},V_7)$ are $\Gamma_{E(2; 2; 1, 1)}$-isometries for $1\leq i \leq 6.$ Since $V_7$ is a isometry, by \textit{Wold decomposition theorem}, $\mathcal{H}$ decomposes into an direct sum $\mathcal{H}_1 \oplus \mathcal{H}_2$ such that $\mathcal{H}_1$ and $\mathcal{H}_2$ reduces $V_7$ and $N_7={V_{7}}_{|\mathcal{H}_1}$ is unitary, $X_7={V_{7}}_{|\mathcal{H}_2}$ is a unilateral shift of some multiplicity. As $V_i$  commutes with $V_7$, by the same argument as in the proof of the Wold decomposition for a $\Gamma_{E(3; 3; 1, 1, 1)}$-isometry, it yields that $\mathcal{H}_1$ and $\mathcal{H}_2$ are reducing subspaces of $V_i, 1\leq i \leq 6$. Set $N_i:={V_i}_{|\mathcal{H}_1}$, $X_i={V_i}_{|\mathcal{H}_2}$, $N:=(N_1,N_2,\ldots,N_7)$ and $X:=(X_1,X_2,\ldots,X_7)$. Furthermore, $V_i=V_{7-i}^*V_7$ gives $N_i=N_{7-i}^*N_7$ and $X_i=X_{7-i}^*X_7$, $1\leq i \leq 6$. Thus, the commuting $7$-tuple $\textbf{N} = (N_1, \dots, N_7)$ is a $\Gamma_{E(3; 3; 1, 1, 1)}$-unitary because it consists of a unitary operator $N_7$ and contractions $N_i$ satisfying the relation $N_i=N_{7-i}^*N_7$, $1\leq i \leq 6$. It  remains to show that the $7$-tuple $\textbf{X} = (X_1, \dots, X_7)$  extends to a $\Gamma_{E(3; 3; 1, 1, 1)}$-unitary in order to show that $\textbf{V}$ is an isometry. 

Realize $X_7$ as multiplication by the coordinate function $z$ on a vector-valued Hardy space $H^2(\mathcal E)$, where $\mathcal E$ possesses the same dimension as the multiplicity of the shift $X_7$. As $X_i$'s  commute  with the shift $X_7$, there exist $H^{\infty}(\mathcal E)$ functions $\varphi_i$ such that $X_i=M^{\mathcal E}_{\varphi_{i}},$ which represents the multiplication on $H^2(\mathcal E)$ by $\varphi_{i}, 1\leq i \leq 6$. Since $X_i$'s  are contractions, the $H^{\infty}$ norms of the operator-valued functions $\varphi_i$  are  at most one. Since $M^{\mathcal E}_{\varphi_{i}}=(M^{\mathcal E}_{\varphi_{7-i}})^*M^{\mathcal E}_{z}$, it follows that
\begin{equation}\label{varphi}
\begin{aligned}
\varphi_i(z)=\varphi^*_{7-i}(z)z \;\;\;{\rm{for}}~z \in \mathbb T, ~1\leq i \leq 6.
\end{aligned}
\end{equation}
Consider  the multiplication operators $U^{\mathcal{E}}_{\varphi_i}$ and $U^{\mathcal{E}}_{z}$  defined on $L^2(\mathcal{E})$, which are multiplications by $\varphi_i,  1\leq i \leq 6$ and  $z$, respectively. Clearly, $U^{\mathcal{E}}_{z}$ is a unitary on $L^2(\mathcal{E})$. It follows from \eqref{varphi} that 
		\begin{equation}\label{unitaryu}
			\begin{aligned}
				U^{\mathcal{E}}_{\varphi_i}
				&=U^{\mathcal{E}}_{\varphi_{7-i}^*z}\\&= (U^{\mathcal{E}}_{\varphi_{7-i}})^*U^{\mathcal{E}}_{z},\;\; 1\leq i \leq 6.
			\end{aligned}
		\end{equation}
		Thus, it follows from Theorem \ref{thm-5} that $(U^{\mathcal{E}}_{\varphi_1}, \dots, U^{\mathcal{E}}_{\varphi_6}, U^{\mathcal{E}}_{z})$ is a $\Gamma_{E(3; 3; 1, 1, 1)}$-unitary. 
	       
	        $(4) \Leftrightarrow (5):$ This follows from Lemma \ref{tetiso}.
		
		$(4) \Leftrightarrow (6):$ This follows from  [Theorem $5.7$, \cite{Bhattacharyya}].

We only prove $(1)\Leftrightarrow (7)$ and $(1)\Leftrightarrow (7)^{\prime}$,  other equivalences, namely, $(1)\Leftrightarrow(7)^{\prime\prime}, (1)\Leftrightarrow(8), (1)\Leftrightarrow(8)^{\prime},(1) \Leftrightarrow(8)^{\prime\prime}, (1)\Leftrightarrow(9), (1)\Leftrightarrow(9)^{\prime}, (1)\Leftrightarrow(9)^{\prime\prime}$ are established in a similar manner.

 $(1) \Rightarrow (7):$ Let $\textbf{V}$ be a $\Gamma_{E(3; 3; 1, 1, 1)}$-isometry. By the definition of $\Gamma_{E(3; 3; 1, 1, 1)}$-isometry, there exists a Hilbert space $\mathcal{K}$ containing $\mathcal{H}$ and a $\Gamma_{E(3; 3; 1, 1, 1)}$-unitary operator $\textbf{N} = (N_1, \dots, N_7)$ on $\mathcal K,$ such that $\mathcal H$ is an invariant subspace of $N_i$'s and $V_i={N_i}_{|_{\mathcal H}}, 1\leq i \leq 7.$ As $\textbf{N} = (N_1, \dots, N_7)$ is a $\Gamma_{E(3; 3; 1, 1, 1)}$-unitary, it follows from Theorem \ref{thm-5} that $||N_2|| < 1$, $\|N_4\|<1,$ $\{((N_1 - z_2N_3)(I - z_2N_2)^{-1}, (N_4 - z_2N_6)(I - z_2N_2)^{-1}, (N_5 - z_2N_7)(I - z_2N_2)^{-1}) : |z_2| = 1\}$  and  $\{((N_1 - z_3N_5)(I - z_3N_4)^{-1},  (N_2 - z_3N_6)(I - z_3N_4)^{-1}, (N_3 - z_3N_7)(I - z_3N_4)^{-1}) : |z_3| = 1\}$ are  commuting family of tetrablock unitaries.
This shows that $||V_2|| < 1$, $\|V_4\|<1,$ $\{((V_1 - z_2V_3)(I - z_2V_2)^{-1}, (V_4 - z_2V_6)(I - z_2V_2)^{-1}, (V_5 - z_2V_7)(I - z_2V_2)^{-1}) : |z_2| = 1\}$  and  $\{((V_1 - z_3V_5)(I - z_3V_4)^{-1},  (V_2 - z_3V_6)(I - z_3V_4)^{-1}, (V_3 - z_3V_7)(I - z_3V_4)^{-1}) : |z_3| = 1\}$ are  commuting family of tetrablock isometries.

 $(7) \Rightarrow (1):$ Suppose that $||V_2|| < 1$, $\|V_4\|<1,$  
 \begin{equation}\label{eq411}\{((V_1 - z_2V_3)(I - z_2V_2)^{-1}, (V_4 - z_2V_6)(I - z_2V_2)^{-1}, (V_5 - z_2V_7)(I - z_2V_2)^{-1}) : |z_2| = 1\}\end{equation}  and  \begin{equation}\label{eq412}\{((V_1 - z_3V_5)(I - z_3V_4)^{-1},  (V_2 - z_3V_6)(I - z_3V_4)^{-1}, (V_3 - z_3V_7)(I - z_3V_4)^{-1}) : |z_3| = 1\}\end{equation} are  commuting family of tetrablock isometries.
For $z_2\in \mathbb{T}$, it yields from  \eqref{eq411} and
[Theorem $5.7$,\cite{Bhattacharyya}] that 
\begin{equation}\label{V-1-V-3}
\begin{aligned}
(I-\bar{z}_2V^*_2)(V_1-z_2V_3)&=(V_4^*-\bar{z}_2V_6^*)(V_5-z_2V_7),\\
(I-\bar{z}_2V^*_2)(V_4-z_2V_6)&=(V_1^*-\bar{z}_2V_3^*)(V_5-z_2V_7),\\
(V_5^*-\bar{z}_2V_7^*)(V_5-z_2V_7)&=(I-\bar{z}_2V_2^*)(I-z_2V_2).
\end{aligned}
\end{equation}	
Also, for $z_3\in \mathbb T,$ it follows from  \eqref{eq412} and
[Theorem $5.7$,\cite{Bhattacharyya}] 
\begin{equation}\label{V-1-V-31}
\begin{aligned}
(I-\bar{z}_3V^*_4)(V_1-z_3V_5)&=(V_2^*-\bar{z}_3V_6^*)(V_3-z_3V_7), \\
(I-\bar{z}_3V^*_4)(V_2-z_3V_6)&=(V_1^*-\bar{z}_3V_5^*)(V_3-z_3V_7).
\end{aligned}
\end{equation}	
Thus, from \eqref{V-1-V-3}, we get
\begin{equation}\label{eqa3}
\begin{aligned}
V_1+V_2^*V_3=V_4^*V_5+V_6^*V_7, V_3=V_4^*V_7 ~{\rm{and}}~V_2^*V_1=V_6^*V_5,
\end{aligned}
\end{equation}
\begin{equation}\label{eqa4}
\begin{aligned}
V_4+V_2^*V_6=V_1^*V_5+V_3^*V_7, V_6=V_1^*V_7 ~{\rm{and}}~V_2^*V_4=V_3^*V_5,
\end{aligned}
\end{equation}
\begin{equation}\label{eqa5}
\begin{aligned}
V_5^*V_5+V_7^*V_7=I+V_2^*V_2, V_2^*=V_7^*V_5.
\end{aligned}
\end{equation}
We also have from \eqref{V-1-V-31} that
\begin{equation}\label{eqa6}
\begin{aligned}
V_1+V_4^*V_5=V_2^*V_3+V_6^*V_7, V_5=V_2^*V_7 ~{\rm{and}}~V_4^*V_1=V_6^*V_3,
\end{aligned}
\end{equation}

\begin{equation}\label{eqa7}
\begin{aligned}
V_2+V_4^*V_6=V_1^*V_3+V_5^*V_7, V_6=V_1^*V_7 ~{\rm{and}}~V_4^*V_2=V_5^*V_3.
\end{aligned}
\end{equation} From \eqref{eqa5} and \eqref{eqa6}, we get $V_5^*V_5=V_5^*V_2^*V_7=V_2^*V_5^*V_7=V_2^*V_2.$ Hence it follows from \eqref{eqa5} that $V_7$ is an isometry. From \eqref{eqa3},\eqref{eqa4} and \eqref{eqa6},  we have
$$V_6=V_1^*V_7, V_3=V^*_4V_7,V_5=V_2^*V_7.$$
As $V_7$ is isometry, $V_6=V_1^*V_7, V_3=V^*_4V_7,V_5=V_2^*V_7$ and $V_1,V_2$ and $V_4$ are contractions, it follows from equivalence of $(1)$ and $(3)$ that $\textbf{V}$ is a $\Gamma_{E(3; 3; 1, 1, 1)}$-isometry.

 $(1) \Rightarrow (7)^{\prime}:$ This implication establishes by using the similar argument as in $(1)\Rightarrow (7)$.
 
$(7)^{\prime} \Rightarrow (1):$ Suppose that $V_2$ is isometry and 
  \begin{equation}\label{eqaa1}\{((V_1 - z_2V_3)(I - z_2V_2)^{-1}, (V_4 - z_2V_6)(I - z_2V_2)^{-1}, (V_5 - z_2V_7)(I - z_2V_2)^{-1}) : |z_2| < 1\}\end{equation} is a commuting family of tetrablock isometries. For $z_2\in \mathbb{D}$, it follows from  \eqref{eqaa1} and
[Theorem $5.7$,\cite{Bhattacharyya}] that 
\begin{equation}\label{V-1-V-1113}
\begin{aligned}
(I-\bar{z}_2V^*_2)(V_1-z_2V_3)&=(V_4^*-\bar{z}_2V_6^*)(V_5-z_2V_7),\\
(I-\bar{z}_2V^*_2)(V_4-z_2V_6)&=(V_1^*-\bar{z}_2V_3^*)(V_5-z_2V_7),\\
(V_5^*-\bar{z}_2V_7^*)(V_5-z_2V_7)&=(I-\bar{z}_2V_2^*)(I-z_2V_2).
\end{aligned}
\end{equation}	
Therefore, from \eqref{V-1-V-1113}, we have
\begin{equation}\label{eqa333}
\begin{aligned}
V_1=V_4^*V_5, V_2^*V_3=V_6^*V_7, V_3=V_4^*V_7 ~{\rm{and}}~V_2^*V_1=V_6^*V_5,
\end{aligned}
\end{equation}
\begin{equation}\label{eqa444}
\begin{aligned}
V_4=V_1^*V_5,V_2^*V_6=V_3^*V_7, V_6=V_1^*V_7 ~{\rm{and}}~V_2^*V_4=V_3^*V_5,
\end{aligned}
\end{equation}
\begin{equation}\label{eqa555}
\begin{aligned}
V_5^*V_5=I,V_7^*V_7=V_2^*V_2, V_2^*=V_7^*V_5.
\end{aligned}
\end{equation}
Since $V_2$ is an isometry, from \eqref{eqa555}, it follows that $V_5$ and $V_7$ are isometries. By putting $z_2=0$ in \eqref{eqaa1}, we deduce from [Theorem $5.7$,\cite{Bhattacharyya}] that $V_1$ and $V_4$ are contractions.  From \eqref{eqa333},\eqref{eqa444} and \eqref{eqa555},  we have
$$V_6=V_1^*V_7, V_3=V^*_4V_7,V_2=V_5^*V_7.$$ Thus, by equivalence of $(1)$ and $(3)$, we conclude that $\textbf{V}$ is a $\Gamma_{E(3; 3; 1, 1, 1)}$-isometry. This completes the proof.

\end{proof}
We state a theorem for $\Gamma_{E(3; 2; 1, 2)}$-isometry whose proof is analogous to that of the preceding theorem. Therefore, we skip the proof.
	\begin{thm}\label{thm-13}
		Let $\textbf{W} = (W_1, W_2, W_3, \tilde{W}_1, \tilde{W}_2)$ be a $5$-tuple of commuting bounded operators on a Hilbert space $\mathcal{H}$. Then the following are equivalent:
		\begin{enumerate}
			\item $\textbf{W}$ is a $\Gamma_{E(3; 2; 1, 2)}$-isometry.
			
			\item $\textbf{W}$ is a $\Gamma_{E(3; 2; 1, 2)}$-contraction and $W_3$ is an isometry.
			
			\item $(W_1, \tilde{W}_2, W_3), ( \frac{\tilde{W}_1}{2},\frac{W_2}{2},  W_3)$ and $(\frac{W_2}{2}, \frac{\tilde{W}_1}{2}, W_3)$ are $\Gamma_{E(2; 2; 1,1)}$-isometries.
			
			\item $W_3$ is an isometry, $W_1, \frac{W_2}{2}, \frac{\tilde{W}_1}{2}, \tilde{W}_2$ are contractions, and $W_1 = \tilde{W}^*_2W_3, W_2 = \tilde{W}^*_1W_3$.
			
			\item $W_3$ is an isometry, $r(W_1)\leq 1, r(\frac{W_2}{2})\leq 1, r(\frac{\tilde{W}_1}{2})\leq 1, r(\tilde{W}_2) \leqslant 1$, $W_1 = \tilde{W}^*_2W_3$ and $W_2 = \tilde{W}^*_1W_3$.
		\item $\|\tilde{W}_1\|< 2$ and $ \{ ((2W_1 - zW_2)(2I - z\tilde{W}_1)^{-1}, (\tilde{W}_1 - 2z\tilde{W}_2)(2I - z\tilde{W}_1)^{-1}, (W_2 - 2zW_3)(2I - z\tilde{W}_1)^{-1}) :|z| = 1 \}$ is a family of tetrablock isometries.

		\item[$(6)^{\prime}$] $\frac{\tilde{W}_1}{2}$ is an isometry and $\{ ((2W_1 - zW_2)(2I - z\tilde{W}_1)^{-1}, (\tilde{W}_1 - 2z\tilde{W}_2)(2I - z\tilde{W}_1)^{-1}, (W_2 - 2zW_3)(2I - z\tilde{W}_1)^{-1}) :|z| < 1 \}$ is a family of tetrablock isometries.
		
		\item  $W_3$ is an isometry, $||W_1|| < 1$ and  $ \{((\tilde{W}_1 - z W_2)(I - zW_1)^{-1}, (\tilde{W}_2 - z W_3)(I - zW_1)^{-1}) : z \in \mathbb{T} \}$ is a family of $\Gamma_{E(2; 1; 2)}$-isometries.
				\item[$(7)^{\prime}$] $W_3$ and $W_1$ are isometries and $\{((\tilde{W}_1 - z W_2)(I - zW_1)^{-1}, (\tilde{W}_2 - z W_3)(I - zW_1)^{-1}) : z \in \mathbb{D} \}$ is a family of $\Gamma_{E(2; 1; 2)}$-isometries.

\end{enumerate}
	\end{thm}
We now prove a structure theorem for  pure $\Gamma_{E(3; 3; 1, 1, 1)}$-isometry.
	\begin{thm}\label{thm-14}
		Let $\textbf{V} = (V_1, \dots, V_7)$ be a commuting $7$-tuple of bounded operators on a separable Hilbert space $\mathcal{H}$. Then $\textbf{V}$ is a pure $\Gamma_{E(3; 3; 1, 1, 1)}$-isometry if and only if there exists a separable Hilbert space $\mathcal{E}$, a unitary $U : \mathcal{H} \to H^2(\mathcal{E})$ , functions $\Phi_1, \dots, \Phi_6$ in $H^{\infty}(\mathcal B(\mathcal{E}))$ and bounded operators $A_1,A_2,\ldots,A_6$ such that
		\begin{enumerate}
			\item $V_7 = U^*M^{\mathcal{E}}_zU$ and $V_i = U^*M^{\mathcal{E}}_{\Phi_i}U,$ where $\Phi_i(z) = A_i + A^*_{7-i}z,  1\leq i \leq  6$;
			
			\item the $H^{\infty}$ norm of the operator valued functions $A_i+A_{7-i}^*z$ is at most $1$ for all $z\in \mathbb T, 1\leq i \leq 6;$ 
			
			\item $[A_i, A_j] = 0$ and $ [A_i, A^*_{7-j}] = [A_j, A^*_{7-i}]$ for $1\leq i, j \leq  6$.
		\end{enumerate}
	\end{thm}
	
	\begin{proof}
Suppose that $\textbf{V} = (V_1, \dots, V_7)$ is a pure $\Gamma_{E(3; 3; 1, 1, 1)} $-isometry. Then, by definition, there is a Hilbert space $\mathcal K$ and a $\Gamma_{E(3; 3; 1, 1, 1)} $-unitary $\textbf{N}=(N_1,N_2,\dots,N_7)$ such that $\mathcal H\subset \mathcal K$ is a common invariant subspace of $N_i$'s and $V_i=N_i|_{\mathcal H}$ for $1\leq i \leq 7$. It follows from \eqref{v11} and \eqref{v12} that $V_7$ is an isometry and $V_i=V_{7-i}^*V_7, 1\leq i\leq 6$. Since $V_7$ is a pure isometry and $\mathcal H$ is separable, so there exists a unitary operator $U :\mathcal H \to H^2(\mathcal E) $ such that $V_7=U^*M^{\mathcal{E}}_zU, $ where $\mathcal E$ has the same dimension as the multiplicity of the unilateral shift $V_7$ and  $M_z$ is the multiplication operator on $H^2(\mathcal E).$ As $V_iV_7=V_7V_i,$ there exists $\Phi_i\in H^{\infty}(\mathcal B(\mathcal{E}))$ such that $V_i = U^*M^{\mathcal{E}}_{\Phi_i}U$ for $1\leq i \leq 6.$ The relations $V_i=V_{7-i}^*V_7$ imply that
\begin{equation}\label{phii}M^{\mathcal{E}}_{\Phi_i}=(M^{\mathcal{E}}_{\Phi_{7-i}})^*M^{\mathcal{E}}_z,\;\;\;1\leq i \leq 6.\end{equation}
Let $\Phi(z)=\sum_{i=0}^{\infty}C_iz^i$ and $\Psi(z)=	\sum_{i=0}^{\infty}D_iz^i$. The relationship $M^{\mathcal{E}}_{\Phi}=(M^{\mathcal{E}}_{\Psi})^*M^{\mathcal{E}}_z$ implies that 
\begin{equation}\label{phipsi}
\sum_{i=0}^{\infty}C_iz^i=z\sum_{j=0}^{\infty}D^*_j\bar{z}^j=D_0^*z+D^*_1+\sum_{j=2}^{\infty}D^*_{j}\bar{z}^{j-1}~\text{for~all}~z\in\mathbb T.
\end{equation}
Thus, it follows from \eqref{phipsi} that $C_1=D_0^*$ and $C_0=D^*_1.$ This implies from \eqref{phii} that $\Phi_i(z)=A_i+B_iz$ for some $A_i, B_i \in \mathcal B(\mathcal E)$, where $B_i=A^*_{7-i}$  for $1 \leq i \leq 6.$ Clearly, the $H^{\infty}$ norm of the operator valued functions $A_i+A_{7-i}^*z$ is at most $1$ for all $z\in \mathbb T$ and for $1\leq i \leq 6$. Since $V_i$ 's commute with each other, we have  $M_{\Phi_i}M_{\Phi_j}=M_{\Phi_j}M_{\Phi_i}$ which implies that 
\begin{equation}\label{A_i}
(A_i+A_{7-i}^*z)(A_j+A_{7-j}^*z)=(A_j+A_{7-j}^*z)(A_i+A_{7-i}^*z)~\text{for}~1 \leq i,j \leq 6.
\end{equation}
It follows from \eqref{A_i} that 	$[A_i, A_j] = 0$ and $ [A_i, A^*_{7-j}] = [A_j, A^*_{7-i}]$ for $1\leq i, j \leq  6$.
		
Conversely, let $\textbf{V} = (V_1, \dots, V_7)$ be a commuting $7$-tuple of bounded operators on a separable Hilbert space $\mathcal{H}$ satisfying the conditions $(1)$ to $(3)$. 
Consider  the multiplication operators $M^{\mathcal{E}}_{\Phi_i}$ and $M^{\mathcal{E}}_{z}$  defined on $L^2(\mathcal{E})$, which are multiplications by $\varphi_i,  1\leq i \leq 6,$ and  $z$, respectively.
 It follows from conditions $(1)$ and $(3)$ that $M^{\mathcal{E}}_{\Phi_i}$'s commute with each other and $$M^{\mathcal{E}}_{\Phi_i}=(M^{\mathcal{E}}_{\Phi_{7-i}})^*M^{\mathcal{E}}_z ~{\rm{for}} ~1\leq i \leq 6.$$ So, by Theorem \ref{thm-5}, $(M^{\mathcal{E}}_{\Phi_1},\ldots,M^{\mathcal{E}}_{\Phi_6},M^{\mathcal{E}}_z)$ is a $\Gamma_{E(3; 3; 1, 1, 1)} $-unitary, and so is $(U^*M^{\mathcal{E}}_{\Phi_1}U,\ldots, U^*M^{\mathcal{E}}_{\Phi_6}U,U^*M^{\mathcal{E}}_zU). $ Observe that $V_i$'s  are the restrictions to the common invariant subspace $H^2(\mathcal E)$ of $M^{\mathcal{E}}_{\Phi_i}, 1\leq i \leq 6,$ and  $V_7$  is the restriction to the common invariant subspace $H^2(\mathcal E)$ of $M^{\mathcal{E}}_z$. Hence, we conclude that $(V_1, \dots, V_7)$  is a $\Gamma_{E(3; 3; 1, 1, 1)} $-isometry. As $V_7$ is an unilateral shift, it follows that $\textbf{V} = (V_1, \dots, V_7) $ is a pure $\Gamma_{E(3; 3; 1, 1, 1)} $-isometry. This completes the proof.
	\end{proof}

The following theorem gives a structure theorem for pure $\Gamma_{E(3; 2; 1, 2)}$-isometry.
	
	\begin{thm}\label{thm-15}
		Let $\textbf{W} = (W_1, W_2, W_3, \tilde{W}_1, \tilde{W}_2)$ be a $5$-tuple of commuting bounded operators on a separable Hilbert space $\mathcal{H}$. Then $\textbf{W}$ is a pure $\Gamma_{E(3; 2; 1, 2)}$ isometry if and only if there exist a separable Hilbert space $\mathcal{F}$, a unitary $\tilde{U} : \mathcal{H} \to H^2(\mathcal{F})$, functions $\tilde{\Phi}_i$ in $H^{\infty}(\mathcal B(\mathcal{F}))$ and $\tilde{\Psi}_j$ in $H^{\infty}(\mathcal B(\mathcal{F}))$ for $1 \leq i, j\leq 2$  and bounded operators $B_1, B_2,C_1,C_2$ on $\mathcal{F }$ such that
		\begin{enumerate}
			\item $W_3 = \tilde{U}^*M^{\mathcal{F}}_z\tilde{U}, W_1 = \tilde{U}^*M^{\mathcal{F}}_{\tilde{\Phi}_1}\tilde{U}, W_2 = \tilde{U}^*M^{\mathcal{F}}_{\tilde{\Psi}_1}\tilde{U}, \tilde{W}_1 = \tilde{U}^*M^{\mathcal{F}}_{\tilde{\Phi}_2}\tilde{U}$ and $ \tilde{W}_2 =\tilde{U}^*M^{\mathcal{F}}_{\tilde{\Psi}_1}\tilde{U},$ where $\tilde{\Phi}_1(z)=B_1+B_{2}^*z, \tilde{\Phi}_2(z)=B_2+B_{1}^*z,\tilde{\Psi}_1(z)=C_1+C_{2}^*z$ and  $\tilde{\Psi}_2(z)=C_2+C_{1}^*z$ for all $z\in \mathbb T;$		
			\item the $H^{\infty}$ norm of the operator functions $B_1+B_{2}^*z$ and $B_2+B_{1}^*z$ are at most one, while the norm of the operator functions $C_1+C_{2}^*z$ and $C_2+C_{1}^*z$ are at most $2$;
			
			\item
			\begin{enumerate}
				\item $[B_1, B_2] = 0, [B_1, B_1^*] = [B_2, B_2^*]$;
				
				\item $[B_1,C_1]=[C_1,B_1],[B_1,C_2^*]=[C_1,B_2^*],[C_2,B_2]=[B_2,C_2]$;
				
				\item $[B_1, C_2] = 0, [B_1,C_1^*] = [C_2,B_2^*],[C_1,B_2]=0$;
				
				\item $[C_1,C_2] = 0, [C_1, C_1^*] = [C_2,C_2^*]$.
			\end{enumerate}
		\end{enumerate}
	\end{thm}
	
	\begin{proof}
Suppose that $\textbf{W}$ is a pure $\Gamma_{E(3; 2; 1, 2)}$-isometry. By definition of pure $\Gamma_{E(3; 2; 1, 2)}$-isometry, there exists a Hilbert space $\mathcal K_1$ and a $\Gamma_{E(3; 2; 1, 2)} $-unitary $\textbf{M}=(M_1,M_2,M_3,\tilde{M}_1,\tilde{M}_2)$ such that $\mathcal H\subset \mathcal K_1$ is a common invariant subspace of $M_1,M_2,M_3,\tilde{M}_1$ and $\tilde{M}_ 2$ and $W_i={M_i}_{|{\mathcal H}}$  and $\tilde{W}_j={\tilde{M}{_j}}_{|\mathcal {H}}$ for $1\leq i \leq 3, 1\leq j \leq 2.$ Since $\textbf{M}=(M_1,M_2,M_3,\tilde{M}_1,\tilde{M}_2)$ is a $\Gamma_{E(3; 2; 1, 2)} $-unitary, so it follows that $W_3$ is an isometry, $W_1 = \tilde{W}^*_2W_3$ and $ W_2 = \tilde{W}^*_1W_3$. As $W_3$ is a pure isometry and $\mathcal H$ is separable, there exists a unitary operator $\tilde{U} :\mathcal H \to H^2(\mathcal F)$ such that $W_3=\tilde{U}^*M^{\mathcal{F}}_z\tilde{U}$, where $\mathcal F$ has the same dimension as the multiplicity of the unilateral shift $W_3$ and $M_z$ is the multiplication operator on $H^2(\mathcal F)$. Since $W_1, W_2, \tilde{W}_1$ and $ \tilde{W}_2$ commute with each other and commute with $ W_3$, so there exist $\tilde{\Phi}_i\in H^{\infty}(\mathcal B(\mathcal{F}))$ for and $\tilde{\Psi}_j\in H^{\infty}(\mathcal B(\mathcal{F}))$ for $1 \leq i, j\leq 2$ such that 
$$W_1=\tilde{U}^*M^{\mathcal{F}}_{\tilde{\Phi}_1}\tilde{U}, W_2=\tilde{U}^*M^{\mathcal{F}}_{\tilde{\Psi}_1}\tilde{U}, \tilde{W}_1=\tilde{U}^*M^{\mathcal{F}}_{\tilde{\Psi}_2}\tilde{U}~{\rm{and}}~\tilde{W}_2=\tilde{U}^*M^{\mathcal{F}}_{\tilde{\Phi}_2}\tilde{U}.$$ The relations $W_1 = \tilde{W}^*_2W_3$ and $ W_2 = \tilde{W}^*_1W_3$ imply that
\begin{equation}\label{phii11}M^{\mathcal{F}}_{\tilde{\Phi}_1}=(M^{\mathcal{F}}_{\tilde{\Phi}_{2}})^*M^{\mathcal{E}}_z~{\rm{and}}~ M^{\mathcal{F}}_{\tilde{\Psi}_1}=(M^{\mathcal{E}}_{\tilde{\Psi}_{2}})^*M^{\mathcal{E}}_z.\end{equation} By  the same argument as  in Theorem \ref{thm-14}, we establish that 
$$\tilde{\Phi}_1(z)=B_1+B_{2}^*z, \tilde{\Phi}_2(z)=B_2+B_{1}^*z,\tilde{\Psi}_1(z)=C_1+C_{2}^*z~\text{ and }~ \tilde{\Psi}_2(z)=C_2+C_{1}^*z$$ for all $z\in \mathbb T.$		
 The $H^{\infty}$ norm of the operator functions $B_1+B_{2}^*z$ and $B_2+B_{1}^*z$ are at most one, while the norm of the operator functions $C_1+C_{2}^*z$ and $C_2+C_{1}^*z$ are at most $2$ for all $z\in \mathbb T.$ As $W_1, W_2, W_3, \tilde{W}_1$ and $ \tilde{W}_2$ commute with each other, we have the following: 
\begin{equation}
\begin{aligned}\label{mphi11}
M_{\tilde{\Phi}_1}M_{\tilde{\Phi}_2}=M_{\tilde{\Phi}_2}M_{\tilde{\Phi}_1},M_{\tilde{\Phi}_1}M_{\tilde{\Psi}_1}=M_{\tilde{\Psi}_1}M_{\tilde{\Phi}_1},M_{\tilde{\Phi}_1}M_{\tilde{\Psi}_2}=M_{\tilde{\Psi}_2}M_{\tilde{\Phi}_1}~\text{and}~M_{\tilde{\Psi}_1}M_{\tilde{\Psi}_2}=M_{\tilde{\Psi}_2}M_{\tilde{\Psi}_1}.
\end{aligned}
\end{equation}
It follows from \eqref{mphi11} that 
\begin{equation}
\begin{aligned}\label{mphi112}
(B_1+B_{2}^*z)(B_2+B_{1}^*z)=(B_2+B_{1}^*z)(B_1+B_{2}^*z),\\
(B_1+B_{2}^*z)(C_1+C_{2}^*z)=(C_1+C_{2}^*z)(B_1+B_{2}^*z),\\
(B_1+B_{2}^*z)(C_2+C_{1}^*z)=(C_2+C_{1}^*z)(B_1+B_{2}^*z),\\
(C_1+C_{2}^*z)(C_2+C_{1}^*z)=(C_2+C_{1}^*z)(C_1+C_{2}^*z),
\end{aligned}
\end{equation}
for all $z\in \mathbb T.$ Part $(3)$ is derived from \eqref{mphi112}.

Conversely, suppose that $\textbf{W} = (W_1, W_2, W_3, \tilde{W}_1, \tilde{W}_2)$ is a commuting $5$-tuple of bounded operators on a separable Hilbert space $\mathcal{H}$ satisfying the conditions $(1)$ to $(3)$. Consider  the multiplication operators $M^{\mathcal{F}}_{\tilde{\Phi}_i}$, $M^{\mathcal{F}}_{\tilde{\Psi}_j}$  and $M^{\mathcal{F}}_{z}$ defined on $L^2(\mathcal{E})$, which are multiplications by $\tilde{\Phi}_i$, $\tilde{\Psi}_j$ for $1\leq i,j \leq 2$ and $z$, respectively. The conditions $(1)$ and $(3)$ imply that 
\begin{equation}
\begin{aligned}\label{mphi1133}
M_{\tilde{\Phi}_1}M_{\tilde{\Phi}_2}=M_{\tilde{\Phi}_2}M_{\tilde{\Phi}_1},M_{\tilde{\Phi}_1}M_{\tilde{\Psi}_1}=M_{\tilde{\Psi}_1}M_{\tilde{\Phi}_1},M_{\tilde{\Phi}_1}M_{\tilde{\Psi}_2}=M_{\tilde{\Psi}_2}M_{\tilde{\Phi}_1}~\text{and}~M_{\tilde{\Psi}_1}M_{\tilde{\Psi}_2}=M_{\tilde{\Psi}_2}M_{\tilde{\Psi}_1}
\end{aligned}
\end{equation}
and 
\begin{equation}\label{phii1144}M^{\mathcal{F}}_{\tilde{\Phi}_1}=(M^{\mathcal{F}}_{\tilde{\Phi}_{2}})^*M^{\mathcal{E}}_z~{\rm{and}}~ M^{\mathcal{F}}_{\tilde{\Psi}_1}=(M^{\mathcal{E}}_{\tilde{\Psi}_{2}})^*M^{\mathcal{E}}_z.\end{equation} 
Thus, it follows from \eqref{mphi1133}, \eqref{phii1144} and Theorem \ref{thm-7} that  $(M^{\mathcal{F}}_{\tilde{\Phi}_1},M^{\mathcal{F}}_{\tilde{\Psi}_1},M^{\mathcal{E}}_z,M^{\mathcal{E}}_{\tilde{\Psi}_{2}},M^{\mathcal{F}}_{\tilde{\Phi}_{2}})$ is a $\Gamma_{E(3; 2; 1, 2)} $-unitary and so is \small{$(\tilde{U}^*M^{\mathcal{F}}_{\tilde{\Phi}_1}\tilde{U},\tilde{U}^*M^{\mathcal{F}}_{\tilde{\Psi}_1}\tilde{U},\tilde{U}^*M^{\mathcal{E}}_z\tilde{U},\tilde{U}^*M^{\mathcal{E}}_{\tilde{\Psi}_{2}}\tilde{U}, \tilde{U}^*M^{\mathcal{F}}_{\tilde{\Phi}_{2}}\tilde{U}).$} Note that $W_1, W_2$  are the restrictions to the common invariant subspace $H^2(\mathcal E)$ of $M^{\mathcal{F}}_{\tilde{\Phi}_1},M^{\mathcal{F}}_{\tilde{\Psi}_1}$; $\tilde{W}_1,\tilde{W}_2$ are the restrictions to the common invariant subspace $H^2(\mathcal E)$ of $M^{\mathcal{E}}_{\tilde{\Psi}_{2}},M^{\mathcal{F}}_{\tilde{\Phi}_{2}}$ and  $W_5$  is the restriction to the common invariant subspace $H^2(\mathcal E)$ of $M^{\mathcal{E}}_z$. Therefore, we conclude that $\textbf{W} = (W_1, W_2, W_3, \tilde{W}_1, \tilde{W}_2)$ is a  $\Gamma_{E(3; 2; 1, 2)} $-isometry. Since $W_3$ is an unilateral shift, it yields that $\textbf{W} = (W_1, W_2, W_3, \tilde{W}_1, \tilde{W}_2)$ is a pure $\Gamma_{E(3; 2; 1, 2)} $-isometry. This completes the proof. 
\end{proof}
%
\begin{thm}\label{thm-201}
		Let $\textbf{V} = (V_1, \dots, V_7)$ be a $7$-tuple of commuting bounded operators on a Hilbert space $\mathcal{H}$. Then the following are equivalent:
		\begin{enumerate}
			\item $\textbf{V}$ is $\Gamma_{E(3; 3; 1, 1, 1)}$-unitary.
			
			\item $\{(V_1, V_3 + \eta V_5, \eta V_7, V_2 + \eta V_4, \eta V_6) : \eta \in \mathbb{T}\}$ is a family of $\Gamma_{E(3; 2; 1, 2)}$-isometry and $\|V_i\|\leq 1$ for $1 \leq i \leq 6.$
			
			\item $\{(V_1, V_5 + \eta V_3, \eta V_7, V_4 + \eta V_2, \eta V_6) : \eta \in \mathbb{T}\}$ is a family of $\Gamma_{E(3; 2; 1, 2)}$-isometry and $\|V_i\|\leq 1$ for $1 \leq i \leq 6.$
		
			\item $ \{(V_2, V_3 + \eta V_6, \eta V_7, V_1 + \eta V_4, \eta V_5) : \eta \in \mathbb{T}\}$ is a family of $\Gamma_{E(3; 2; 1, 2)}$-isometry and $\|V_i\|\leq 1$ for $1 \leq i \leq 6.$
			
			\item $ \{ (V_2, V_6 + \eta V_3, \eta V_7, V_4 + \eta V_1, \eta V_5) : \eta \in \mathbb{T}\}$ is a family of $\Gamma_{E(3; 2; 1, 2)}$-isometry and $\|V_i\|\leq 1$ for $1 \leq i \leq 6.$
			
			\item $ \{V_4, V_5 + \eta V_6, \eta V_7, V_1 + \eta V_2, \eta V_3) : \eta \in \mathbb{T}\}$ is a family of $\Gamma_{E(3; 2; 1, 2)}$-isometry and $\|V_i\|\leq 1$ for $1 \leq i \leq 6.$
			
			\item $ \{(V_4, V_6 + \eta V_5, \eta V_7, V_2 + \eta V_1, \eta V_3): \eta \in \mathbb{T}\}$ is a family of $\Gamma_{E(3; 2; 1, 2)}$-isometry and $\|V_i\|\leq 1$ for $1 \leq i \leq 6.$
		\end{enumerate}
	\end{thm}
	\begin{proof}
We demonstrate $(1) \Leftrightarrow(2)$. The equivalences $(1) \Leftrightarrow (3)$, $(1) \Leftrightarrow (4)$, $(1) \Leftrightarrow (5)$, $(1) \Leftrightarrow (6)$, and $(1) \Leftrightarrow (7)$ are established in a similar manner.
		
$(1) \Rightarrow (2):$ Suppose that $\textbf{V}$ be a $\Gamma_{E(3; 3; 1, 1, 1)}$-isometry. Then, by definition, there is a Hilbert space $\mathcal K$ and a $\Gamma_{E(3; 3; 1, 1, 1)} $-unitary $\textbf{N}=(N_1,N_2,\dots,N_7)$ such that $\mathcal H\subset \mathcal K$ is a common invariant subspace of $N_i$'s and $V_i=N_i|_{\mathcal H}$ for $1\leq i \leq 7$. It follows from Theorem \ref{thm-9} that 
$$\{(N_1, N_3 + \eta N_5, \eta N_7, N_2 + \eta N_4, \eta N_6) : \eta \in \mathbb{T}\}$$ is a family of $\Gamma_{E(3; 2; 1, 2)}$-unitary. Hence we conclude that $$\{(V_1, V_3 + \eta V_5, \eta V_7, V_2 + \eta V_4, \eta V_6) : \eta \in \mathbb{T}\}$$ is a family of $\Gamma_{E(3; 2; 1, 2)}$-isometry and $\|V_i\|\leq 1$ for $1 \leq i \leq 6$. 
		
$(2) \Rightarrow (1):$ Let $\{(V_1, V_3 + \eta V_5, \eta V_7, V_2 + \eta V_4, \eta V_6) : \eta \in \mathbb{T}\}$ be a family of $\Gamma_{E(3; 2; 1, 2)}$-isometry and $\|V_i\|\leq 1$ for $1 \leq i \leq 6.$ It yields from Theorm \ref{thm-12} that 
\begin{equation} \label{Muni1} V_1=\bar{\eta}V_6^*\eta V_7, V_3+\eta V_5=(V_2^*+\bar{\eta}V_4^*)\eta V_7\end{equation} and $\eta V_7$ is an isometry for all $\eta \in \mathbb T.$ In particular, for $\eta =1,$ we have $V_7$ is a unitary. From \eqref{Muni1}, we have $V_1=V_6^*V_7, V_2=V_5^*V_7$ and $V_3=V_4^*V_7.$ Hence, by Theorem \ref{thm-12}, we conclude that  $\textbf{V}$ is $\Gamma_{E(3; 3; 1, 1, 1)}$-isometry. This completes the proof.
	\end{proof}	
	
\begin{thm}\label{thm-102}
		Let $(W_1, W_2, W_3, \tilde{W}_1, \tilde{W}_2)$ be a commuting $5$-tuple of bounded operators on Hilbert space $\mathcal{H}$. Then $(W_1, W_2, W_3, \tilde{W}_1, \tilde{W}_2)$ is a $\Gamma_{E(3; 2; 1, 2)}$-isometry if and only if $\Big(W_1, \frac{\tilde{W}_1}{2}, \frac{W_2}{2}, \frac{\tilde{W}_1}{2}, \frac{W_2}{2}, \tilde{W}_2, W_3\Big)$ is a $\Gamma_{E(3; 3; 1, 1, 1)}$-isometry.	\end{thm}	
\begin{proof}
Suppose that $(W_1, W_2, W_3, \tilde{W}_1, \tilde{W}_2)$ is a $\Gamma_{E(3; 2; 1, 2)}$-isometry. By definition of  $\Gamma_{E(3; 2; 1, 2)}$-isometry, there exists a Hilbert space $\mathcal K_1$ and a  $\Gamma_{E(3; 2; 1, 2)} $-unitary $\textbf{M}=(M_1,M_2,M_3,\tilde{M}_1,\tilde{M}_2)$ such that $\mathcal H\subset \mathcal K_1$ is a common invariant subspace of $M_1,M_2,M_3,\tilde{M}_1$ and $\tilde{M}_ 2$ and $W_i=M_i|_{\mathcal H}$  and $\tilde{W}_j=\tilde{M}_j|_{\mathcal H}$ for $1\leq i \leq 3, 1\leq j \leq 2.$ It yields from Theorem \ref{thm-9} that $\Big(M_1, \frac{\tilde{M}_1}{2}, \frac{M_2}{2}, \frac{\tilde{M}_1}{2}, \frac{M_2}{2}, \tilde{M}_2, M_3\Big)$ is a $\Gamma_{E(3; 3; 1, 1, 1)}$-unitary and hence  $\Big(W_1, \frac{\tilde{W}_1}{2}, \frac{W_2}{2}, \frac{\tilde{W}_1}{2}, \frac{W_2}{2}, \tilde{W}_2, W_3\Big)$ is a $\Gamma_{E(3; 3; 1, 1, 1)}$-isometry. 

Conversely, let  $\Big(W_1, \frac{\tilde{W}_1}{2}, \frac{W_2}{2}, \frac{\tilde{W}_1}{2}, \frac{W_2}{2}, \tilde{W}_2, W_3\Big)$ be a $\Gamma_{E(3; 3; 1, 1, 1)}$-isometry. Then, again by definition, there is a Hilbert space $\mathcal K$ and a $\Gamma_{E(3; 3; 1, 1, 1)} $-unitary   $\Big(M_1, \frac{\tilde{M}_1}{2}, \frac{M_2}{2}, \frac{\tilde{M}_1}{2}, \frac{M_2}{2}, \tilde{M}_2, M_3\Big)$ such that $\mathcal H\subset \mathcal K$ is a common invariant subspace of  $M_1, \frac{\tilde{M}_1}{2}, \frac{M_2}{2}, \frac{\tilde{M}_1}{2}, \frac{M_2}{2}, \tilde{M}_2$ and $ M_3$ and $W_i=M_i|_{\mathcal H}$  and $\tilde{W}_j=\tilde{M}_j|_{\mathcal H}$ for $1\leq i \leq 3, 1\leq j \leq 2.$ It follows from Theorem \ref{thm-9} that  $\textbf{M}=(M_1,M_2,M_3,\tilde{M}_1,\tilde{M}_2)$ is a $\Gamma_{E(3; 2; 1, 2)} $-unitary. Therefore, we conclude that $(W_1, W_2, W_3, \tilde{W}_1, \tilde{W}_2)$ is a $\Gamma_{E(3; 3; 1, 1, 1)}$-isometry.  This completes the proof.
\end{proof}

\textsl{Acknowledgements:}
The second-named author is supported by the research project of SERB with ANRF File Number: CRG/2022/003058, by the Science and Engineering Research Board (SERB), Department of Science and Technology (DST), Government of India.

	\vspace{.5cm}
	
	\vspace{.5cm}
	\end{document}